\documentclass[12pt, twoside]{article}
\usepackage{amsmath,amsthm,amssymb}
\usepackage{times}
\usepackage{enumerate}
\usepackage{graphicx}
\usepackage{epstopdf}
\usepackage{etex}
\usepackage{esint}

\theoremstyle{definition}
\newtheorem{thm}{Theorem}[section]

\newtheorem{lem}[thm]{Lemma}
\newtheorem{defin}[thm]{Definition}

\newtheorem{prop}[thm]{Proposition}

\newtheorem*{xrem}{Remark}

\numberwithin{equation}{section}

\newcommand{\subjclass}[1]{\bigskip\noindent\emph{2010 Mathematics Subject Classification:}\enspace#1}
\newcommand{\keywords}[1]{\noindent\emph{Keywords:}\enspace#1}

\begin{document}


\baselineskip=17pt


\title{Existence and uniqueness of axially symmetric compressible subsonic jet impinging on an infinite wall}

\author{Jianfeng Cheng\\
Department of Mathematics, Sichuan University, \\
              Chengdu 610065, P. R. China\\
jianfengcheng@126.com\\
\\
Lili Du\\
Department of Mathematics, Sichuan University,\\
 Chengdu 610065, P. R. China\\
dulili@scu.edu.cn.\ \ \ corresponding author.\\
\\
Qin Zhang\\
Department of Mathematics, Chongqing Jiaotong University,\\
Chongqing 400074, P. R. China. \\
            zqcfl2626@gmail.com}

\date{}

\maketitle


\begin{abstract}
This paper is concerned with the well-posedness theory of the impact
of a subsonic axially symmetric jet emerging from a semi-infinitely long nozzle, onto a
rigid wall. The fluid motion is described by the steady isentropic
Euler system. We showed that there exists a critical value $M_{cr}>0$,
if the given mass flux is less than $M_{cr}$, there exists a unique smooth subsonic axially symmetric jet
issuing from the given semi-infinitely long nozzle and hitting a given
uneven wall. The surface of the axially symmetric impinging jet is a free
boundary, which detaches from the edge of the nozzle smoothly. It is showed that a unique suitable
choice of the pressure difference between the chamber and the atmosphere guarantees the continuous fit condition of
the free boundary.
Moreover, the asymptotic behaviors and the decay properties of the
impinging jet and the free surface in downstream were also obtained.
The main results in this paper solved the open problem on the
well-posedness of the compressible axially symmetric impinging jet, which
has proposed by A. Friedman in Chapter 16 in \cite{FA2}. The key
ingredient of our proof is based on the variational method to the
quasilinear elliptic equation with the Bernoulli's type free boundaries.


\end{abstract}

\def\aint{\dashint}

\def\arraystretch{2}
\def\eps{\varepsilon}

\def\s#1{\mathbb{#1}} 
\def\t#1{\tilde{#1}} 
\def\b#1{\overline{#1}}
\def\N{\mathcal{N}} 
\def\M{\mathcal{M}} 
\def\R{{\mathbb{R}}}
\def\B{{\mathcal{B}}}
\def\BB{\mathfrak{B}}
\def\F{{\mathcal{F}}}
\def\G{{\Gamma}}
\def\ba{\begin{array}}
\def\ea{\end{array}}
\def\be{\begin{equation}}
\def\ee{\end{equation}}

\def\bes{\begin{mysubequations}}
\def\ees{\end{mysubequations}}

\def\cz#1{\|#1\|_{C^{0,\alpha}}}
\def\ca#1{\|#1\|_{C^{1,\alpha}}}
\def\cb#1{\|#1\|_{C^{2,\alpha}}}
\def\psir{\left|\frac{\nabla\psi}{r}\right|^2}
\def\lb#1{\|#1\|_{L^2}}
\def\ha#1{\|#1\|_{H^1}}
\def\hb#1{\|#1\|_{H^2}}
\def\th{\theta}
\def\Th{\Theta}
\def\cin{\subset\subset}
\def\Ld{\Lambda}
\def\ld{\lambda}
\def\l{\lambda}
\def\ol{{\Omega_L}}
\def\sla{{S_L^-}}
\def\slb{{S_L^+}}
\def\e{\varepsilon}
\def\C{\mathbf{C}} 
\def\ra{\rightarrow}
\def\xra{\xrightarrow}
\def\g{\nabla}
\def\a{\alpha}
\def\b{\beta}
\def\d{\delta}
\def\th{\theta}
\def\fai{\varphi}
\def\O{\Omega}
\def\ol{{\Omega_L}}
\def\psirk{\left|\frac{\nabla\psi}{r+k}\right|^2}
\def\tO{\tilde{\Omega}}
\def\tu{\tilde{u}}
\def\tv{\tilde{v}}
\def\trho{\tilde{\rho}}
\def\W{\mathcal{W}}
\def\f{\frac}
\def\p{\partial}
\def\o{\omega}
\def\B{\mathcal{B}}
\def\H{\Theta}
\def\msS{\mathscr{S}}
\def\bq{\mathbf{q}}
\def\msE{\mathcal{E}}
\def\mfa{\mathfrak{a}}
\def\mfb{\mathfrak{b}}
\def\mfc{\mathfrak{c}}
\def\mfd{\mathfrak{d}}
\def\ff{\texthtbardotlessj}
\def\mcL{\mathcal{L}}
\def\mcR{\mathcal{R}}
\def\Div{\text{div}}

\section{Introduction}
\label{intro}
The problem of a compressible jet falling from a channel and
impacting on a wall is a fascinating one, with very practical
applications. The canonical problem is of interest in a number of
areas, such a flow is produced by the downwards-directed jet from a
vertical take-off aircraft spreading out over the ground, or by a
jet of water form a tap falling into a full sink. The monographs of
Birkhoff and Zarantonello in \cite{Birk}, Jacob in \cite{Ja},
Gurevich in \cite{Gu} and Milne-Thomson in \cite{MT} gave good
surveys of these flows.

The two-dimensional impinging jets have been considered by Helmholtz
and Kirchhoff in 1968. They constructed a solution to the steady
irrotational flows of ideal incompressible weightless fluid, bounded
by the walls and free streamlines. A first systematic well-posedness
result on the incompressible impinging jet was mentioned in Page
364 and Page 416 in \cite{FA1} that A. Friedman and L. Caffarelli
established the existence of the incompressible irrotational jet
issuing from a two-dimensional semi-infinite channel and impinging
on an infinite plate (see also in Chapter 16.3 in \cite{FA2}), see \cite{ACF1,ACF2,ACF3,CDW2,CDW3} in different settings.
Furthermore, A. Friedman investigated the compressible subsonic free surface flow theory
on the sharped charge jets in \cite{FA2} and proposed several open problems on the impinging jets
in two dimensions. The one is
$$\text{"\it{Problem (4)}. Do the same for the compressible case."}$$ In the recent work \cite{CDW}, the authors established
the existence result on the subsonic jets issuing from a convergent
nozzle and impact on a flat plate for some special atmospheric
pressure (Please see Figure \ref{fi0}).

\begin{center}\begin{figure*}
  \includegraphics[width=0.75\textwidth]{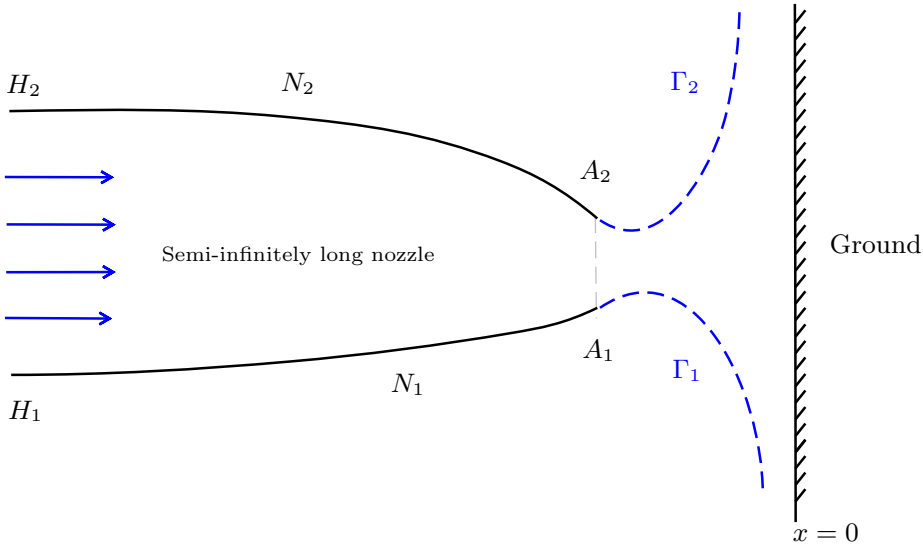}
\caption{Subsonic impinging jet in two dimensions}
\label{fi0}       
\end{figure*}\end{center}

As A. Friedman pointed out in \cite{FA2},  "$\cdots$ the
compressible axially symmetric case is quite open $\cdots$". In this
paper, we will focus on another open problem pointed out by A.
Friedman in \cite{FA2}:
$$\text{"\it{Problem (5)}. Extend the results to the axially
symmetric flows."}$$ We will establish the existence and uniqueness
of the compressible impinging jet in axially symmetric case, and
solve the open problem (5) pointed out by A. Friedman. Many numerical
simulations of the impact of a compressible flows from a cylinder on
a rigid wall are referred to \cite{GE,GY,KO,SW,Ve1}.
\begin{figure*}
  \includegraphics[width=0.75\textwidth]{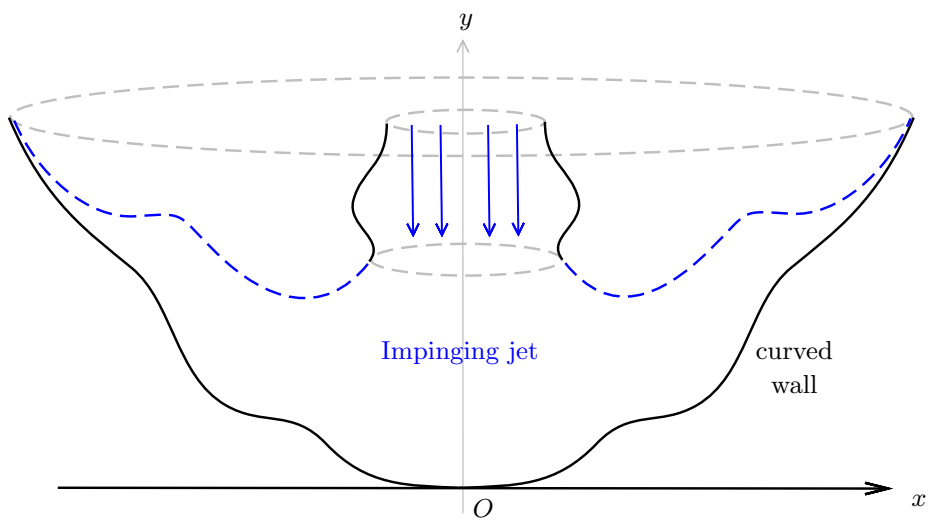}
\caption{Subsonic axially symmetric impinging jet}
\label{fi2}       
\end{figure*}

The present paper treats the compressible impinging jet problem
created by the impingement of a subsonic axially symmetric jet
emerging from a semi-infinitely long nozzle on a solid curved wall
(see Figure \ref{fi2}). The geometry considered here is a
semi-infinite nozzle in the form of a circular cylinder, in an
unbounded space. The infinite uneven wall is solid and undeformable.
The fluid is assumed to be steady, inviscid and irrotational
throughout, and the jet emerges from the orifice of the nozzle of
circular cross-section bounded by a stream surface, the nozzle wall
and the curved wall.

\subsection{Formulation of the physical problem}

The steady isentropic compressible flow is governed by the following
three-dimensional Euler system
\be\label{a01}\left\{\ba{rl} &\nabla\cdot(\rho U)=0,\\
&(\rho U\cdot\nabla)U+\nabla p=0,\ea\right.\ee with the irrotational
condition \be\label{a03}\nabla\times U=0.\ee Here, $U=(u_1,u_2,u_3)$
is the velocity, $\rho=\rho(x_1,x_2,x_3)$ is the density and
$p=p(\rho)$ denotes the pressure, $(x_1,x_2,x_3)\in\mathbb{R}^3$ is
the space variable. Without loss of generality, we assume that the
flow is perfect polytropic gas satisfying the $\gamma$-law
\be\label{a02}p=\mathcal{A}\rho^\gamma,\ee with $\mathcal{A}>0$ and
the adiabatic exponent $\gamma>1$. The sound speed of the flow is
defined as
$c(\rho)=\sqrt{p'(\rho)}=\sqrt{\mathcal{A}\gamma\rho^{\gamma-1}}$,
and the flow is subsonic if and only if $|U|<c(\rho)$.

Here, we consider the axially symmetric flow in this paper, and take
$y=x_3$ to be the axis of symmetry and $x=\sqrt{x_1^2+x^2_2}$.
 Let the fluid density and velocity be $\rho(x,y)$ and $u(x,y),v(x,y),w(x,y)$ in cylindrical coordinates, where $u,v,w$ are radial velocity, axially velocity and swirl velocity respectively. Furthermore, we look for such an axisymmetric flow without swirl in this paper, one has
$$u_1(x_1,x_2,x_3)=u(x,y)\f{x_1}{x},\ u_2(x_1,x_2,x_3)=u(x,y)\f{x_2}{x},\ u_3(x_1,x_2,x_3)=v(x,y).$$ Then the governing
system \eqref{a01} and \eqref{a03} are written in the cylindrical
coordinates as
\be\label{a04}\left\{\ba{rl}&(x\rho u)_x+(x\rho v)_y=0,\\
&(x\rho u^2)_x+(x\rho uv)_y+xp_x=0,\\
&(x\rho uv)_x+(x\rho v^2)_y+xp_y=0,\ea\right.\ee with the irrotational
condition \be\label{a05}u_y-v_x=0.\ee
\begin{figure*}
  \includegraphics[width=0.75\textwidth]{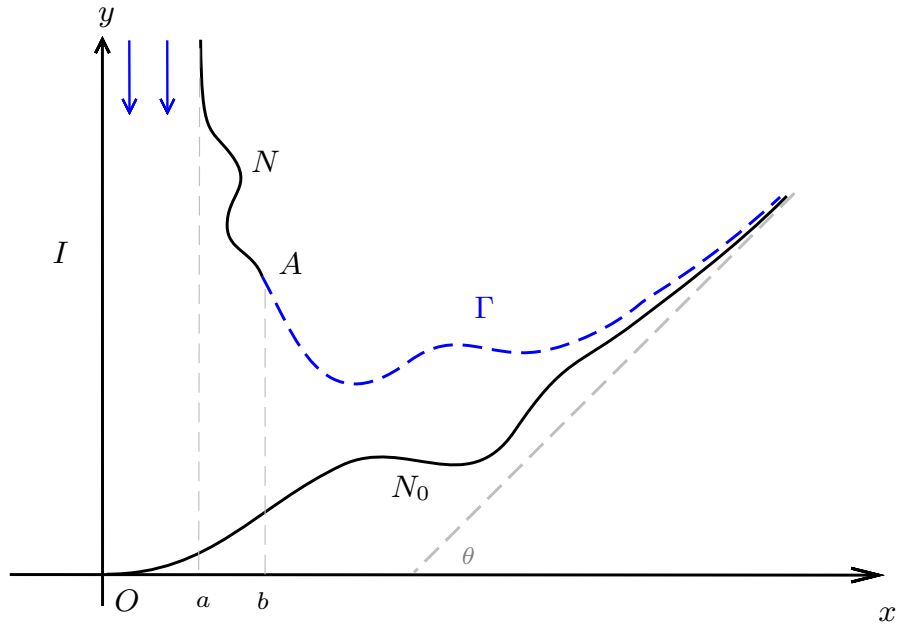}
\caption{Axially symmetric impinging jet}
\label{fi3}       
\end{figure*}

In order to clarify the physical problem, we start with the notation
and the assumptions on the geometry of the nozzle and the
impermeable wall as follows. As shown in Figure \ref{fi3}, we denote
the semi-infinite nozzle as \be\label{a06}N: y=g(x)\in
C^{2,\a}((a,b]),~ ~ g(b)=1~~\text{and}~~ \lim_{x\rightarrow
a^+}g(x)=+\infty, \ee~with~$A=(b,1)$ being the endpoint of the
nozzle, and $g(x)$ is decreasing in $(a,a+\e_0)$ for any small
$\e_0>0$. Denote the uneven wall as $N_0: ~y=g_0(x)$ for~~$x\geq0$
satisfying \be\label{a07} g_0(x)\in
C^{2,\a}([0,+\infty)),~~g_0(0)=0,~~ g(x)>g_0(x)\ \ \text{ for any
$x\in(a,b]$ },\ee and there exist a
$\theta\in\left[0,\f{\pi}{2}\right)$, and a $R_0>b$, such that
\be\label{a09}~~g_0'(x)\rightarrow \tan\theta\ \ \text{ as
$x\rightarrow+\infty$, and}\ \ g_0''(x)\geq 0\ \ \text{for any
$x>R_0$}.\ee

The boundary conditions require that the nozzle wall $N$ and the
uneven wall $N_0$ are assumed to be impermeable, thus
\be\label{a10}(u,v)\cdot\vec{n}=0~~~\text{on}~~N\cup N_0,\ee
 where $\vec{n}$ is the unit outward normal to $N\cup N_0$. We denote the incoming mass flux as $M_0$ in the axially symmetric nozzle, namely
\be\label{a11}-\int_{0}^{x_0} 2\pi x\rho(x,y_0) v(x,y_0)dx=M_0>0,\ee
for any $y_0\in(1,+\infty)$, where $x_0=\inf\{x\mid g(x)=y_0\}$.

The well-known Bernoulli's law gives that
\be\label{a12}\f{q^2}{2}+\f{\mathcal{A}\gamma}{\gamma-1}\rho^{\gamma-1}=\mathcal{B}
\ \text{ in the fluid field,}\ee where $q=\sqrt{u^2+v^2}$ is the
flow speed and $\mathcal{B}$ is the Bernoulli's constant.



The free surface $\G$ is defined as an interface between the fluid
issuing from the nozzle wall $N$ and the fluid outside. And then
the fluid still satisfies the slip boundary condition on the free
surface $\G$. Moreover, the pressure on $\G$ balances to the
atmospheric pressure $p_{atm}$, and thus we assume that
\be\label{a15}p=p_{atm}~~\text{on}~~\G.\ee



Hence, we can formulate the compressible subsonic impinging jet problem into
the following free boundary problem (FBP).

\begin{defin}\label{def1}
{\bf The free boundary problem (FBP).}\\
 Given a semi-infinitely long nozzle wall $N$, an uneven wall $N_0$, for some appropriate incoming mass flux $M_0>0$, whether there exists a unique axially symmetric subsonic impinging jet flow,
 such that the free surface $\G$ detaches smoothly from the endpoint of the nozzle wall $N$,
 and goes to infinity in $x$-direction, and the pressure balances to the atmospheric pressure $p_{atm}$ on the free surface?
\end{defin}

Next, we give the definition of the subsonic solution to the FBP.

\begin{defin}
{\bf(A subsonic solution to the FBP).}\\
A vector $(u,v,\rho,\G)$ is called a subsonic solution to the FBP, provided that \\
(1) the free surface $\G$ is given by a $C^1$-smooth function
$y=k(x)$ for $x\in(b,+\infty)$ with \be\label{a18}k(b+0)=g(b-0)=1,~
~~k^{'}(b+0)=g'(b-0),\ee and$$ k'(x)\rightarrow \tan\theta,~~~
k(x)-g_0(x)\rightarrow0~~\text{as}~~ x\rightarrow+\infty.$$
(2)~$(u,v,\rho)\in~C^{1,\a}(\O_0)\cap C(\bar{\O}_0)$ solves the
compressible Euler system \eqref{a04} in $\O_0$, where
$\O_0$ is the flow field bounded by $N$, $N_0$, $I$ and $\G$;\\
(3)~$\sup_{(x,y)\in\bar{\O}_0}\f{\sqrt{u^2+v^2}}{c(\rho)}<1$ and $p=p_{atm}$ on $\G$.\\

\end{defin}

\begin{xrem}The conditions \eqref{a18} are so-called {\it continuous fit condition} and {\it smooth fit
condition} to the impinging jet, which imply that the free surface $\G$ initiates smoothly from the endpoint $A$ of the nozzle wall $N$.
\end{xrem}

\subsection{Main results} Before we state the main results in this paper, we would like to emphasize that the atmospheric pressure $p_{atm}$ is an arbitrary constant here. Once it is fixed, we found that there exists an interval $(p_1,p_2)$ for the constant pressure $p_{in}$ in the inlet, and then our results reveal that we can impose a unique $p_{in}\in(p_1,p_2)$ to guarantee the unique existence of the axially symmetric impinging jet. And the critical values $p_1$, $p_2$ depend on the atmospheric pressure $p_{atm}$ and the mass flux $M_0$, which can be determined uniquely by the following formulas,
$$\f{M^2_0}{2\pi^2a^4\left(\f{p_1}{\mathcal{A}}\right)^{\f2\gamma}}+\f{\mathcal{A}\gamma}{\gamma-1}\left(\f{p_1}{\mathcal{A}}\right)^{\f{\gamma-1}\gamma}
=\f{\mathcal{A}\gamma}{\gamma-1}\left(\f{p_{atm}}{\mathcal{A}}\right)^{\f{\gamma-1}{\gamma}}$$ and
$$\f{M^2_0}{2\pi^2a^4\left(\f{p_2}{\mathcal{A}}\right)^{\f2\gamma}}+\f{\mathcal{A}\gamma}{\gamma-1}\left(\f{p_2}{\mathcal{A}}\right)^{\f{\gamma-1}\gamma}
=\f{\mathcal{A}\gamma(\gamma+1)}{2(\gamma-1)}\left(\f{p_{atm}}{\mathcal{A}}\right)^{\f{\gamma-1}{\gamma}}.$$ Obviously, $p_2>p_1$ and the interval $(p_1,p_2)$ is well-defined.

The main results in this paper are stated as follows.
\begin{thm}\label{the1}
Assume that the semi-infinitely long nozzle wall $N$ and the
uneven wall $N_0$ satisfy the conditions \eqref{a06}-\eqref{a09},
for any given atmospheric pressure $p_{atm}>0$, then there exists a critical mass flux $M_{cr}$, such that for any $M_0\in(0,M_{cr})$, there exist a unique incoming pressure $p_{in}\in(p_1,p_2)$ and a
 unique subsonic solution $(u,v,\rho,\G)$ to the
FBP. Moreover, \\
(1)~the subsonic impinging jet flow satisfies the asymptotic
behavior in upstream as follows,
$$(u,v,\rho)(x,y)\rightarrow(0,v_{in},\rho_{in})~~~\text{and}~~\nabla(u,v,\rho)\rightarrow0,$$
uniformly in any compact subset of $(0,a)$ as $y\rightarrow+\infty$,
where $v_{in}=-\f{M_0}{\pi a^2\rho_{in}}$ and $\rho_{in}=\left(\f{p_{in}}{\mathcal{A}}\right)^{\f1\gamma}$.

Similarly, the subsonic impinging jet flow satisfies the asymptotic
behavior in downstream as follows,
$$(u,v,\rho)(x,y)\rightarrow\left(q_0\cos\th,q_0\sin\th,\rho_0\right),
 \ \ (x,y)\in\O_0,$$ as $x^2+y^2\rightarrow+\infty$,
where $\rho_0=\left(\f{p_{atm}}{\mathcal{A}}\right)^{\f1\gamma}$ and $q_0=\sqrt{v^2_{in}+\f{2\mathcal{A}\gamma}{\gamma-1}\left(\rho^{\gamma-1}_{in}-\rho_0^{\gamma-1}\right)}$;\\
(2)~the free boundary $\G$ converges to $N_0$ in the far field, and
\be\label{a20}x(k(x)-g_0(x))\rightarrow\f{M_0}{2\rho_0q_0\pi\cos\theta}~~~\text{as}~~x\rightarrow+\infty.\ee Furthermore, the free boundary $\Gamma$ is analytic;\\
(3)~the radial velocity $u>0$ in $\bar{\O}_0\setminus I$;\\
(4) $M_{cr}$ is
the upper critical value of mass flux for the existence of subsonic jet flow
in the following sense: either
$$\sup_{\bar\O_0}(u^2+v^2-c^2(\rho))\rightarrow 0 \ \ \text{as}\ \ M_0\rightarrow
M^-_{cr},$$ or there is no $\sigma>0$, such that for any
$M_0\in(M_{cr},M_{cr}+\sigma)$, there exists a subsonic solution to the
compressible jet flow problem and
$$\sup_{M_0\in(M_{cr},M_{cr}+\sigma)}\left(\sup_{\bar\O_0}(u^2+v^2-c^2(\rho))\right)<0.$$
\end{thm}

\begin{xrem} In view of the statement (4) in Theorem \ref{the1}, we have that either
$$\sup_{\bar\O_0}(u^2+v^2-c^2(\rho))\rightarrow 0 \ \ \text{as}\ \ M_0\rightarrow
M^-_{cr},$$ or for any $M_0>M_{cr}$, there exists an incoming mass flux $\t M_0\in[M_{cr},M_0]$, such that there are no an incoming pressure $p_{in}\in(p_1,p_2)$ and a subsonic solution to the
compressible jet flow problem satisfying
$$\sup_{\bar\O_0}(u^2+v^2-c^2(\rho))<0.$$
\end{xrem}

\begin{xrem}\label{ref1} Our result indicates that the chamber pressure $p_{in}$ in the inlet is determined uniquely by the continuous fit condition, provided that the atmospheric pressure $p_{atm}$ is imposed. On another hand, the result implies that there exists a unique pressure difference between the chamber and the outside, such that there exists a unique axially symmetric impinging jet with continuous fit condition.
\end{xrem}

\begin{xrem} The result \eqref{a20} in fact gives the convergence rate of the distance between the free surface $\Gamma$ and the curved wall $N_0$ in the downstream, which is quite different from the two-dimensional case.
\end{xrem}

\begin{xrem} Here, we restrict the magnitude of the incoming mass flux $M_0$ to guarantee the global subsonicity of the impinging jet, this idea is motivated by the recent works \cite{CDX1,CDW4,DX,DXX,DXY,XX1,XX2,XX3} on the global subsonic flow in an infinitely long
nozzle. This is also quite different from the recent results on
two-dimensional subsonic impinging jets in \cite{CDW}.
\end{xrem}

Based on the significant work \cite{ACF6} by Alt, Caffarelli and
Friedman, we can obtain the higher regularity of the free boundary
near the end point $A$.
\begin{thm}\label{the2} If $N$ is $C^{3,\alpha}$ near $A$, then the solution $(u,v,\rho,\Gamma)$ established in Theorem \ref{the1} satisfies that either \\
(1). $N\cup\Gamma$ is $C^2$ at $A$ or,\\
(2). the optimal regularity of $N\cup\Gamma$ at $A$ is only $C^{1,\f12}$ and the curvature of $\Gamma$ goes to $\pm\infty$ as
$x\rightarrow b^+$.

\end{thm}

\begin{xrem}
The results of Theorem \ref{the2} imply that either $$\text{the curvature
along $\Gamma$ tends to the curvature of $N$ at $A$,}$$ or
$$\text{the curvature of $\Gamma$ tends to $\pm\infty$ in absolute value as one approaches $A$ along $\Gamma$}.$$
The second case is so-called {\it abrupt separation}. The proof of Theorem \ref{the2} follows from Theorem 1.1 in \cite{ACF6} directly.
\end{xrem}

To investigate the well-posedness of the compressible subsonic
impinging jet in axially symmetric case, from the mathematical point
of view, there are at least three difficulties and key points here.
The one is how to discover a mechanism to guarantee the smoothness
and the global subsonicity of the impinging jet in the whole fluid
field. In the first well-posedness result \cite{ACF5} on
compressible subsonic jet, the authors suggested to constrain the
atmospheric pressure with subsonic condition and the convex geometry
condition of the nozzle wall. With the aid of geometry property,
they can conclude that the compressible jet achieves its maximal
speed on the free boundary, and then the subsonic condition on the
free boundary implies the global subsonicity of the jet. A similar
idea has been adapted in the recent work on subsonic impinging jet
in two dimensions in \cite{CDW}. However, in the present work, our
idea is quite different from the one in \cite{ACF5}. We do not
restrict the condition on the atmospheric pressure and the geometry
condition on the nozzle wall, and we find an upper critical value of
the incoming mass flux and show the regularity and global
subsonicity of the impinging jet provided that the incoming mass
flux is less than the upper critical value.

 The second difficulty is how to fulfill the continuous fit condition between the nozzle wall and the free boundary. In the pioneer work \cite{ACF5}, the continuous fit condition was fulfilled for special choice of the incoming mass flux. Namely, they showed that there exists a unique incoming mass flux, such that the free boundary connects smoothly at the endpoint of the nozzle wall. Here, we choose the pressure in the inlet as a parameter and show that there exists a unique pressure in the inlet lying in an appropriate interval $(p_1,p_2)$, such that the continuous fit condition holds. As mentioned in Remark \ref{ref1}, the result implies that there exists a pressure difference between the inlet and the outlet, such that the continuous fit condition is fulfilled. This makes the result more reasonable from the physical point of
 view.
 The third key point here is that for the 3D axially symmetric impinging jet there is no uniform positive distance between the free surface $\Gamma$ and the uneven wall $N_0$. Our proof firstly focuses on the decay estimates of the solution in far field. Moreover, with the optimal decay rate in hand, we get the convergence rate of the distance between the free surface and the ground. And then rescaling the impinging jet in downstream obtains many important facts, such as the asymptotic behavior of the impinging jet in downstream.



The rest of the paper is organized as follows. In Section 2, we
introduce a variational problem to solve the free boundary problem,
and moreover, establish some properties of the minimizer, such as
the bounded gradient lemma and the non-degeneracy lemma. The Section
3 is about the free boundary of the minimizer, we prove the
continuous dependence of the minimizer and the free boundary with
respect to the parameter $\ld$, and obtain the continuous fit
condition of the free boundary. In Section 4, we establish the
existence and uniqueness of the subsonic solution to the axially
symmetric impinging flow problem, provided that the incoming mass
flux is small enough. The existence of critical mass flux is
obtained in Section 5. In the final section, we give a summary of
the proof to the main results in this paper.

\section{Mathematical setting on the FBP}
\subsection{Stream function setting}

Based on the continuity equation, the stream function $\psi$
can be introduced such that \be\label{b1}u=\f{1}{x\rho}\psi_y \ \ \text{and}\ \
v=-\f{1}{x\rho}\psi_x.\ee Without loss of generality, we can impose the boundary condition as
$$\psi=m_0\ \ \text{on $N\cup \Gamma$ and}\ \psi=0\ \ \text{on $N_0\cup I$},$$ where $m_0=\f{M_0}{2\pi}$. Denote $\O$ as the possible fluid field and $E=\O\cap\{x>b\}$. Define the free boundary of the stream function as follows
$$\Gamma=E\cap\p\{\psi<m_0\}.$$

Since the pressure is equal to the constant atmospheric pressure
$p_{atm}>0$ on the free boundary $\Gamma$, it follows from
Bernoulli's law \eqref{a12} that the density $\rho$ and the momentum
$\rho q$ are also constants on $\Gamma$, denote
\be\label{b2}\rho_0=\left(\f{p_{atm}}{\mathcal{A}}\right)^{\f1\gamma},\
\
q_0=\left(2\mathcal{B}-\f{2\mathcal{A}\gamma}{\gamma-1}\rho_0^{\gamma-1}\right)^{\f12}\
\ \text{and}\ \ \ld=\rho_0 q_0\ \ \text{on}\ \ \Gamma.\ee It is easy
to see that
\be\label{b3}\left|\f{\nabla\psi}{x}\right|=\f{1}{x}\f{\p\psi}{\p\nu}=\l=\rho q=\rho_0q_0=\rho_0\left(2\mathcal{B}-\f{2\mathcal{A}\gamma}{\gamma-1}\rho_0^{\gamma-1}\right)^{\f12}~~\text{on}~~\G,\ee
where $\nu$ is the outer normal vector of $\Gamma$. Moreover, one
has
\be\label{b4}\f{2m^2_0}{a^4\rho^2_{in}}+\f{\mathcal{A}\gamma}{\gamma-1}\rho_{in}^{\gamma-1}
=\f{\ld^2}{2\rho_0^2}+\f{\mathcal{A}\gamma}{\gamma-1}\rho_0^{\gamma-1}=\mathcal{B}. \ee By virtue of \eqref{b4}, $\ld$
is uniquely determined by the density $\rho_{in}$ in upstream, once
$m_0$ and $\rho_0$ are fixed. Therefore, we can take $\ld$ as a
parameter to solve the free boundary problem firstly, and denote
\be\label{bf4}\mathcal{B}(\ld^2)=\f{\ld^2}{2\rho_0^2}+\f{\mathcal{A}\gamma}{\gamma-1}\rho_0^{\gamma-1}.\ee
As we know, there exist some critical quantities,
$$q_{\ld,cr}=\left(2\mathcal{B}(\ld^2)\f{\gamma-1}{\gamma+1}\right)^{\f12},\
 \rho_{\ld,cr}=\left(\f{2\mathcal{B}(\ld^2)}{\mathcal{A}\gamma}\f{\gamma-1}{\gamma+1}\right)^{\f1{\gamma-1}},
  \ \rho_{\ld,max}=\left(\f{\mathcal{B}(\ld^2)}{\mathcal{A}\gamma}
  (\gamma-1)\right)^{\f1{\gamma-1}},$$
  such that the flow is subsonic if and only if $q<q_{\ld,cr}$
  or $\rho_{\ld,cr}<\rho\leq\rho_{\ld,max}$ (see also in \cite{Bers2} and \cite{CFO}).

Let $t=\left|\f{\nabla\psi}{x}\right|^2$ be the square norm of the
momentum and the Bernoulli's law \eqref{a12} gives that
$$
\f{t}{2\rho^2}+\f{\mathcal{A}\gamma}{\gamma-1}\rho^{\gamma-1}=\f{\ld^2}{2\rho_0^2}+\f{\mathcal{A}\gamma}{\gamma-1}\rho_0^{\gamma-1}=\mathcal{B}(\ld^2).
$$
Moreover, set
$$\Pi_\ld=\rho_{\ld,cr}q_{\ld,cr}$$and \be\label{bf2} \mathcal{H}(t,\rho;\ld^2)=\f{t}{2\rho^2}+\f{\mathcal{A}\gamma}{\gamma-1}\rho^{\gamma-1}-\mathcal{B}(\ld^2)=0,\ee for $t>0$, $\rho_{\ld,cr}<\rho\leq\rho_{\ld,max}$ and $\ld\in (0,\Pi_\ld)$.

A simple
manipulation leads to that \be\label{bf5}\f{\p\mathcal{H}}{\p (\ld^2)}(t,\rho;\ld^2)=-\f{1}{2\rho_0^2}\ \ \ \text{and}\ \ \ \f{\p\mathcal{H}}{\p t}(t,\rho;\ld^2)=\f{1}{2\rho^2},\ee and
\be\label{bf3}\ba{rl}\f{\p\mathcal{H}}{\p \rho}(t,\rho;\ld^2)&=\f{1}{\rho}\left(\mathcal{A}\gamma\rho^{\gamma-1}-\f{t}{\rho^2}\right)\\
&=\f{1}{\rho}\left(\f{\gamma+1}{\gamma-1}\mathcal{A}\gamma\rho^{\gamma-1}-2\mathcal{B}(\ld^2)\right)\\
&=\f{1}{\rho}\f{\gamma+1}{\gamma-1}\mathcal{A}\gamma\left(\rho^{\gamma-1}-\rho^{\gamma-1}_{\ld,cr}\right)\\
&>0,\ea\ee for any $\rho_{\ld,cr}<\rho\leq\rho_{\ld,max}$ and
$\ld\in(0,\Pi_\ld)$, where we used the equality \eqref{bf2}.  Thus,
noticing \eqref{bf2}, the density $\rho$ can be described as a
function of $t$ with a parameter $\ld\in(0,\Pi_\ld)$ saying
$\rho(t;\ld^2)$, provided that
$\rho_{\ld,cr}<\rho\leq\rho_{\ld,max}$. Furthermore, it follows from
\eqref{bf5} and \eqref{bf3} that
\be\label{b5}\f{\p\rho(t;\ld^2)}{\p
t}=-\f{\f{\p\mathcal{H}(t,\rho;\ld^2)}{\p
t}}{\f{\p\mathcal{H}(t,\rho;\ld^2)}{\p
\rho}}=-\f{1}{2\rho^2\f{\p\mathcal{H}(t,\rho;\ld^2)}{\p
\rho}}<0,\ee and
\be\label{bb5}\f{\p\rho(t;\ld^2)}{\p(\ld^2)}=-\f{\f{\p\mathcal{H}(t,\rho;\ld^2)}{\p
(\ld^2)}}{\f{\p\mathcal{H}(t,\rho;\ld^2)}{\p \rho}}
=\f{1}{2\rho_0^2\f{\p\mathcal{H}(t,\rho;\ld^2)}{\p
\rho}}>0,\ee for any $\rho_{\ld,cr}<\rho\leq\rho_{\ld,max}$ and
$\ld\in(0,\Pi_\ld)$.
 In view of \eqref{b5}, it is easy to check that for any $\ld\in (0,\Pi_\ld)$,
\be\label{bf6}\text{$\rho_{\ld,cr}<\rho\leq\rho_{\ld,max}$ if and
only if $t\in(0,\Pi^2_\ld)$}.\ee Thus we can conclude that
$\rho(t;\ld^2)$ is a decreasing smooth function with
respect to $t\in[0,\Pi^2_\ld)$ for $\lambda\in(0,\Pi_\ld)$.
Moreover, the density $\rho$ for subsonic flow has the following
uniform estimates
\be\label{b6}\left(\f2{\gamma+1}\right)^{\f1{\gamma-1}}\rho_0\leq\rho_{\ld,cr}<\rho\leq\rho_{\ld,max}\leq
\left(\f{\gamma+1}2\right)^{\f1{\gamma-1}}\rho_0.\ee


After a direct computation, there exists a
$\ld_{cr}=\left(\mathcal{A}\gamma\rho_0^{\gamma+1}\right)^{\f12}=(\mathcal{A}\gamma)^{\f12}
\left(\f{p_{atm}}{\mathcal{A}}\right)^{\f{\gamma+1}{2\gamma}}$, such that
$$\ld<\Pi_\ld\ \ \text{for any $\ld<\ld_{cr}$, and}\ \ \ld_{cr}=\Pi_{\ld_{cr}}.$$
In this paper, we assume that the parameter $\ld<\ld_{cr}$. In the
following, denote
$$\rho_1(t;\ld^2)=\f{\p\rho(t;\ld^2)}{\p t}\ \ \text{and}\ \ \rho_2(t;\ld^2)=\f{\p\rho(t;\ld^2)}{\p (\ld^2)}$$
the derivatives of the function $\rho(t;\ld^2)$ for any
$t\in(0,\Pi^2_\ld)$ and $\ld\in(0,\Pi_\ld)$.

 Furthermore,
the irrotational condition \eqref{a05} deduces to the governing
equation for the stream function in the flow field that
\be\label{b8}Q_\ld\psi=\g\cdot\left(\f{\nabla\psi}{x\rho(|\f{\nabla\psi}{x}|^2;\ld^2)}\right)=0~~~~\text{in}~~\O_0,\ee
where $\g=(\p_x,\p_y)$ and $\O_0=\O\cap\{\psi<m_0\}$ is the fluid
field.

It is not difficult to see that the equation in \eqref{b8} becomes
degenerate as $\f{|\g\psi|}{x}\rightarrow\Pi_\ld$, in order to guarantee the
uniform ellipticity, at first we consider the following modified
problem. The essence of this idea has already been illustrated in
the compressible subsonic problem, seeing \cite{ACF5,Bers1,Bers2,CDX,DD,DWX,DX,DXX,DXY,XX1,XX2,XX3}.

Let $\tilde{\rho}(t;\ld^2)$ be a smooth decreasing function
satisfying
\be\label{b9}\tilde{\rho}(t;\ld^2)=\left\{\ba{rl} \rho(t;\ld^2),~~~&\text{for}~~t\leq(\Pi_\ld-2\t\e)^2,\\
\rho((\Pi_\ld-\t\e)^2;\ld^2),~~~&\text{for}~~t\geq(\Pi_\ld-\t\e)^2,\ea\right.\ee
for any small $\t\e>0$, and
\be\label{b10}\text{$-\f{\t\rho_1(t;\ld^2)}{\t\rho^2(t;\ld^2)}\leq
\f{C_{\t\e}}{1+t}$\ \ \ for $C_{\t\e}>0$.}\ee Here,
$\t\rho_1(t,\ld^2)$ denotes the derivative function of
$\t\rho(t,\ld^2)$ with respect to $t$.

 Hence, we first
consider the following modified free boundary problem
\be\label{b11}\left\{\ba{rl}\t Q_\ld\psi=\g\cdot\left(\f{\nabla\psi}{x\t\rho(|\f{\nabla\psi}{x}|^2;\ld^2)}\right)=0~~&\text{in}~~\O\cap\{\psi<m_0\},\\
\f{1}{x}\f{\p\psi}{\p\nu}=\l~~&\text{on}~~\G,\\
\psi=0~~~~\text{on}~~N_0\cup I,\ \
\psi=m_0~~~~&\text{on}~~N\cup\G.\ea\right.\ee In the end of this section, we will verify that
$\f{|\nabla\psi|}{x}\leq\Pi_\ld-2\t\e$ in $\bar\O_0$, thus the subsonic cut-off can be taken away and $\tilde{\rho}\left(\left|\f{\g\psi}{x}\right|^2;\ld^2\right)=\rho\left(\left|\f{\g\psi}{x}\right|^2;\ld^2\right)$.
\subsection{Variational approach}

To solve the free boundary value problem \eqref{b11} with any parameter
$\ld<\Pi_\ld$, we will introduce the variational method, which has been adapted to solve the compressible jet problem in
\cite{ACF5}. Next, we give the corresponding variational problem as
follows. Firstly, define an admissible set (see Figure \ref{fi3})
as
$$\ba{rl} K=\{\psi\in~H^{1}_{loc}(\mathbb{R}^2)\mid & \psi\leq m_0\ \text{a.e. in $\mathbb{R}^2$}, \ \psi=m_0 \ \text{lies above}\ N,\\
&\psi=0\ \text{lies below $N_0$ and lies left $I$}\}.\ea $$ Denote
$$F(t;\ld)=\int_0^t\f{1}{\tilde{\rho}(\tau;\ld^2)}d\tau,\
\ F_1(t;\ld)=\f{\p F(t;\ld)}{\p t}\ \
\text{and}~~F_{11}(t;\ld)=\f{\p^2 F(t;\ld)}{\p t^2},$$ which
together with \eqref{b10} yield that
\be\label{b12}F(0;\ld)=0,\ \ F(\ld^2;\ld)=\f1{\rho_0},~~0\leq
F_{11}(t;\ld)\leq\f{C_{\t\e}}{1+t}.\ee Set
$$\Ld=\Ld(\ld^2)=\sqrt{2F_1(\ld^2;\ld)\ld^2-F(\ld^2;\ld)}.$$ In view of \eqref{b6}, it
is easy to check that \be\ba{rl}\label{b7}\f{\Ld^2(\ld^2)}{\ld^2}=\f2{\rho_0}-\f1{\ld^2}\int_0^{\ld^2}\f{1}{\tilde{\rho}(\tau;\ld^2)}d\tau\leq \f{1}{\rho_0}\left(2-\left(\f{2}{\gamma+1}\right)^{\f1{\gamma-1}}\right),\ea\ee  and
\be\ba{rl}\label{b13}\f{\Ld^2(\ld^2)}{\ld^2}=\f2{\rho_0}-\f1{\ld^2}\int_0^{\ld^2}\f{1}{\tilde{\rho}(\tau;\ld^2)}d\tau\geq \f{1}{\rho_0}\left(2-\left(\f{\gamma+1}2\right)^{\f1{\gamma-1}}\right)>0,\ea\ee for any
$\ld\in[0,\Pi_\ld)$. Furthermore, it follows from \eqref{bb5} that
\be\label{bb0}\f{d\Ld^2(\ld^2)}{d(\ld^2)}=\f{1}{\rho_{0}}+\int_0^{\ld^2}\f{\rho_2(\tau;\ld^2)}{{\rho}^2(\tau;\ld^2)}d\tau> \f1{\rho_0},\ee  for any
$\ld\in[0,\Pi_\ld-2\t\e)$, which implies that $\Ld(\ld^2)$ is uniquely determined by $\ld\in(0,\Pi_\ld-2\t\e)$.

Set \be\label{bf1}\text{$f(\mathbf{\eta};\ld)=
F(|\mathbf{\eta}|^2;\ld)$ with $\mathbf{\eta}=(\eta_1,\eta_2)\in\mathbb{R}^2$,}\ee it follows from \eqref{b12}
that $f(\mathbf{\eta};\ld)$ is convex with respect to $\eta$, and there exists a constant $\vartheta$ depending on $\ld_{cr}$
and $\t\e$, such that
$$\vartheta|\mathbf{\xi}|^2\leq \sum_{i,j=1}^{2}\f{\p^2f(\mathbf{\eta};\ld)}{\p \eta_i\p \eta_j}\xi_i\xi_j\leq \vartheta^{-1}|\mathbf{\xi}|^2\ \ \text{for any
$\mathbf{\xi}=(\xi_1,\xi_2)\in\mathbb{R}^2$},$$ and
$$\vartheta|\eta|^2\leq f_\eta(\eta;\ld)\cdot \eta\ \ \text{and}\ \ \vartheta|\eta|^2\leq f(\eta;\ld)\leq\vartheta^{-1}|\eta|^2.$$

Define a function with any parameter $\ld\in(0,\Pi_\ld)$ as
follows,
\be\label{b14}G(\g\psi,\psi,x;\ld)=xF\left(\left|\f{\nabla\psi}{x}\right|^2;\ld\right)+\left(x\Ld^2-2\l F_1(\ld^2;\ld)\nabla\psi\cdot
e\right)\chi_{\{\psi<m_0\}\cap E},\ee where $\chi_D$ is
the indicator function of the set $D$ and $e=(-\sin\theta,\cos\theta)$. By virtue of the convexity of $F(t;\ld)$ with respect to $t$, one has
\be\label{b15}\ba{rl}{}&G(\g\psi,\psi,x;\ld)\\
\geq&x\left(
F\left(\left|\f{\nabla\psi}{x}\right|^2;\ld\right)-F(\ld^2;\ld)-F_1(\ld^2;\ld)\left(\left|\f{\nabla\psi}{x}\right|^2-\ld^2\right)\right)\chi_{\{\psi<m_0\}\cap E}\\
&+xF_1(\ld^2;\ld)\left|\f{\nabla\psi}{x}-\ld
e\right|^2\chi_{\{\psi<m_0\}\cap E}\\
\geq& xF_1(\ld^2;\ld)\left|\f{\nabla\psi}{x}-\ld
e\right|^2\chi_{\{\psi<m_0\}\cap E},\ea\ee and
\be\label{b16}\ba{rl}{}G(\g\psi,\psi,x;\ld)\leq
Cx\left|\f{\nabla\psi}{x}-\ld e\right|^2\chi_{\{\psi<m_0\}\cap
E}+xF_1\left(\left|\f{\nabla\psi}{x}\right|^2;\ld\right)\chi_{\{\psi<m_0\}\setminus
E}.\ea\ee

Hence, we define a functional
$$J_{\ld}(\psi)=\int_\O G(\g\psi,\psi,x;\ld)dxdy.$$
\begin{figure}[!h]
\includegraphics[width=100mm]{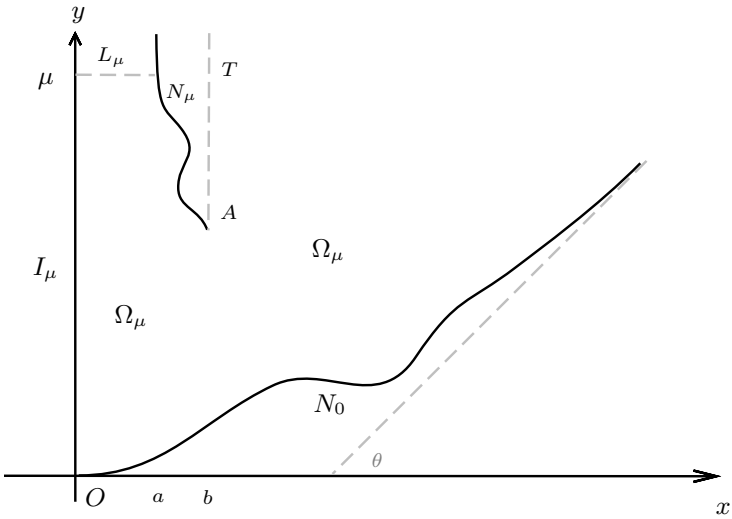}
\caption{Truncated domain}\label{f3}
\end{figure}
It follows from \eqref{b15} that the functional $J_{\ld}(\psi)$ is
non-negative for any $\psi\in K$. Obviously, $J_{\ld}(\psi)$ is
unbounded for any $\psi\in K$. Thus we will truncate the domain as
$\O_\mu$ for any $\mu>1$ (see Figure \ref{f3}), which is bounded by $N_\mu, I_\mu,N_0, L_\mu$ and $T=\{(b,y)\mid y>1\}$, where $$\text{$N_{\mu}=N\cap\{x\geq x_\mu\}$, $I_\mu=I\cap\{y\leq\mu\}$, and $L_{\mu}=\{(x,y)\mid 0\leq x\leq x_\mu,y=\mu\}$},$$ with $x_\mu=\min\{x\mid g(x)=\mu\}$. Define the following functional in the truncated domain $\O_\mu$,
$$J_{\ld,\mu}(\psi)=\int_{\O_\mu} G(\g\psi,\psi,x;\ld)dxdy.$$

To overcome the singularity of the functional $J_{\ld,\mu}$ near $y$-axis, we first consider the following variational problem.

{\bf The truncated variational problem $(P^\delta_{\ld,\mu})$}: For
any $\ld\leq\Pi_\ld-3\t\e$, $\mu>1$ and small $\delta>0$, find a $\psi_{\ld,\mu}^\delta\in
K_{\mu}^\delta$ such that
$$J^\delta_{\ld,\mu}(\psi^\delta_{\ld,\mu})=\min_{\psi\in K^\delta_{\mu}}
J^\delta_{\ld,\mu}(\psi),$$ where $$J^\delta_{\ld,\mu}(\psi)=
\int_{\O_\mu} G(\g\psi,\psi,x+\delta;\ld)dxdy$$ and $$
K_{\mu}^\delta=\left\{\psi\in K\mid\psi=\min\left\{\f{m_0}{x^2_\mu}((x+\delta)^2-\delta^2),m_0\right\}~~~~\text{on}~~L_{\mu}\right\}.$$

\begin{lem}\label{lb1}
The variational problem $(P^\delta_{\ld,\mu})$ has a minimizer
$\psi^\delta_{\ld,\mu}$ and $\psi_{\ld,\mu}^\delta\in C^{0,1}(\O_\mu)$. Furthermore, the minimizer
$\psi^\delta_{\ld,\mu}$ satisfies that
$$\int_{\O_\mu}f_\eta\left(\f{\g\psi^\delta_{\ld,\mu}}{x+\delta};\ld\right)\cdot\g\xi dxdy\geq0\ \ \text{for any $\xi\in C^\infty_0(\O_\mu)$ and $\xi\geq 0$},$$
and
\be\label{b17}\int_{\O_\mu}f_\eta\left(\f{\g\psi^\delta_{\ld,\mu}}{x+\delta};\ld\right)\cdot\g\zeta
dxdy=0\ \ \text{for any $\zeta\in
C^\infty_0(\O_\mu\cap\{\psi^\delta_{\ld,\mu}<m_0\})$},\ee where $f(\eta;\lambda)$ is defined in \eqref{bf1}.
Furthermore, \be\label{b18}\text{ $0\leq\psi_{\ld,\mu}^\delta(x,y)\leq
\min\left\{\f{m_0}{x^2_\mu}((x+\delta)^2-\delta^2),m_0\right\}\ $ in $\ \O_\mu$.}\ee
\end{lem}

\begin{proof}


Define
$$\psi_0(x,y)=\min\{\ld\max\{(x+\delta)\left((y-g_0(R_0))\cos\theta-(x+\delta-R_0)\sin\theta\right),0\},m_0\},$$
it follows from \eqref{a09} that $\psi_0=0$ on $N_0\cap\{x\geq R_0\}$. Then we
can extend $\psi_0$ into the domain $\O_\mu\setminus\{x\leq R_0\}$
so that it belongs to the admissible set $K^\delta_{\mu}$ and
$\{\psi_0<m_0\}\cap(\O_\mu\setminus\{x\leq R_0\})$ is bounded.
Hence, it suffices to verify that $$\int_{\O_\mu\cap\{x>
R_0\}}G(\g\psi_0,\psi_0,x+\delta;\ld)dxdy<+\infty.$$ In fact, it
follows from \eqref{b16} that
$$\ba{rl} &\int_{\O_\mu\cap\{x> R_2\}}G(\g\psi_0,\psi_0,x+\delta;\ld)dxdy\\
&\leq C\int_{R_2}^{+\infty}\int_{g_0(R_0)+(x+\delta-R_0)\tan\theta}^{g_0(R_0)+(x+\delta-R_0)\tan\theta+\f{m_0}{\l(x+\delta)\cos\theta}}
(x+\delta)\left|\f{\nabla\psi_0}{x+\delta}-\l e\right|^2dydx\\
&\leq C\int_{R_0+\delta}^{+\infty}\int_{g_0(R_0)+(x-R_0)\tan\theta}^{g_0(R_0)+(x-R_0)\tan\theta+\f{m_0}{\l x\cos\theta}}\f{\l^2\left((y-g_0(R_0))\cos\theta-(x-R_0)\sin\theta\right)^2}{x}dydx\\
&\leq\f{C}{(R_0+\delta)^3}.\ea$$
The existence of the minimizer to the variational problem
$(P_{\ld,\mu}^\delta)$ can be obtained via the similar arguments
in Lemma 1.1 in \cite{ACF3} and Theorem 1.1 in \cite{ACF4}, and
denote $\psi^\delta=\psi^\delta_{\ld,\mu}$ as the minimizer to the
variational problem $(P^\delta_{\ld,\mu})$ for simplicity.

For any nonnegative function $\xi\in C_0^{\infty}(\O_\mu)$ and
$\e>0$, it is easy to check that $\min\{\psi^\delta+\e\xi,m_0\}\in
K_\mu$ and
$\{\min\{\psi^\delta+\e\xi,m_0\}<m_0\}\subset\{\psi^\delta<m_0\}$.
Thus, we have
$$\ba{rl}0\leq&J^\delta_{\ld,\mu}\left(\min\{\psi^\delta+\e\xi,m_0\}\right)-J^\delta_{\ld,\mu}(\psi^\delta)\\
=&\int_{\O_\mu}(x+\delta)F\left(\f{|\g\min\{\psi^\delta+\e\xi,m_0\}|^2}{(x+\delta)^2};\ld\right)-(x+\delta)F\left(\f{|\g\psi^\delta|^2}{(x+\delta)^2};\ld\right)dxdy\\
\leq&\int_{\O_\mu\cap\{\psi^\delta+\e\xi\leq
m_0\}}(x+\delta)F\left(\f{|\g(\psi^\delta+\e\xi)|^2}{(x+\delta)^2};\ld\right)-(x+\delta)
F\left(\f{|\g\psi^\delta|^2}{(x+\delta)^2};\ld\right)dxdy\\
\leq&2\e\int_{\O_\mu\cap\{\psi^\delta+\e\xi\leq
m_0\}}F_1\left(\f{|\g(\psi^\delta+\e\xi)|^2}{(x+\delta)^2};\ld\right)\f{\g\psi^\delta\cdot\g\xi}{x+\delta}dxdy+o(\e),\ea$$ which implies that
$$0\leq\int_{\O_\mu\cap\{\psi^\delta+\e\xi\leq m_0\}}F_1\left(\f{|\g(\psi^\delta+\e\xi)|^2}{(x+\delta)^2};\ld\right)\f{\g\psi^\delta\cdot\g\xi}{x+\delta}dxdy.$$
Taking $\e\rightarrow 0$ in above inequality, we have
$$0\leq\int_{\O_\mu}F_1\left(\f{|\g\psi^\delta|^2}{(x+\delta)^2};\ld\right)\f{\g\psi^\delta\cdot\g\xi}{x+\delta}dxdy.$$
Similarly, we can verify that \eqref{b17} holds.

Next, we will show that
\be\label{b19}\psi^\delta(x,y)\geq 0\ \ \ \text{in}\ \ \O_\mu.\ee Denote $\psi^\delta_\e=\psi^\delta-\e\min\{\psi^\delta,0\}$ for $\e\in(0,1)$. It is easy to check that $\psi_\e^\delta\in K_\mu^\delta$,
$$\psi_\e^\delta>0\ \ \text{if and only if}\ \ \psi^\delta>0\ \ \text{and}\ \ \psi_\e^\delta\geq\psi^\delta\ \ \text{in}\ \ \O_\mu.$$
Since $\psi^\delta$ is the minimizer to the truncated variational problem $(P_{\ld,\mu}^\delta)$, one has
\be\label{b190}0\leq J^\delta_{\ld,\mu}(\psi_\e^\delta)-J^\delta_{\ld,\mu}(\psi^\delta).\ee For any sufficiently large $R>0$, denote $\O_{\mu,R}=\O_\mu\cap\{y<R\}$ and $E_{R}=\O_{\mu,R}\cap\{x>b\}$, we have
\be\label{b191}\ba{rl}&\int_{\O_{\mu,R}}G(\g\psi_\e^\delta,\psi_\e^\delta,x+\delta;\ld)dxdy-\int_{\O_{\mu,R}}G(\g\psi^\delta,\psi^\delta,x+\delta;\ld)dxdy\\
=&\int_{\O_{\mu,R}}(x+\delta)F\left(\left|\f{\g\psi_\e^\delta}{x+\delta}\right|^2;\ld\right)
-(x+\delta)F\left(\left|\f{\g\psi^\delta}{x+\delta}\right|^2;\ld\right)dxdy\\
&-2\ld F_1(\ld^2;\ld)\int_{\O_{\mu,R}\cap E}\g\psi_\e^\delta\cdot e\chi_{\{\psi^\delta_\e<m_0\}}-\g\psi^\delta\cdot e\chi_{\{\psi^\delta<m_0\}}dxdy\\
\leq&\int_{\O_{\mu,R}}F_1\left(\left|\f{\g\psi_\e^\delta}{x+\delta}\right|^2;\ld\right)\f{|\g\psi_\e^\delta|^2-|\g\psi^\delta|^2}{x+\delta}dxdy\\
&-2\ld F_1(\ld^2;\ld)\int_{\p E_R}(\psi_\e^\delta-\psi^\delta)e\cdot \nu dS \\
\leq&\int_{\O_{\mu,R}}F_1\left(\left|\f{\g\psi_\e^\delta}{x+\delta}\right|^2;\ld\right)\f{((1-\e)^2-1)|\g\min\{\psi^\delta,0\}|^2}{x+\delta}dxdy.\ea\ee Here, we have used the fact
$$\int_{\p E_R}(\psi_\e^\delta-\psi_\e)e\cdot \nu dS=\sin\th\int_{\p E_R\cap\{x=b\}}(\psi_\e^\delta-\psi_\e)dS+\cos\th\int_{\p E_R\cap\{y=R\}}(\psi_\e^\delta-\psi_\e) dS\geq 0.$$
Taking $R\rightarrow+\infty$ in \eqref{b191}, it follows from \eqref{b190} that
$$0\leq((1-\e)^2-1)\int_{\O_{\mu}}F_1\left(\left|\f{\g\psi_\e^\delta}{x+\delta}\right|^2;\ld\right)\f{|\g\min\{\psi^\delta,0\}|^2}{x+\delta}dxdy,$$
which implies that \eqref{b19} holds.

Since $0\leq\psi^\delta_{\ld,\mu}\leq m_0$ in $\O_\mu$, it suffices
to show that
\be\label{b192}\psi^\delta_{\ld,\mu}\leq\f{m_0}{x_\mu^2}((x+\delta)^2-\delta^2)\ ~~\text{in}~~\ \O_\mu\cap\left\{x\leq
\sqrt{x^2_\mu+\delta^2}-\delta\right\}.\ee

In view of \eqref{b17}, the maximum principle gives that the
inequality \eqref{b18} holds.

\end{proof}

With the aid of Lemma \ref{lb1}, we consider the following truncated variational problem.

{\bf The truncated variational Problem $(P_{\ld,\mu})$}: For any
$\ld\leq\Pi_\ld-3\t\e$ and $\mu>1$, find a $\psi\in K_{\mu}$ such
that
$$J_{\ld,\mu}(\psi_{\ld,\mu})=\min_{\psi\in K_{\mu}} J_{\ld,\mu}(\psi),$$ where $$ K_{\mu}=\left\{\psi\in K\mid \psi=\f{m_0}{x^2_\mu}x^2~~~~\text{on}~~L_{\mu}\ \ \text{and}\ \ \psi\leq\min\left\{\f{m_0}{x_\mu^2}x^2,m_0\right\}\ \ \text{a.e. in}\ \ \O_\mu \right\}.$$

\subsection{Existence and fundamental properties of minimizer}

\begin{lem}\label{lb2}
There exists a minimizer $\psi_{\ld,\mu}$ to the variational problem $(P_{\ld,\mu})$ and $\psi_{\ld,\mu}\in C^{0,1}(\O_\mu)$. Moreover, \\
(1) the minimizer $\psi_{\ld,\mu}$
satisfies that
$$\int_{\O_\mu}f_\eta\left(\f{\g\psi_{\ld,\mu}}x;\ld\right)\cdot\g\xi dxdy\geq0\ \ \text{for any $\xi\in C^\infty_0(\O_\mu)$ and $\xi\geq 0$},$$
and $$\t Q_\ld\psi_{\ld,\mu}=0\ \ \text{in $\O_\mu\cap\{\psi_{\ld,\mu}<m_0\}$ and $\psi_{\ld,\mu}\in C^{2,\alpha}(\O_\mu\cap\{\psi_{\ld,\mu}<m_0\})$}.$$ Furthermore,  \be\label{b20}\text{ $\psi_{\ld,\mu}(x,y)\geq 0$ in $\O_\mu$.}\ee
(2) the free boundary $\Gamma_{\ld,\mu}=E\cap\p\{\psi_{\ld,\mu}<m_0\}$ is analytic, and
$$\f1x|\g\psi_{\ld,\mu}|=\ld\ \ \text{on}\ \ \Gamma_{\ld,\mu},$$ and
$$\f1x|\g\psi_{\ld,\mu}|\geq\ld\ \ \text{on}\ \ l,$$ where $l$ is a segment
with $l\subset T\cap\p\{\psi_{\ld,\mu}<m_0\}$.
\end{lem}
\begin{proof}(1) Along the similar arguments in the proof of Lemma 1.1 in \cite{ACF3} and Theorem 1.1 in \cite{ACF4}, one has that there exists a sequence $\{\delta_n\}$ with $\delta_n\rightarrow0$ as
$n\rightarrow+\infty$, such that $$\psi_{\ld,\mu}^{\delta_n}\rightharpoonup\psi_0\ \ \text{in
$H^1_{loc}(\O_\mu)$ and $\psi_{\ld,\mu}^{\delta_n}\rightarrow\psi_0$ uniformly in any compact subset of $\O_\mu$}. $$ It follows from \eqref{b18} that
$$0\leq\psi_0(x,y)\leq\min\left\{\f{x^2}{x_\mu^2}m_0,m_0\right\}\ \ \text{in}\ \ \O_\mu,$$ which together with \eqref{b17} gives that $$\g\cdot f_\eta\left(\f{\nabla\psi_0}{x};\ld\right)=0\ \ \text{in}\ \ \O_\mu\cap\left\{x<\f{x_\mu}2\right\}.$$
Next, we will check that $J_{\ld,\mu}(\psi_0)<+\infty$. By virtue of the proof of Lemma \ref{lb1}, it suffices to show that
$$\f{|\g\psi_0(X)|}{x}\leq C\ \ \text{near $I$}.$$ For any $X_0=(x_0,y_0)\in\O_\mu$ with $x_0<\f{x_\mu}4$, denote $\phi(X)=\f{\psi_0(X_0+r_0X)}{r^2_0}$ with $r_0=\f{x_0}2$. It is easy to check that
$$\g\cdot f_\eta\left(\f{\nabla\phi}{2+x};\ld\right)=0\ \ \text{and}\ \ 0\leq\phi\leq(2+x)^2m_0\ \ \text{in}\ \ B_1(0).$$
Thanks to the gradient estimate in Chapter 12 in \cite{GT}, one has
$$|\g\phi(0)|\leq C,$$ where the constant $C$ is independent of $x_0$. This gives that
$$|\g\psi_0(X_0)|=r_0|\g\phi(0)|\leq Cx_0.$$

Since $J_{\ld,\mu}(\psi_0)<+\infty$ and $\psi_0\in K_{\mu}$, the minimal functional $J_{\ld,\mu}(\psi)$ is finite. By using the proof of Lemma 1.1 in \cite{ACF3}, we can conclude that there exists a minimizer to the variational problem $(P_{\ld,\mu})$.

Denote $\psi_{\ld,\mu}$ be the minimizer to the variational problem $(P_{\ld,\mu})$ and $\Gamma_{\ld,\mu}=E\cap\p\{\psi_{\ld,\mu}<m_0\}$ as the free boundary of $\psi_{\ld,\mu}$. Thanks to Lemma \ref{lb1}, we can show that $\psi_{\ld,\mu}\in C^{0,1}(\O_\mu)$ satisfies the assertion (1) of this lemma.

(2) Since $F(t;\ld)$ is $C^2$-smooth with respect to $t\in[0,+\infty)$, it
follows from Theorem 6.3 in \cite{ACF4} that the free boundary $\Gamma_{\ld,\mu}$ is $C^{1,\alpha}$, and thus $\psi_{\ld,\mu}$ is $C^{1,\alpha}$ up to the free boundary. Since $\ld\leq\Pi_\ld-3\t\e$, the subsonic cut-off can be removed near the free boundary. Then $F(t;\ld)$ is analytic near $\Gamma_{\ld,\mu}$, the Remark 6.4 in \cite{ACF4} gives that the free boundary $\Gamma_{\ld,\mu}$ is analytic.
By using the similar arguments in the proof of Lemma 9.1 in \cite{CDW}, we can conclude that $$\f1x|\g\psi_{\ld,\mu}|=\ld\ \ \text{on}\ \ \Gamma_{\ld,\mu},\ \ \text{ and}\ \ \f1x|\g\psi_{\ld,\mu}|\geq\ld\ \ \text{on}\ \ l,$$ where $l$ is a segment
with $l\subset T\cap\p\{\psi_{\ld,\mu}<m_0\}$.
  \end{proof}

  Next, we will give the bounded gradient lemma in the following.
\begin{lem}\label{lb3} Let $X_0=(x_0,y_0)$ be a free boundary point
and let $B_r(X_0)\subset B_R(X_0)\subset E$ with $r<R$. Then
$$|\g\psi_{\ld,\mu}(X)|\leq C\Lambda x\ \ \text{in $B_r(X_0)$,}$$ where the constant $C$ depends only on $\vartheta, N, N_0$ and $\left(1-\f rR\right)^{-1}$, but not on
$m_0$.
\end{lem}\begin{figure}[!h]
\includegraphics[width=100mm]{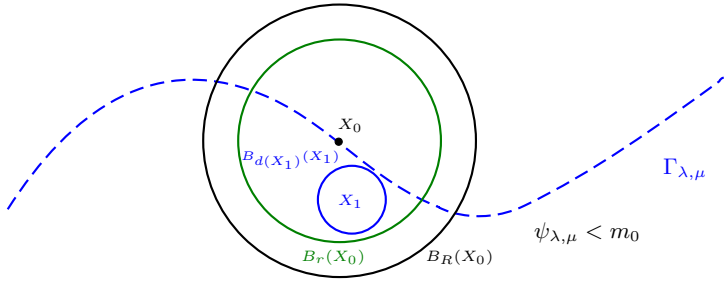}
\caption{$B_{d(X_1)}(X_1)$ and $B_R(X_0)$}\label{fi5}
\end{figure}

\begin{proof} {\bf Step 1.} In this step, we will show that
\be\label{b21}m_0-\psi_{\ld,\mu}(X)\leq C\Lambda
xd(X) \ \ \text{for any $X\in B_{r}(X_0)$},\ee where $d(X)=dist(X,\Gamma_{\ld,\mu})$ and the constant $C$ depends only on $\vartheta, b$ and $\left(1-\f rR\right)^{-1}$.

Denote $d_0=R-r$. Suppose that $X_1=(x_1,y_1)\in B_r(X_0)\cap\{\psi<m_0\}$ and $d(X_1)<d_0$. Thus $B_{d(X_1)}(X_1)\subset B_R(X_0)\cap\{\psi_{\ld,\mu}<m_0\}$ (please see Figure \ref{fi5}). Next, we assume that
\be\label{b22}m_0-\psi_{\ld,\mu}(X_1)>Md(X_1)x_1,\ee and we will derive an upper bound of $M$ in the following.
Denote \be\label{b23}\text{$\phi(X)=\f{m_0-\psi_{\ld,\mu}(X_1+dX)}{dx_1}$ with $d=d(X_1)$,}\ee and one has $$\g\cdot f_\eta\left(\f{x_1\g\phi}{x_1+d x};\ld\right)=0  \ \text{in}~~B_{1}(0). $$
 It follows from \eqref{b22} that $$\phi(0)>M.$$
It follows from Harnack's inequality (see Theorem 8.20 in
\cite{GT}) that \be\label{b24}\phi(X)\geq cM\ \ \text{in
$B_{\f34}(0)$ for some $c>0$},\ee where the constant $c$ is
independent of $d$ and $x_1$. On another hand, there exists a
$\t X=(\t x,\t y)\in\p B_{1}(0)\cap\Gamma_{\ld,\mu}$. Define a function $\Psi$, which satisfies that $$\left\{\ba{ll} &\g\cdot f_\eta\left(\f{x_1\g\Psi}{x_1+d x};\ld\right)=0  \ \text{in}~~B_{1}(\t X),\\
&\Psi=\phi\ \ \
\text{outside} \ \  B_{1}(\t X). \ea\right. $$ Since $\g\cdot f_\eta\left(\f{x_1\g\phi}{x_1+d x};\ld\right)\geq 0$ in $B_{1}(\t X)$, the maximum principle gives that \be\label{b25}\text{$\phi\leq\Psi$ in $B_1(\t X)$}.\ee Then we have
$$\ba{rl}0\leq&\int_{B_1(\t X)}(x_1+dx)\left\{F\left(\f{|x_1\nabla\Psi|^2}{(x_1+dx)^2};\ld\right)-F\left(\f{|x_1\nabla\phi|^2}{(x_1+dx)^2};\ld\right)\right\}dxdy\\
&-2\ld F_1(\ld^2;\ld)\int_{B_1(\t X)}\g(\Psi-\phi)\cdot e dxdy+\Ld^2\int_{B_1(\t X)}(x_1+dx)(\chi_{\{\Psi>0\}}-\chi_{\{\phi>0\}})dxdy\\
\leq&\int_{B_1(\t X)}-\vartheta\f{x_1^2|\g(\Psi-\phi)|^2}{x_1+dx}+f_\eta\left(\f{x_1\nabla\Psi}{x_1+dx};\ld\right)\cdot\g(\Psi-\phi)+\Ld^2(x_1+dx)\chi_{\{\phi=0\}}dxdy\\
=&-\vartheta\int_{B_1(\t X)}\f{x_1^2|\g(\Psi-\phi)|^2}{x_1+dx}dxdy
+\Ld^2\int_{B_1(\t X)}(x_1+dx)\chi_{\{\phi=0\}}dxdy,\ea$$ which implies that
\be\label{b26}\ba{rl}\int_{B_1(\t X)}|\g(\Psi-\phi)|^2dxdy\leq& C\int_{B_1(\t X)}\f{x_1|\g(\Psi-\phi)|^2}{x_1+dx}dxdy\\
\leq& \f{C\Ld^2}{x_1}\int_{B_1(\t X)}(x_1+dx)\chi_{\{\phi=0\}}dxdy\\
\leq& C\Ld^2\int_{B_1(\t X)}\chi_{\{\phi=0\}}dxdy,\ea\ee where $C$ is a constant depending only on $\vartheta$ and $\left(1-\f rR\right)^{-1}$.

It follows from \eqref{b24} and \eqref{b25} that
$$\Psi(X)\geq\phi(X)\geq cM\ \ \text{in}\ \ B_{\f34}(0)\cap B_{1}(\t X).$$ Applying Harnack's inequality for $\Psi$ in $B_1(\t X)$, one has \be\label{b27}\Psi(X)\geq C_0\ \ \text{in}\ \  B_{\f12}(\t X),\ \ \ C_0=cM.\ee Define $\varphi(X)=C_0\left(e^{-\nu|X-\t X|^2-e^{-\nu}}\right)$, after a direct computation, we have $$\g\cdot f_\eta\left(\f{x_1\g\varphi}{x_1+d x};\ld\right)=\f{2\nu C_0 x_1 e^{-\nu|\xi|^2}}{x_1+dx}\left(f_{\eta_i\eta_j}(2\nu \xi_i\xi_j-\delta_{ij})+\f{d(f_{\eta_1\eta_1}\xi_1+f_{\eta_1\eta_2}\xi_2)}{x_1+dx}\right)>0$$ in $B_1(\t X)\setminus B_{\f12}(\t X)$, provided that $\nu$ is large enough, where $f_{\eta_i\eta_j}=f_{\eta_i\eta_j}\left(\f{x_1\g\varphi}{x_1+d x};\ld\right)$ and $\xi=(\xi_1,\xi_2)=X-\t X$ for $i,j=1,2$.

It is easy to check that
$$\Psi\geq\varphi\ \ \text{on}\ \ \p(B_1(\t X)\setminus B_{\f12}(\t X)).$$ The maximum principle gives that
 $$\Psi(X)\geq\varphi(X)=C_0\left(e^{-\nu|X-\t X|^2-e^{-\nu}}\right)\geq cC_0(1-|X-\t X|)\ \ \text{in}\ \ B_1(\t X)\setminus B_{\f12}(\t X),$$ which together with \eqref{b27} gives that  \be\label{b28}\Psi(X)\geq cM(1-|X-\t X|)\ \ \text{in}\ \ B_1(\t X)\setminus B_{\f12}(\t X).\ee

With the aid of \eqref{b26} and \eqref{b28}, along the similar
arguments in the proof of Lemma 3.2 in \cite{AC1} and Lemma 2.2 in
\cite{ACF4}, one has
$$M^2\leq C\Ld^2,$$ where the constant $C$ depends only on $\vartheta, N, N_0$ and $\left(1-\f rR\right)^{-1}$. This implies that
\be\label{b29}m_0-\psi_{\ld,\mu}(X_1)\leq C\Ld d(X_1)x_1.\ee

Take any point $X_2=(x_2,y_2)\in B_r(X_0)$ such that $d(X_2)>d_0$ and there exists a point $X_1\in B_{\f{d_0}2}(X_2)$ with $d(X_1)<d_0$. By using Harnack's inequality for $m_0-\psi_{\ld,\mu}$ in $B_{d_0}(X_2)$ and \eqref{b29}, one has
$$m_0-\psi_{\ld,\mu}(X_2)\leq C(m_0-\psi_{\ld,\mu}(X_1))\leq C\Ld d(X_1)x_1\leq C\Ld d(X_2)x_2.$$  For any $X\in B_r(X_0)$, we can repeat this argument step by step, and after a finite steps $N$ (depending only on $\vartheta, N, N_0$ and $\left(1-\f rR\right)^{-1}$), such that $$m_0-\psi_{\ld,\mu}(X)\leq C\Ld d(X)x.$$ Hence, we complete the proof of \eqref{b21}.

{\bf Step 2.} In this step, we will complete the proof of this lemma. For any $X_1\in B_r(X_0)$, denote $d_0=R-r$ and $d(X)=dist(X,\Gamma_{\ld,\mu})$, and we consider the following two cases.

{\bf Case 1}. $d(X_1)<d_0$. Then it follows from \eqref{b21} that
$$\g\cdot f_\eta\left(\f{x_1\g\phi}{x_1+d x};\ld\right)=0 \ \ \text{and}\ \ 0\leq \phi\leq \f{C\Ld(x_1+dx)d(X_1+dX)}{dx_1}\leq C\Ld \ \text{in}~~B_{1}(0),$$ where $\phi$ and $d$ are defined in \eqref{b23}, the constant $C$ depends only on $\vartheta, N, N_0$ and $\left(1-\f rR\right)^{-1}$. Applying the elliptic estimate for the quasilinear equation in \cite{GT}, one has $$|\g\phi(0)|\leq C,$$ which gives that $$|\g\psi_{\ld,\mu}(X_1)|=x_1|\g\phi(0)|\leq Cx_1.$$

{\bf Case 2}. $d(X_1)\geq d_0$. Obviously, $B_{d_0}(X_1)\subset B_R(X_0)\cap\{\psi_{\ld,\mu}<m_0\}$. Denote $\phi_0(X_1)=\f{m_0-\psi_{\ld,\mu}(X_1+d_0X)}{d_0x_1}$, it follows from \eqref{b23} that $$\g\cdot f_\eta\left(\f{x_1\g\phi_0}{x_1+d_0 x};\ld\right)=0 \ \ \text{and}\ \ 0\leq \phi_0\leq \f{C\Ld(x_1+d_0x)d(X_1+d_0X)}{d_0x_1}\leq C\Ld \ \text{in}~~B_{1}(0).$$ By using the elliptic estimate for $\phi_0$ in $B_1(0)$, one has $$|\g\phi_0(0)|\leq C\ \ \text{and}\ \ |\g\psi_{\ld,\mu}(X_1)|=x_1|\g\phi_0(0)|\leq Cx_1.$$

\end{proof}

With the aid of Lemma \ref{lb3}, applying the similar arguments in the proof of Lemma 2.4 in \cite{ACF4}, we can obtain the following lemma.

\begin{lem}\label{lb4}
There exists a
positive constant $C^*$, such that for any disc
$B_r(X_0)\subset \subset\O_\mu$ with $X_0=(x_0,y_0)$, $r\leq \f{x_0}{2}$, then
$$\f{1}{r}\fint_{\p B_r(X_0)}(m_0-\psi_{\ld,\mu})dS\geq C^*\Lambda x_0,$$ implies that
$$\psi_{\ld,\mu}<m_0~~~~\text{in}~~B_r(X_0).$$
\end{lem}

We next establish a non-degeneracy lemma.

\begin{lem}\label{lb5} There is a universal constant $c^*>0$ such that for any disc $B_{r}(X_0)$ with center $X_0=(x_0,y_0)\in E$ and $r\leq \f{x_0}{2}$, then
\be\label{b30}\f{1}{r}\left(\fint_{B_{r}(X_0)}(m_0-\psi_{\ld,\mu})^2dxdy\right)^{\f{1}{2}}\leq c^*\Lambda x_0,\ee implies that
$$\psi_{\ld,\mu}=m_0~~~~\text{in}~~B_{\f r8}(X_0)\cap E.$$
\end{lem}
\begin{proof}
It is easy to check that the set
$$B_{\f r 8}(X_0)\cap E$$ can be covered by discs of the form
$$B_{r_1}(X_1)\subset B_{\f r4}(X_0)\cap E\ \ \text{with}\ \ r_1=\f{r}{16}.$$ Thus, it suffices to show that $\psi_{\ld,\mu}=m_0$ in any discs $B_{r_1}(X_1)\subset B_{\f r4}(X_0)\cap E$, provided that the assumption \eqref{b30} holds. Let $\psi_0$ solves the following boundary value problem
\be\label{b31}\left\{\begin{array}{ll}&\g\cdot f_\eta\left(\f{\nabla
\psi_0}{x};\ld\right)=0\ \ \text{in}\ B_{2r_1}(X_1)\setminus B_{r_1}(X_1),\\
&\psi_0=m_0\ \text{in}\ \overline{B_{r_1}(X_1)},\ \psi_0=\psi_{\ld,\mu}\ \text{outside of}\ B_{2r_1}(X_1).\end{array}\right.\ee Obviously, $\max\{\psi_0,\psi_{\ld,\mu}\}\in K_\mu$, and thus
\be\label{b32}\ba{rl}0\leq& J_{\ld,\mu}(\max\{\psi_{\ld,\mu},\psi_0\})-J_{\ld,\mu}(\psi_{\ld,\mu})\\
=&\int_{B_{2r_1}(X_1)} xF\left(\left|\f{\g\max\{\psi_{\ld,\mu},\psi_0\}}{x}\right|^2;\ld\right)- xF\left(\left|\f{\g\psi_{\ld,\mu}}{x}\right|^2;\ld\right)dxdy\\
&-2F_1(\ld^2;\ld)\ld\int_{B_{2r_1}(X_1)} \g\max\{\psi_0-\psi_{\ld,\mu},0\}\cdot edxdy\\
&+\Ld^2\int_{B_{2r_1}(X_1)} x\chi_{\{\max\{\psi_{\ld,\mu},\psi_0\}<m_0\}}-x\chi_{\{\psi_{\ld,\mu}<m_0\}} dxdy\\
=&I_1+I_2+I_3.\ea\ee For the first term on the right hand side of \eqref{b32}, one has
\be\label{b33}\ba{rl}I_1\leq&\int_{B_{2r_1}(X_1)\setminus B_{r_1}(X_1)} \g\max\{\psi_0-\psi_{\ld,\mu},0\}\cdot f_\eta\left(\f{\g\max\{\psi_{\ld,\mu},\psi_0\}}{x};\ld\right)dxdy\\
&-\int_{B_{r_1}(X_1)} xF\left(\left|\f{\g\psi_{\ld,\mu}}{x}\right|^2;\ld\right)dxdy\\
\leq& -\int_{B_{r_1}(X_1)} xF\left(\left|\f{\g\psi_{\ld,\mu}}{x}\right|^2;\ld\right)dxdy\\
&+2\int_{\p B_{r_1}(X_1)}(m_0-\psi_{\ld,\mu}) F_1\left(\left|\f{\g\psi_0}{x}\right|^2;\ld\right)\f{\g\psi_0\cdot\nu}{x} dS.\ea\ee

It is easy to check that $I_2=0$ and
$$I_3=-\Ld^2\int_{B_{2r_1}(X_1)} x\chi_{\{\psi_{\ld,\mu}<\psi_0=m_0\}}dxdy\leq -\Ld^2\int_{B_{r_1}(X_1)} x\chi_{\{\psi_{\ld,\mu}<m_0\}}dxdy,$$
which together with \eqref{b32} and \eqref{b33} gives that
\be\label{b34}\ba{rl}&\int_{B_{r_1}(X_1)} xF\left(\left|\f{\g\psi_{\ld,\mu}}{x}\right|^2;\ld\right)dxdy+\Ld^2\int_{B_{r_1}(X_1)} x\chi_{\{\psi_{\ld,\mu}<m_0\}}dxdy\\
\leq&2\int_{\p B_{r_1}(X_1)}(m_0-\psi_{\ld,\mu}) F_1\left(\left|\f{\g\psi_0}{x}\right|^2;\ld\right)\f{\g\psi_0\cdot\nu}{x} dS.\ea\ee
Set $\t\psi(X)=\f{m_0-\psi_{\ld,\mu}(X_1+r_1X)}{r_1x_0}$ and $\t\psi_0(X)=\f{m_0-\psi_0(X_1+r_1X)}{r_1x_0}$, one has
$$\g\cdot f_\eta\left(\f{x_0\nabla\t\psi}{x_1+r_1x};\ld\right)\geq0\ \ \text{in}\  B_{16}\left(\f{X_0-X_1}{r_1}\right).$$ Moreover, it follows from the assumption \eqref{b30} that
  $$\left(\fint_{B_{16}\left(\f{X_0-X_1}{r_1}\right)}\t\psi^2 dxdy\right)^{\f{1}{2}}\leq \delta\Lambda,$$ where $\delta$ is to be chosen later on. By using the $L^\infty$ estimate in Theorem 8.17 in \cite{GT}, one has
\be\label{b35}\sup_{X\in B_{8}\left(\f{X_0-X_1}{r_1}\right)}\t\psi(X)\leq C\left(\fint_{B_{16}\left(\f{X_0-X_1}{r_1}\right)}\t\psi^2 dxdy\right)^{\f{1}{2}}\leq C\delta\Lambda,\ee where $C$ is a constant depending only on $\vartheta$ and $a$. Since $B_2(0)\subset B_{8}\left(\f{X_0-X_1}{r_1}\right)$,  it follows from \eqref{b35} that
\be\label{b36}0\leq\t\psi_0=\t\psi\leq C\Ld\delta\ \ \text{on}\ \ \p B_2(0)\ \ \text{and}\ \ \t\psi_0=0\ \ \text{on}\ \ \p B_1(0).\ee It is easy to check that
$$\g\cdot f_\eta\left(\f{x_0\nabla
\t\psi_0}{x_1+r_1x};\ld\right)=0\ \ \text{in}\ B_{2}(0)\setminus
B_{1}(0).$$ Applying the boundary elliptic estimate in Lemma 6.10 in
\cite{GT}, one has $$|\g\t\psi_0\cdot\nu|\leq C\delta\Ld \ \
\text{on}\ \ \p B_1. $$ In view of \eqref{b34}, \eqref{b36} and the
trace theorem, one has
\be\label{b37}\ba{rl}&\int_{B_{1}(0)} (x_1+r_1x)F\left(\f{x_0^2|\g\t\psi|^2}{(x_1+r_1x)^2};\ld\right)+\Ld^2 (x_1+r_1x)\chi_{\{\t\psi>0\}}dxdy\\
\leq&2x_0\int_{\p B_{1}(0)}\t\psi F_1\left(\f{x^2_0|\g\t\psi_0|^2}{|x_1+r_1x|^2};\ld\right)\f{x_0|\g\psi_0\cdot\nu|}{x_1+r_1x} dS\\
\leq&Cx_0\delta\Ld\int_{\p B_{1}(0)}\t\psi dS\\
\leq&Cx_0\delta\Ld\left(\int_{B_{1}(0)}\t\psi dxdy+\int_{B_{1}(0)}|\g\t\psi| dxdy\right)\\
\leq&Cx_0\delta\Ld\left\{\left(C\delta\Ld+\f1{4\e}\right)\int_{B_{1}(0)}\chi_{\{\t\psi>0\}} dxdy+\e\int_{B_{1}(0)}|\g\t\psi|^2 dxdy\right\},\ea\ee where we have used the fact
$$|\g\t\psi|\leq \e|\g\t\psi|^2+\f1{4\e}\ \ \text{a.e. in $B_1\cap\{\t\psi>0\}$}.$$
On the other hand, we have
\be\label{b38}\ba{rl}&x_0\int_{B_{1}(0)}|\g\t\psi|^2dxdy+x_0\int_{B_{1}(0)}\Ld^2\chi_{\{\t\psi>0\}}dxdy\\
\leq&C\int_{B_{1}(0)} (x_1+r_1x)F\left(\f{x_0^2|\g\t\psi|^2}{(x_1+r_1x)^2};\ld\right)+\Ld^2(x_1+r_1x)\chi_{\{\t\psi>0\}}dxdy\\
\leq &Cx_0\delta\Ld\left\{\left(C\delta\Ld-\f1{4\e}\right)\int_{B_{1}(0)}\chi_{\{\t\psi>0\}} dxdy+\e\int_{B_{1}(0)}|\g\t\psi|^2 dxdy\right\}.\ea\ee Taking $\e=\f{1}{C\delta\Ld}$, it follows from \eqref{b37} and \eqref{b38} that $$\ba{rl}\left(1-C\delta^2\right)x_0\Ld^2\int_{B_{1}(0)}\chi_{\{\t\psi>0\}}dxdy\leq 0,\ea$$ which implies that  $$\t\psi=0\ \ \text{in} \ \ \ B_1(0),$$ provided that $\delta<\sqrt{\f{1}{C}}$, where the constant $C$ depends on $\vartheta$ and $b$. The proof is completed.
\end{proof}

\begin{thm}\label{lb6} The minimizer $\psi_{\ld,\mu}$ is Lipschitz continuous in any compact subset of $\bar\O_\mu$ that does not contain $A$ or the points where $\p\O_\mu$ is not $C^{1,\alpha}$.

\end{thm}

\begin{proof} Denote $\psi=\psi_{\ld,\mu}$ and $\Gamma=\Gamma_{\ld,\mu}$ for simplicity. The Lipschitz continuity of $\psi$
in any compact subset of $\O_\mu$ follows from the proof of Lemma \ref{lb3}. On another hand, the Lipschitz continuity of $\psi$
near $L_\mu\cup N_\mu\cup(N_{0,\mu}\cap\left\{x\leq \f b2\right\})$ can be obtained by using the elliptic estimate. Along the similar arguments in the proof of Lemma \ref{lb2}, we can obtain the Lipschitz continuity of $\psi$ near the symmetric axis $I$.

We next consider the Lipschitz continuity of $\psi$
near $T$ or near the wall $N_0\cap\left\{x\geq \f b2\right\}$.

For $X=(x,y)\in\O_\mu$
with $y-1>\delta$, denote $X_0=(b,y)$, $d(X)=dist(X,\Gamma_{\ld,\mu})$ and
$d_1(X)=dist(X,T)$. If $d(X)\leq d_1(X)$, by using the similar
arguments in the proof Lemma \ref{lb3}, we have
$|\g\psi(X)|\leq C$, where the constant $C$ depends on $\Lambda$ and
$\vartheta$.

For the case $d(X)>d_1(X)=x-b$, set
$r_0=\min\left\{\f b2,y-1\right\}$ and $B_{r_0}=B_{r_0}(X_0)$.

Consider a function $\phi$, which solves the following boundary value problem
$$\left\{\ba{ll} &\g\cdot f_\eta\left(\f{\g\phi}{x};\ld\right)=0  \ \text{in}~~B_{r_0}\cap\{x>b\},\\
&\phi=0\ \ \ \text{on} \ \ B_{r_0}\cap\{x=b\},\ \ \phi=m_0-\psi\ \ \
\text{on} \ \ \p B_{r_0}\cap\{x>b\}. \ea\right. $$ The maximum
principle gives that \be\label{b39}\text{$m_0-\psi\leq\phi$ in $B_{r_0}\cap\{x>b\}$.}\ee Set
$\t\phi(\t X)=\f{\phi(X_0+r_0 \t X)}{r_0}$ with $\t X=(\t x,\t y)$.
Noting $0\leq\phi\leq m_0$,  one has
$$\left\{\ba{ll} \g\cdot f_\eta\left(\f{\g\t\phi}{b+r_0\t x};\ld\right)=0  \ &\text{in}~~B_{1}(0)\cap\{\t x>0\},\\
0\leq\t\phi\leq\f{m_0}{r_0} \ \ \
&\text{on} \ \ \p (B_{1}(0)\cap\{\t x>0\}). \ea\right. $$

Applying the elliptic estimates for $\t\phi$ in $B_1(0)\cap\{\t x>0\}$, one has
$$\t\phi(\t X)\leq C\f{\t x}{r_0}\ \ \text{in $B_{\f12}(0)\cap\{\t
x>0\}$},$$ which gives that \be\label{b40}\phi(X)=\t\phi\left(\f{X-X_0}{r_0}\right)\leq C\f{x-b}{r_0}\ \ \text{in
$B_{\f{r_0}2}(X_0)\cap\{x>b\}$}.\ee

If $r=d_1(x)=x-b<\f{r_0}{4}$, we have $B_r(X)\subset
B_{\f{r_0}2}(X_0)\cap\{x>b\}$. Set $\t\psi(\t X)=\f{m_0-\psi(X+r \t
X)}{r}$ with $\t X=(\t x,\t y)$, it follows from \eqref{b39} and \eqref{b40} that
$$\t\psi(\t X)\leq\f{\phi(X+r \t X)}{r}\leq\f{C(x+r\t x-b)}{rr_0}\leq\f{C}{r_0}\ \ \text{in}\ \
B_1(0).$$ By using the elliptic estimate, one has
$$|\g\psi(X)|=|\g\t\psi(0)|\leq \f{C}{r_0}.$$

If $r=d_1(x)=x-b\geq\f{r_0}{4}$, the elliptic estimate gives the desired uniform bound for $\g\psi(X)$.

Finally, we consider the Lipschitz continuity of $\psi$
near $N_0\cap\{x\geq b\}$. Since $N_0$ is $C^{2,\alpha}$ and $\psi_{\ld,\mu}=0$ on $N_0$, the Harnack's inequality is still
valid up to the boundary $N_0\cap\left\{x\geq \f b2\right\}$. It follows from the similar arguments in the proof of Lemma \ref{lb3} that
$$\f{|\g\psi(X)|}{x}\leq C\Ld\ \ \text{near $N_0\cap\left\{x\geq \f b2\right\}$}.$$

\end{proof}

\section{The free boundary of the minimizer $\psi_{\ld,\mu}$}

In this section, we will show some important properties of the free boundary, such as the continuity of the graph and the continuous fit condition.

\subsection{Uniqueness and monotonicity of the minimizer}

To obtain the continuous fit condition of the free boundary, we construct the uniqueness and the monotonicity of
the minimizer to the truncated variational problem $(P_{\ld,\mu})$.
\begin{lem}\label{lc1}For any $\ld\leq\Pi_\ld-3\t\e$ and $\mu>1$, the minimizer $\psi_{\ld,\mu}$ to the truncated variational problem $(P_{\ld,\mu})$ is unique, and
$\psi_{\ld,\mu}(x,y_1)\geq\psi_{\ld,\mu}(x,y_2)$ for any $y_1> y_2$.
\end{lem}
\begin{proof} Suppose that $\psi_1$ and $\psi_2$ are two minimizers to the truncated variational problem $(P_{\ld,\mu})$. Set
$$\psi_1^{\varepsilon}(x,y)=\psi_1(x,y-\varepsilon)\ \ \text{for any}\ \ \e>0.$$
Notice that $\psi_1^{\varepsilon}(x,y)$ is a minimizer of the
functional $J^\e_{\ld,\mu}$ in $\O_\mu^\e$ with the corresponding
admissible set $K_{\mu}^\e$ as follows
$$\O_\mu^\e=\{(x,y)|~(x,y-\e)\in\O_\mu\}~~\text{and}~~K_{\mu}^\e=\{\psi^\e(x,y-\e)\in K_{\mu}|~(x,y)\in\O_\mu^\e\}.$$

Extend $\psi_2(x,y)=\f{m_0}{x_\mu^2}x^2$ in $\{(x,y)\mid 0<x\leq
x_\mu,\mu<y\leq\mu+\varepsilon\}$ and denote
$$\varphi_1=\min\{\psi_1^{\varepsilon},\psi_2\}\ \ \ \text{and}\ \ \ \varphi_2=\max\{\psi_1^{\varepsilon},\psi_2\}.$$

Obviously, $\varphi_1\in K^\e_{\mu}$ and
$\varphi_2\in K_{\mu}$. For any sufficiently large $R>R_0$, denote
$\O_{\mu,R}=\O_\mu\cap\{y<R\}$ and
$\O_{\mu,R}^\e=\O_\mu^\e\cap\{y<R\}$. Since $\varphi_1=\psi_1^\e$ in $\O_{\mu,R}\setminus\O_{\mu,R}^\e$ and $\varphi_1=\psi_1^\e$ in $\O_{\mu,R}^\e\setminus\O_{\mu,R}$, it is easy to check that
\be\label{c1}\ba{rl}&\int_{\O_{\mu,R}^{\e}}xF\left(\f{|\nabla\varphi_{1}|^2}{x^2};\ld\right)dxdy+\int_{\O_{\mu,R}}xF\left(\f{|\nabla\varphi_2|^2}{x^2};\ld\right)dxdy\\
=&\int_{\O_{\mu,R}^{\e}}xF\left(\f{|\nabla\psi_{1}^\e|^2}{x^2};\ld\right)dxdy+\int_{\O_{\mu,R}}xF\left(\f{|\nabla\psi_2|^2}{x^2};\ld\right)dxdy,\ea\ee and \be\label{c2}\ba{rl}&\int_{E_{R}^{\e}}x\chi_{\{\varphi_1<m_0\}}dxdy+\int_{E_{R}}x\chi_{\{\varphi_2<m_0\}}dxdy\\
=&\int_{E_{R}^{\e}}x\chi_{\{\psi_1^{\varepsilon}<m_0\}}dxdy+\int_{E_{R}}x\chi_{\{\psi_2<m_0\}}dxdy,\ea\ee where $E_R=\O_{\mu,R}\cap\{x>b\}$ and $E_R^\e=\O_{\mu,R}^\e\cap\{x>b\}$.

Integration by parts, one has
\be\label{c3}\ba{rl}&\int_{E_{R}^{\e}}\nabla\varphi_1\cdot e \chi_{\{\varphi_1<m_0\}}dxdy+\int_{E_R}\nabla\varphi_2\cdot e
\chi_{\{\varphi_2<m_0\}} dxdy\\
&-\int_{E_{R}^{\e}}\nabla\psi_1^\e\cdot e \chi_{\{\varphi_1<m_0\}}dxdy+\int_{E_R}\nabla\psi_2\cdot e
\chi_{\{\psi_2<m_0\}} dxdy\\
=&\int_{\p E_{R}^{\e}}(\varphi_1-\psi_1^\e)e\cdot \nu dS+\int_{E_R}(\varphi_2-\psi_2)e\cdot\nu dS\\
=&\int_{\p E_{R}^{\e}\cap \p E_R}(\varphi_1+\varphi_2-\psi_1^\e-\psi_2)e\cdot \nu dS\\
=&0,\ea\ee where we have used the facts $\varphi_2=\psi_2=0$ in $E_R\setminus E_R^\e$.

In view of \eqref{c1} - \eqref{c3}, one has
\be\label{c4}\ba{rl}&\int_{\O_{\mu,R}^{\e}}G(\g\varphi_1,\varphi_1,x;\ld)dxdy+\int_{\O_{\mu,R}}G(\g\varphi_1,\varphi_1,x;\ld)dxdy\\
=&\int_{\O_{\mu,R}^{\e}}G(\g\psi^\e_1,\psi^\e_1,x;\ld)dxdy+\int_{\O_{\mu,R}}G(\g\psi_2,\psi_2,x;\ld)dxdy.\ea\ee
Taking $R\rightarrow+\infty$ in \eqref{c4} yields that
\be\label{c5}\ba{rl}J_{\lambda,\mu}^\e(\psi_1^{\varepsilon})+J_{\lambda,\mu}(\varphi_2)=J_{\lambda,\mu}^\e(\varphi_1)+J_{\lambda,\mu}(\psi_2).\ea\ee
Since $\psi_1^\e$ and $\psi_2$ are minimizers, it follows from \eqref{c5} that
\be\label{c6}J_{\lambda,\mu}^\e(\psi_1^{\varepsilon})=J_{\lambda,\mu}^\e(\varphi_1)\ \ \text{and}\ \ J_{\lambda,\mu}(\psi_2)=J_{\lambda,\mu}(\varphi_2).\ee

Next, we claim that \be\label{c7}\psi_1^\e(x,y)<\psi_2(x,y)\ \
\text{in $D$},\ee where $D$ is the maximal connected component of
$\O_\mu\cap\{\psi_2<m_0\}$, which contains an $\O_\mu$-neighborhood of
$N_{\mu}$.

Suppose not, note that $\psi_1^\e<m_0=\psi_2$ on $N_{\mu}$, then there
exists a dist $B_1$, such that
$$\psi_1^\e<\psi_2\ \ \text{in}\ \ B_1,\ \ \bar
B_1\subset\O_\mu\cap\{\psi_2<m_0\},$$ and $$\psi_1^\e=\psi_2\ \
\text{at some points $X_0\in \p
B_1\cap(\O_\mu\cap\{\psi_2<m_0\})$}.$$ Thanks to Hopf's lemma, one
has
$$\f{\p(\psi_1^\e-\psi_2)}{\p\nu}>0\ \ \text{at $X_0$},$$ where $\nu$
is the outer normal vector of $\p B_1$ at $X_0$. This implies that
the level set $\{\psi_1^\e=\psi_2=\psi_2(X_0)\}$ is smooth curve in a
neighborhood of $X_0$. Then there exists a smooth curve
$\Gamma_0=\{X\mid \psi_1^\e(X)=\psi_2(X)=\psi_2(X_0)\}$ passing
through $X_0$ and a disc $B_2$, such that $$\psi_1^\e>\psi_2\ \
\text{in $B_2$ and $X_0\in \Gamma_0\cap \p B_2\cap\p B_1$}.$$ Hence,
one has
$$\f{\p(\varphi_1-\psi_2)(X)}{\p\nu}=\f{\p(\psi_1^\e-\psi_2)(X)}{\p\nu}\rightarrow\f{\p(\psi_1^\e-\psi_2)(X_0)}{\p\nu}>0\ \ \text{as $X\rightarrow X_0$, $X\in
B_1$},$$ and
$$\f{\p(\varphi_1-\psi_2)(X)}{\p\nu}=\f{\p(\psi_2-\psi_2)(X)}{\p\nu}=0,\
\ \text{$X\in B_2$},$$ which implies that $\varphi_1$ is not $C^1$-smooth
in a neighborhood of $X_0$, due to that $\psi_2$ is smooth at $X_0$.
On the other hand, it follows from \eqref{c6} that $\varphi_1$ is a
minimizer, and $\varphi_1(X_0)<m_0$. By virtue of the elliptic
regularity, we can conclude that $\varphi_1$ is smooth in a
neighborhood of $X_0$. This leads a contradiction. Hence, we
complete the proof of the claim \eqref{c7}.

We next show that \be\label{c8}\text{$\psi_1(x,y)$ is monotone
increasing with respect to $y$ in $\O_\mu$.}\ee Choosing
$\psi_1=\psi_2$ in \eqref{c7}, it implies that
\be\label{c9}\text{$\f{\p\psi_1}{\p y}\geq 0$ in $D$.}\ee

To obtain \eqref{c8}, it suffices to show that
$$\text{$D=\O_\mu\cap\{\psi_1<m_0\}$}.$$

Suppose not, it follows from \eqref{c7} that $D\cap\{x>b\}=\{x>b, y<\phi(x)\}$. As a part of the free
boundary of the graph $\phi$, we
can conclude that $\phi(x)$ is continuous (see the proof of Lemma \ref{lc3} later).  Define $\psi_0=\psi_1$ in
$D$ and $\psi_0=m_0$ in $\O_\mu\cap\{x>b,y\geq\phi(x)\}$, it is easy to check that
$\psi\in K_\mu$. Therefore, we have

$$J_{\ld,\mu}(\psi_0)-J_{\ld,\mu}(\psi_1)=- \int_{\O_\mu\setminus D}
G(\g\psi_1,\psi_1,x;\ld)dxdy<0,$$ which leads a contradiction.

Similarly, we can obtain that
$$\text{$\psi_2(x,y)$ is monotone increasing with respect to $y$ in
$\O_\mu$,}$$ which together with \eqref{c8} gives that
$$\text{$\O_\mu\cap\{\psi_1<m_0\}$ and $\O_\mu\cap\{\psi_2<m_0\}$ are connected}.$$
 In view of \eqref{c7}, one has
$$\psi_1^\e(x,y)\leq\psi_2(x,y)~~\text{in}~~~ \O_\mu.$$ Taking
$\e\rightarrow0$ in above inequality, we have
$$\psi_1\leq\psi_2\ \ \text{in}\ \ \O_\mu.$$

Similarly, we can obtain that $$\psi_1\geq\psi_2\ \ \text{in}\ \
\O_\mu.$$

Hence, $\psi_1=\psi_2$ and the minimizer to the variational problem
$(P_{\ld,\mu})$ is unique.

\end{proof}

\subsection{Fundamental properties of the free boundary}
In this subsection, we show some significant properties of the
free boundary $\G_{\ld,\mu}$ to the truncated variational problem
$(P_{\ld,\mu})$. Thanks to the monotonicity of the minimizer
$\psi_{\ld,\mu}(x,y)$ with respect to $y$, there exists a mapping
$y=k_{\ld,\mu}(x)$, such that
$$E\cap\{\psi_{\ld,\mu}<m_0\}=\{(x,y)\mid b<x<\infty, g_0(x)<y<k_{\ld,\mu}(x)\}.$$

To obtain the continuity of the function $k_{\ld,\mu}(x)$, we need the following non-oscillation lemma and the proof can be found in Lemma 4.4 in \cite{ACF5}.
\begin{lem}\label{lc2}Let G be a domain in $E\cap\{\psi_{\ld,\mu}<m_0\}$, bounded by two disjointed arcs $\gamma_1$, $\gamma_2$ of the free boundary, ${y=\beta_1}$, ${y=\beta_2}$.
Suppose that the arcs $\gamma_i$ ($i=1,2$) lie in
$\{\beta_1<y<\beta_2\}$ with the endpoints $(\a_i,\beta_1)$ and
$(\zeta_i,\beta_2)$. Suppose the distant $d=dist(G,B)>0$, then
$$|\beta_2-\beta_1|\leq C\max\{|\a_1-\a_2|,~~|\zeta_1-\zeta_2|\},$$
where $C$ is a constant depending only on $\Ld, \vartheta$, $d, N, N_0$ and $m_0$.
\end{lem}

\begin{xrem}\label{re5}
The nonoscillation Lemma \ref{lc2} remains true provided that one of the
arcs $\gamma_2$ is a line segment on $T=\{(b,y)\mid y\geq 1\}$, and
$$\f{1}{x}\f{\p\psi_{\ld,\mu}}{\p\nu}\geq\l~~\text{on}~~\gamma_2.$$
\end{xrem}

\begin{lem}\label{lc3} The function $y=k_{\ld,\mu}(x)$ is continuous for $x\in(b,+\infty)$. Moreover, $k_{\ld,\mu}(b)=\lim_{x\rightarrow b^+}k_{\ld,\mu}(x)$ exists and is finite.
\end{lem}\begin{figure}[!h]
\includegraphics[width=100mm]{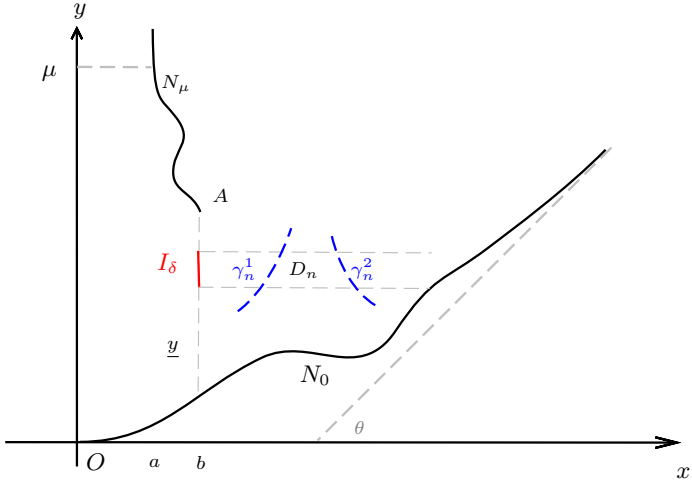}
\caption{The domain $D_n$}\label{fi6}
\end{figure}
\begin{proof}
We first consider the existence of the limit $\lim_{x\rightarrow b^+}k_{\ld,\mu}(x)$.

 Suppose not, one has that $\liminf_{x\rightarrow b^+}k_{\ld,\mu}(x)<\limsup_{x\rightarrow b^+}k_{\ld,\mu}(x)$, then we consider the following two cases for $\underline y=\lim_{x\rightarrow b^+}k_{\ld,\mu}(x)$.

{\bf Case 1.} $\underline y<1$. Denote $\delta=\f{1-\underline y}{4}$ and $I_\delta=\left\{(1,y)\mid 1-\f52\delta\leq y\leq 1-\f32\delta\right\}$. Then there exist two sequences $\{x_n\}$ and $\{\t x_n\}$ with $x_n\downarrow b$ and $\t x_n\downarrow b$, such that \be\label{c10}k_{\ld,\mu}(x_n)\rightarrow 1-\delta\ \ \text{and}\ \ k_{\ld,\mu}(x_n)\rightarrow 1-3\delta,\ \ \delta>0, \ee and \be\label{c11}\psi_{\ld,\mu}(x_n,y)=m_0,\ \ \psi_{\ld,\mu}(\t x_n,y)<m_0 \ \ \text{for}\ \ |y-1+2\delta|\leq\f\delta2,\ \  x_{n+1}<\t x_n<x_n.\ee
By virtue of Lemma \ref{lb3} and the monotonicity of $\psi_{\ld,\mu}$, we have that $\psi_{\ld,\mu}$ is Lipschitz continuous in an $\bar E$-neighborhood of $I_\delta$ and $\psi_{\ld,\mu}=m_0$ on $I_\delta$.

It follows from \eqref{c11} that there exists a domain $D_n\subset
E\cap\{\psi_{\ld,\mu}<m_0\}$ (please see Figure \ref{fi6}), which is
bounded by the arcs $y_1=1-\f32\delta$, $y_2=1-\f52\delta$,
$\gamma_n^1$ and $\gamma_n^2$. Here, $\gamma_n^1$ and $\gamma_n^2$
are parts of free boundary $\Gamma_{\ld,\mu}\cap\{x\leq x_{n-1}\}$,
and the curve $\gamma_n^1$ lies the right of the curve $\gamma_n^2$.
Denote $h_n=\text{dist}(\gamma_n^1,\gamma_n^2)$. By virtue of that
$x_n\rightarrow b^+$, one has \be\label{c12}h_n\rightarrow 0\ \
\text{as}\ \ n\rightarrow+\infty.\ee

Thanks to the non-oscillation Lemma \ref{lc2}, we have $$0<\delta\leq Ch_n,$$ which contradicts to \eqref{c12}, provided that $n$ is sufficiently large.

{\bf Case 2.} $\underline y\geq 1$. Take a constant $\delta>0$, such that $\delta\leq\f{\limsup_{x\rightarrow b^+}k_{\ld,\mu}(x)-\underline y}{4}$. Denote $\gamma_\delta=\left\{(b,y)\mid \underline y+\f32\delta\leq y\leq \underline y+\f52\delta\right\}$. In an $E$-neighborhood of $\gamma_\delta$, we can obtain a contradiction by using the non-oscillation Lemma \ref{lc2}.

Similarly, we can show that the limits $\lim_{x\rightarrow x_0^+}k_{\ld,\mu}(x)$ and $\lim_{x\rightarrow x_0^-}k_{\ld,\mu}(x)$ exist for any $x_0\in(b,+\infty)$.

{\bf Step 2.} $\lim_{x\rightarrow x_0^+}k_{\ld,\mu}(x)=\lim_{x\rightarrow x_0^-}k_{\ld,\mu}(x)$ for any $x_0\in(b,+\infty)$.

Suppose that there exists a $x_0\in(b,+\infty)$, such that $\lim_{x\rightarrow x_0^+}k_{\ld,\mu}(x)\neq\lim_{x\rightarrow x_0^-}k_{\ld,\mu}(x)$. Without loss of generality, we assume that $\lim_{x\rightarrow x_0^+}k_{\ld,\mu}(x)>\lim_{x\rightarrow x_0^-}k_{\ld,\mu}(x)$. Denote $\gamma=\{(x_0,y)\mid y_1\leq y\leq y_2\}$ with $\lim_{x\rightarrow x_0^-}k_{\ld,\mu}(x)<y_1<y_2<\lim_{x\rightarrow x_0^-}k_{\ld,\mu}(x)$. The monotonicity and Lipschitz continuity of $\psi_{\ld,\mu}$ give that $\gamma$ is the free boundary of $\psi_{\ld,\mu}$ and $\psi_{\ld,\mu}<m_0$ in $E_\e=\{(x,y)\mid x_0<x<x_0+\e,y_1<y<y_2\}$ for small $\e>0$. Then we have $$\t Q_\ld\psi_{\ld,\mu}=0\ \ \text{in}\ \ E_\e,\ \ \psi_{\ld,\mu}=m_0\ \ \text{and}\ \ \f1x\f{\psi_{\ld,\mu}}{\p x}=-\ld\ \ \text{on}\ \ \gamma. $$ Since $\psi_{\ld,\mu}$ is analytic in $E_\e$ for small $\e>0$, thanks to Cauchy-Kovalevskaya theorem, one has
$$\psi_{\ld,\mu}(x,y)=-\ld(x^2-x^2_0)+m_0\ \ \text{in}\ \ \O_\mu\cap\{x_0<x<x_0+\e\},$$ which contradicts to $\psi_{\ld,\mu}=0$ on $N_0$.

{\bf Step 3.} $k_{\ld,\mu}(x)<+\infty$ for any $x\in[b,+\infty)$. We first show that
\be\label{c13}\text{the free boundary $\Gamma_{\ld,\mu}$ is non-empty in $E$}.\ee Suppose that $\Gamma_{\ld,\mu}$ is empty, it implies that
\be\label{c14}\psi_{\ld,\mu}<m_0\ \ \text{in}\ \ E.\ee For any $R>0$, there exists a disc $B_R(X_0)\subset E$ with $X_0=(x_0,y_0)$, such that
$$\f{1}{R}\left(\fint_{B_{R}(X_0)}(m_0-\psi_{\ld,\mu})^2dxdy\right)^{\f{1}{2}}\leq \f{m_0}{R}\leq c^*\Lambda x_0,$$for sufficiently large $R$. It follows from non-degeneracy Lemma \ref{lb5} that
$\psi_{\ld,\mu}=m_0$ in $B_{\f R8}(X_0)$. This contradicts to \eqref{c14}.
With the aid of \eqref{c13}, we can take a maximal interval $(\gamma_1,\gamma_2)\subset (b,+\infty)$, such that $$k_{\ld,\mu}(x)\ \ \text{is finite in $(\gamma_1,\gamma_2)$ and $k_{\ld,\mu}(\gamma_1+0)=k_{\ld,\mu}(\gamma_2-0)=+\infty$}.$$ We first show that \be\label{c15}\gamma_2=+\infty.\ee If not, then $\gamma_2<+\infty$. By using the proof of \eqref{c13}, we can conclude that $\Gamma_{\ld,\mu}$ is non-empty in $E\cap\{x>\gamma_2\}$, and there exists a $\gamma_3\in[\gamma_2,+\infty)$, such that
$$\Gamma_{\ld,\mu}\cap\{\gamma_2<x<\gamma_3\}=\varnothing\ \ \text{and $k_{\ld,\mu}(x)$ is finite in $(\gamma_3,\gamma_3+\e)$ for small $\e>0$}. $$
Denote $D_R=\left\{(x,y)\mid \f{\gamma_1+\gamma_2}{2}<x<\gamma_3+\f\e2, R<y<2R\right\}$ for large $R$, applying the non-oscillation Lemma \ref{lc2} for $\psi_{\ld,\mu}$ in $D_R$, one has
$$R\leq C\left|\f{\gamma_1+\gamma_2}{2}-\gamma_3-\f\e2\right|,$$ where the constant $C$ is independent of $R$. This leads a contradiction for sufficiently large $R>0$.

Next, we will show that $$\gamma_1=b.$$ If not, then $\gamma_1>b$ and we consider the following two cases.

{\bf Case 1.} $k_{\ld,\mu}(x)=\infty$ for any $x\in(b,\gamma_1)$. It follows from Lemma \ref{lb2} that $\f1x\f{\p\psi_{\ld,\mu}}{\p\nu}\geq\ld$ on $T_R=\{(b,y)\mid R<y<2R\}$ for large $R>0$. Denote $D_R=\left\{(x,y)\mid b<x<\f{\gamma_1+\gamma_2}{2}, R<y<2R\right\}$, by using the non-oscillation Lemma \ref{lc2} and Remark \ref{re5} for $\psi_{\ld,\mu}$ in $D_R$, one has
$$R\leq C\f{\gamma_1+\gamma_2-2b}{2},$$ which leads a contradiction for sufficiently large $R>0$.

{\bf Case 2.} There exists a $\gamma_0\in(b,\gamma_1]$, such that
$$\Gamma_{\ld,\mu}\cap\{\gamma_0<x<\gamma_1\}=\varnothing\ \ \text{and $k_{\ld,\mu}(x)$ is finite in $(\gamma_0-\e,\gamma_0)$ for small $\e>0$}. $$
Using the non-oscillation Lemma \ref{lc2} for $\psi_{\ld,\mu}$ in $D_R$ leads a contradiction for sufficiently large $R>0$, where $D_R=\left\{(x,y)\mid \gamma_0-\f\e2<x<\f{\gamma_1+\gamma_2}{2}, R<y<2R\right\}$.

Finally, we can show that $k_{\ld,\mu}(b)<+\infty$ by using the non-oscillation Lemma \ref{lc2} and Remark \ref{re5}.

\end{proof}

In the following, we will show some
important properties, such as, the optimal decay rate of the free
boundary, the convergence rate and the asymptotic behavior of the
subsonic impinging jet in downstream.

\begin{lem}\label{lc4}
The minimizer $\psi_{\ld,\mu}$ and the free boundary $y=k_{\ld,\mu}(x)$
satisfy that\\
(1)   for any sufficiently large $x_0>b$, there exists a constant C (independent of $x_0$) such that
 \be\label{c16}\int_{\O_\mu\cap\{x\geq x_0\}}G(\g\psi,\psi,x;\ld)dxdy\leq \f{C}{x_0^3},\ee where
 $$G(\g\psi,\psi,x;\ld)=xF\left(\left|\f{\nabla\psi}{x}\right|^2;\ld\right)+\left(x\Ld^2-2\l F_1(\ld^2;\ld)\nabla\psi\cdot
e\right)\chi_{\{\psi<m_0\}\cap E}.$$ \\
(2) in the downstream,
 \be\label{c17}x(k_{\ld,\mu}(x)-g_0(x))\rightarrow\f{m_0}{\l\cos\theta}~~ \text{as}~~x\rightarrow+\infty,\ee
 and
  \be\label{c18}\f{\nabla\psi_{\ld,\mu}(x,y)}{x}\rightarrow(-\ld\sin\th,\ld\cos\th)~~ \text{for $(x,y)\in\O_\mu\cap\{\psi_{\ld,\mu}<m_0\}$, as $x\rightarrow+\infty$.}\ee
\end{lem}
\begin{proof}
(1) By using the inequality \eqref{b16}, we have
\be\label{c19}\int_{\O_\mu\cap\{x\geq x_0\}}G(\g\psi,\psi,x;\ld)dxdy\leq C\int_{\O_\mu\cap\{x\geq
x_0\}}x\left|\f{\nabla\psi_{\ld,\mu}}{x}-\l
e\chi_{\{\psi<m_0\}}\right|^2dxdy,\ee for $x_0>b$.

It follows from the proof of Proposition 4.4 in \cite{CD} that
$$\int_{\O_\mu\cap\{x\geq
x_0\}}x\left|\f{\nabla\psi_{\ld,\mu}}{x}-\l
e\chi_{\{\psi<m_0\}}\right|^2dxdy\leq \f{C}{x_0^3},$$ for sufficiently large $x_0>b$, where $C$ is a
constant independent of $x_0$. This together with \eqref{c19} gives \eqref{c16}.

(2) For a sequence $\{x_n\}$ with $x_n\rightarrow+\infty$, set
$\psi_n(\t X)=\psi_{\ld,\mu}\left(X_n+\f{\t X}{x_n}\right)$ with $X_n=(x_n,g_0(x_n))$ and $\t X=(\t x,\t y)$.
Obviously,
$\nabla\psi_n(\t X)=\f1{x_n}\nabla\psi_{\ld,\mu}\left(X_n+\f{\t X}{x_n}\right)$.
For any $R>0$, thanks to \eqref{c16}, we have
$$\ba{rl}I_n=&\int_{\{|\t x|<Rx_n\}}\left(1+\f{\t x}{x^2_n}\right)\left|\f{\nabla\psi_n}{1+\f{\t x}{x_n^2}}-\ld e\chi_{\{\psi_n<m_0\}}\right|^2d\t xd\t y\\
=&x_n\int_{\O_\mu\cap\{x_n-R|<x<x_n+R\}}x\left|\f{\nabla\psi_{\ld,\mu}}{x}-\ld e\chi_{\{\psi_{\ld,\mu}<m_0\}}\right|^2dxdy\\
\leq&\f{Cx_n}{(x_n-R)^3}\rightarrow 0\ \ \text{as
$x_n\rightarrow+\infty$}.\ea$$ 

Without loss of generality, we may assume that \be\label{c20}\text{$\psi_{n}\rightarrow
\psi_0$ weakly in $H_{loc}^1(\mathbb{R}^2)$ and $a.e.$ in $\mathbb{R}^2$,}\ee and $$\chi_{\{\psi_n<m_0\}}\rightarrow\gamma\ \ \text{weakly star in $L_{loc}^{\infty}(\mathbb{R}^2)$ and $\chi_{\{\psi_0<m_0\}}\leq\gamma\leq1$}.$$ Furthermore, one has
$$\ba{rl}\int_{\mathbb{R}^2}\left|\nabla\psi_0-\ld
e\chi_{\{\psi_0<m_0\}}\right|^2d\t xd\t y\leq\int_{\mathbb{R}^2}\left|\nabla\psi_0-\ld
e\gamma\right|^2d\bar{x}d\bar{y}\leq
\liminf_{n\rightarrow+\infty} I_n=0,\ea$$ which gives that
\be\label{c21}\nabla\psi_0=\l \chi_{\{\psi_0<m_0\}} e\ \ \ \text{a.e. in $\mathbb{R}^2$ and}~~\psi_{0}(0)=0.\ee


Denote $\o(s,t)=\psi_0(\t x,\t y)$ with $s=\t x\cos\th+\t y\sin\th$ and $t=\t y\cos\th-\t x\sin\th$, one has
$$\f{\p\o}{\p s}=\f{\p\psi_0}{\p \t x}\cos\th+\f{\p\psi_0}{\p\t y}\sin\th\ \ \text{and}\ \ \f{\p\o}{\p t}=-\f{\p\psi_0}{\p \t x}\sin\th+\f{\p\psi_0}{\p\t y}\cos\th.$$ By virtue of \eqref{c21}, one has
\be\label{c22}\f{\p\o(s,t)}{\p s}=0\ \ \text{and}\ \ \f{\p\o(s,t)}{\p t}=\ld \chi_{\{\o<m_0\}}\ \ \text{a.e. in $\mathbb{R}^2$},\ee which imply that $\o(s,t)$ is only a function of $t$. Moreover, $\o(s,t)$ is monotone increasing with respect to $t$. In view of $\o(0,0)=0$ and $0\leq\o\leq m_0$, it follows from \eqref{c22} that $$\psi_0(\t x,\t y)=\o(t)=\min\{\max\{\ld t,0\},m_0\}=\min\{\max\{\ld (\t y\cos\th-\t x\sin\th),0\},m_0\}\ \text{ in $\mathbb{R}^2$}.$$

To obtain the asymptotic behavior of the free boundary $\Gamma_{\ld,\mu}$, we first show that \be\label{c23}\p\{\psi_n<m_0\} \ \ \text{converges to}\ \  \p\{\psi_0<m_0\} \ \text{locally in Hausdorff metric,} \ \text{as} \ \ n\rightarrow+\infty. \ee
Definition of Hausdorff distance $d(D,F)$
between two sets $D$ and $F$ is as follows
$$d(D,F)=\inf\left\{\e>0\mid D\subset\bigcup_{X\in F} B_\e(X) \ \ \text{and}\ \ F\subset\bigcup_{X\in D} B_\e(X)\right\}.$$

For any $X_0=(x_0,y_0)\notin \p\{\psi_0<m_0\}$, the continuity of $\psi_0$ gives that $\psi_0(X_0)<m_0$ or $\psi_0(X_0)=m_0$.
If $\psi_0<m_0$ in $B_r(X_0)$, it follows from \eqref{c20} that
$$\lim_{n\rightarrow+\infty}\fint_{\p B_r(X_0)}(m_0-\psi_n)dS>0,$$ which implies that $$\f1r\fint_{\p B_{\f r{x_n}}\left(X_n+\f{X_0}{x_n}\right)}(m_0-\psi_{\ld,\mu})dS=\f1r\fint_{\p B_r(X_0)}(m_0-\psi_n)dS>C\Ld,$$ for sufficiently large $n$ and small $r>0$. Thanks to Lemma \ref{lb4}, one has $$\psi_{\ld,\mu}<m_0\ \ \text{in $B_{\f r{x_n}}\left(X_n+\f{X_0}{x_n}\right)$, namely, $\psi_n<m_0$ in $B_r(X_0)$},$$ for sufficiently large $n$ and small $r>0$.

 If $\psi_0=m_0$ in $B_r(X_0)$, it follows from \eqref{c20} that for a.e $r>0$,
$$\lim_{n\rightarrow+\infty}\f2r\left(\fint_{B_{\f r2}(X_0)}(m_0-\psi_n)^2dxdy\right)^{\f12}=0,$$ which together with Lemma \ref{lb5} gives that $\psi_n=m_0 \ \text{in} \ B_{\f r{8}}(X_0)$ for
sufficiently large $n$.

 Hence, we
have the convergence of the free boundary in the Hausdorff distance.

For any $R>0$ and small $\e>0$, it follows from \eqref{c23} that
there exists a large $N_{\e,R}$, such that the free boundary
$B_R(0)\cap\p\{\psi_n<m_0\}$ and $B_R(0)\cap\p\{\psi_0<m_0\}$  lie
each within an $\e$-neighborhood of one another, provided that
$n>N_{\e,R}$. Thus we can check that the free boundary
$B_R(0)\cap\p\{\psi_n<m_0\}$ satisfies the flatness condition (see
Section 5 in \cite{ACF4}), it follows from Theorem 6.3 in
\cite{ACF4} that
$$B_R(0)\cap\p\{\psi_n<m_0\}\rightarrow B_R(0)\cap\p\{\psi_0<m_0\}\ \ \text{in}\ \ C^{1,\alpha},$$ which implies that
$$k_{\ld,\mu}'(x_n+\t x)\rightarrow\tan\th\ \ \text{as}\ \ x_n\rightarrow+\infty.$$ Since $\t y=x_n\left(k_{\ld,\mu}\left(x_n+\f{\t x}{x_n}\right)-g_0(x_n)\right)$ is the free boundary of $\psi_n$, we have  $$\f{d\t y}{d\t x}=k'_{\ld,\mu}\left(x_n+\f{\t x}{x_n}\right)\rightarrow \tan\th\ \ \text{and}\ \ x_n\left(k_{\ld,\mu}\left(x_n+\f{\t x}{x_n}\right)-g_0(x_n)\right)\rightarrow \t x\tan\th+\f{m_0}{\ld\cos\th},$$ which imply that
$$k'_{\ld,\mu}\left(x_n\right)\rightarrow\tan\th\ \ \text{and}\ \ x_n\left(k_{\ld,\mu}\left(x_n\right)-g_0(x_n)\right)\rightarrow \f{m_0}{\ld\cos\th}.$$

The compactness of $\psi_n$ gives that
$$\g\psi_n\rightarrow \g\psi_0=\ld(-\cos\th,\sin\th)\ \ \text{uniformly in any compact subset of $S$},$$ where $S=\left\{(\t x,\t y)\mid0<\t y\cos\th-\t x\sin\th<\f{m_0}{\ld}\right\}$, and thus \eqref{c18} holds.



\end{proof}

Finally, we will give the gradient estimate of $\psi_{\ld,\mu}$ near the initial point
of $\Gamma_{\ld,\mu}$ in the following.

\begin{lem}\label{lc5} Let $P=(b,k_{\ld,L}(b))$, then
$$|\g\psi_{\ld,\mu}(X)|\leq C\ \ \text{ in $B_r(P)$},$$ for some $r>0$,
where $C$ is constant depending only on $\Lambda, b$ and $\vartheta$,
but not on $m_0$.

\end{lem}

\begin{proof} Without loss of generality, we assume that $P=A$. For any small $\gamma>0$, it suffices to show that
$$|\g\psi_{\ld,\mu}(X)|\leq C\ \ \text{in}\ \
\O\cap\{\gamma<r<2\gamma\},$$ where $r=|X-A|$. Denote
$$\psi_\gamma(X)=\f{m_0-\psi_{\ld,\mu}(A+\gamma X)}{\gamma}\ \ \text{in}\ \ D_1=\left\{X\mid \f12<|X|<\f52, A+\gamma
X\in\O\right\},$$ and
$$D_2=\left\{X\mid 1<|X|<2, A+\gamma
X\in\O\right\}.$$ Then we have $$\g\cdot f_\eta\left(\f{\g\psi_\gamma}{b+\gamma x};\ld\right)=0  \ \text{in}~~ D_1\cap\{\psi_\gamma>0\}.$$
It follows from Lemma \ref{lb1} that
$$\psi_\gamma=0,\ \ \f{1}{b+\gamma x}\f{\p\psi_\gamma}{\p\nu}=\ld\ \ \text{on the free boundary of $\psi_\gamma$, and $0\leq\psi_\gamma\leq\f{m_0}\gamma$}.$$

Since $\p D_1\cap\{X\mid A+\gamma X\in\p\O_\mu\}$ is
$C^{2,\alpha}$-smooth and $\psi_\gamma=0$ on $\p D_1\cap\{X\mid
A+\gamma X\in\p\O_\mu\}$, the Harnack's inequality is still valid up
to this part of the boundary. By using the similar arguments in the
proof of Lemma \ref{lb3}, we have
$$|\g\psi_{\ld,\mu}(A+\gamma X)|=|\g\psi_\gamma(X)|\leq C\ \ \text{in}\ \ D_2,$$ where $C$ is a
constant depending on $\Lambda, b$ and $\vartheta$, but not on
$\f{m_0}\gamma$. Therefore, we obtain the assertion of this lemma.

\end{proof}

\subsection{ Continuous dependence of $\psi_{\ld,\mu}$ and $\Gamma_{\ld,\mu}$ with respect to $\ld$}

To obtain the continuous fit condition, we will show that the minimizer $\psi_{\ld,\mu}$ and the free boundary $\Gamma_{\ld,\mu}$ are continuous dependence with respect to the parameter $\ld\leq\Pi_\ld-4\t\e$.

\begin{lem}\label{lc6}Let $\psi_{\ld_n,\mu}$ be a minimizer to the variational problem $(P_{\ld_n,\mu})$
with the admissible set $K_\mu$ and $\ld_n\leq\Pi_{\ld_n}-4\t\e$, then we have
$$\psi_{\ld_n,\mu}\rightharpoonup \psi_{\ld,\mu}\ \ \text{in}\ \ H^1_{loc}(\O_{\mu}) \ \text{and}
\ \psi_{\ld_n,\mu}\rightarrow \psi_{\ld,\mu}\ \ \text{a.e. in}\ \
\O_{\mu},$$ as $\lambda_n\rightarrow\ld\leq\Pi_\ld-4\t\e$, where
$\psi_{\ld,\mu}$ is the minimizer to the variational problem
$(P_{\ld,\mu})$.

\end{lem}

\begin{proof} Denote $\psi_n=\psi_{\ld_n,\mu}$ for simplicity.
By virtue of Proposition \ref{lb1}, we have
$|\g\psi_n|\leq C$ in any compact subset of $\O_\mu$,
where the constant $C$ is independent of $n$. Then there exists a
subsequence $\{\psi_n\}$ such that
\be\label{c24}\psi_n\rightharpoonup \o\ \ \text{weakly
in}\ \ H_{loc}^1(\O_{\mu})\ \text{and}\ \psi_n\rightarrow
\o \ \text{in} \ C_{loc}^\alpha(\O_{\mu})\ \ \text{for all}\ \
0<\alpha<1,\ee and $$\g\psi_n\rightarrow \g\o\
\ \text{weakly star in}\ \ L_{loc}^\infty(\O_{\mu}).$$

{\bf Step 1.} $E\cap\p \{\psi_n<m_0\}\rightarrow E\cap\p \{\o<m_0\}$
locally in the Hausdorff distance in $\O_{\mu}$.

For any $X_0=(x_0,y_0)\notin E\cap\p\{\o<m_0\}$, the continuity of $\o$ gives that there exists a small $r>0$, such that $B_r(X_0)\cap\p\{\o<m_0\}=\varnothing$. We next claim that \be\label{c25}\text{$B_{\f
r{16}}(X_0)\cap\p\{\psi_n<m_0\}=\varnothing$ for sufficiently large
$n$.}\ee
If $\o<m_0$ in $B_r(X_0)$, \eqref{c24} implies that the claim \eqref{c25} is valid. If $\o=m_0$ in $B_r(X_0)$, for any small $\e>0$,  it follows from \eqref{c24} that there exists a $N_\e$ such that
$$\text{$|m_0-\psi_n(X)|<\e$ in $B_{\f r2}(X_0)$ for
$n>N_\e$},$$ which implies that
$$\f2r\left(\fint_{B_{\f r2}(X_0)}(m_0-\psi_n)^2dxdy\right)^{\f12}<\f{2\e}r\leq\lambda c^*\Ld x_0 \ \text{for $n>N_\e$}.$$ Thanks to Lemma \ref{lb5} for $\psi_n$, one has $\psi_n=m_0 \ \text{in} \ B_{\f r{16}}(X_0)$ for
sufficiently large $n$, and the claim \eqref{c25} holds.

On the other hand, for any $X_0=(x_0,y_0)\notin E\cap\p\{\psi_n<m_0\}$. Then $B_r(X_0)\cap\p\{\psi_n<m_0\}=\varnothing$ for small $r>0$, and we claim that \be\label{c26}B_{\f r2}(X_0)\cap\p\{\o<m_0\}=\varnothing.\ee

If $\psi_n<m_0$ in $B_r(X_0)$ for a subsequence
$\{\psi_n\}$, one has
$$\t Q_{\ld_n}\psi_n=0 \ \ \text{in}\ \ B_r(X_0),$$ which implies that
$$\t Q_\ld\o=0  \ \text{in} \
B_r(X_0),\ \o\leq m_0\ \text{in} \ B_r(X_0).$$

The strong maximum
principle yields that $\o=m_0$ or $\o<m_0$ in $B_r(X_0)$. Thus, the claim \eqref{c26} holds.

It is easy to check that the claim \eqref{c26} is valid, if $\psi_n=m_0$ in $B_r(X_0)$ for a subsequence
$\{\psi_n\}$.

 Hence, we
complete the proof of the convergence of the free boundary in the Hausdorff distance.

{\bf Step 2.} $\chi_{\{\psi_n<m_0\}\cap E}\rightarrow
\chi_{\{\o<m_0\}\cap E}$ locally in $L^1(\O_{\mu})$.

In view of \eqref{c24}, we can deduce that Lemma \ref{lb4} and Lemma \ref{lb5}
still hold for $\o$ by taking the limit $n\rightarrow\infty$.
By using Theorem 2.8 in \cite{ACF4},  one has
$$\mathcal{L}^2(E\cap\p\{\o<m_0\}\cap B_R)=0,  \ \text{for any $R>0$},$$ where $\mathcal{L}^2$ is the
Lebesgue measure in $\mathbb{R}^2$.

Applying the results in Step 1, there exists a sequence $\{\e_n\}$ with $\e_n\rightarrow0$, such that $$E\cap\{\psi_n<m_0\}\subset O_{\e_n},$$where $O_{\e_n}$ be an $\e_n$-neighborhood of $E\cap\p\{\o<m_0\}$. Then we have
$$\int_{\O_{\mu}\cap B_R}\left|\chi_{\{\psi_n<m_0\}\cap
E}-\chi_{\{\o<m_0\}\cap E}\right|dxdy\leq\int_{\O_{\mu}\cap
B_R\cap O_{\e_n}}dxdy\rightarrow 0,$$ for any $R>0$.

 {\bf Step 3.} $\nabla\psi_n\rightarrow \nabla\o$ a.e. locally in $\O_{\mu}$.

 Let $D$ be any compact subset of $\O_{\mu}\cap\{\o<m_0\}$. Then one has $$\t Q_{\ld_n}\psi_n=0\ \ \text{in}\ \ D\ \ \text{for sufficiently large}\ n.$$

The elliptic estimates for $\psi_n$ gives that
\be\label{c27}\nabla\psi_n\rightarrow\nabla\o \ \ \text{in}\ \ D.\ee

 Next, we claim that
\be\label{c28}\nabla\psi_n\rightarrow\nabla\o\ \ \text{a.e. in
$\O_\mu\cap\{\o=m_0\}$}.\ee

Since $\{\o=m_0\}$ is $\mathcal{L}^2$-measurable, it follows from
Corollary 3 in \cite{EV} that
$$\lim_{r\rightarrow 0}\f{\mathcal{L}^2(B_r(X)\cap\{\o=m_0\})}{\mathcal{L}^2(B_r(X))}=1\ \ \text{for $\mathcal{L}^2$ a.e. $X\in\{\o=m_0\}$}.$$
Denote $$\mathcal{S}=\left\{X\in\{\o=m_0\}\mid \lim_{r\rightarrow
0}\f{\mathcal{L}^2(B_r(X)\cap\{\o=m_0\})}{\mathcal{L}^2(B_r(X))}=1\right\}.$$
We next show that \be\label{c29}\text{$m_0-\o(X_0+X)=o(|X|)$ for
any $X_0\in\mathcal{S}$}.\ee In fact, if $m_0-\o(Y)>kr$ for some
$Y\in B_r(X_0)$ with $r\rightarrow 0$ and $k>0$. The Lipschitz continuity of $\o$ gives that
$$m_0-\o(X)>\f k2 r\ \ \text{in $B_{\e kr}(Y)$ for some small $\e>0$},$$
which implies that $\{\o<m_0\}$ has positive density at $X_0$, and then it
contradicts to $X_0\in\mathcal{S}$.

With the aid of \eqref{c24} and \eqref{c29}, for any $\e>0$, we
have
$$\f{m_0-\psi_n}{r}<\e \ \ \text{in $B_r(X_0)$ for small $r$},$$
provided that $n$ is sufficiently large, that is $n>N(\e,r)$. It
follows from the non-degeneracy Lemma \ref{lb5} that $\psi_n\equiv
m_0$ in $B_{\f r8}(X_0)$, which implies that $\o\equiv m_0$ in
$B_{\f r{10}}(X_0)$. Then we have that the set $\mathcal{S}$ is open.
Furthermore, one has
$$\text{$\psi_n\equiv \o$ in any compact subset of $\mathcal{S}$, provided that $n$ is sufficiently
large.}$$ This completes the proof of the claim \eqref{c28}.

Since $\mathcal{L}^2(E\cap\p\{\o<m_0\})=0$, it follows from
\eqref{c27} and \eqref{c28} that $\nabla\psi_n\rightarrow
\nabla\o$ a.e. in $\mathbb{R}^2$.

{\bf Step 4.} $\o=\psi$, where $\psi=\psi_{\ld,\mu}$ is the minimizer to the truncated
variational problem $(P_{\ld,\mu})$.

First, we will show that \be\label{c30}J_{R}(\o)\leq
J_{R}(\phi)\ \ \text{for any $\phi\in K_\mu$ and $\phi=\o$ on $\p\O_R$},\ee where
$J_{R}(\phi)=\int_{\O_R}G(\g\phi,\phi,x;\ld)dxdy$ and $\O_R=\O_{\mu}\cap B_R(0)$ for sufficiently large $R>0$.

 For any $\phi\in K_\mu$ and $\phi=\o$ on $\p\O_R$, set
$$\phi_n=\phi+(1-\xi_\delta)(\psi_n-\o),$$ where $\xi_\delta(X)=\min\left\{\f{\text{dist}(X, \mathbb{R}^2\setminus \O_R)}{\delta},1\right\}$. Obviously, $\phi_n=\psi_n$ on $\p\O_R$ and extend $\phi_n=\psi_n$ outside $\O_R$, such that $\phi_n\in K_\mu$. Then one has
$$\int_{\O_R}G(\g\psi_n,\psi_n,x;\ld_n)dxdy
\leq\int_{\O_R}G(\g\phi_n,\phi_n,x;\ld_n)dxdy.$$

 By using the convergence results in
Step 2 and Step 3, taking $n\rightarrow\infty$ in the above inequality, we have
\be\label{c31}\ba{rl}\int_{\O_R}G(\g\o,\o,x;\ld)dxdy
\leq&\int_{\O_R\cap\{\xi_\delta=1\}}G(\g\phi,\phi,x;\ld)dxdy\\
&+\int_{\O_R\cap\{\xi_\delta<1\}}G(\g\phi,\phi,x;\ld)dxdy.\ea\ee
Taking $\delta\rightarrow 0$ in \eqref{c31} yields that   \be\label{c32}\ba{rl}\int_{\O_R}G(\g\o,\o,x;\ld)dxdy
\leq\int_{\O_R}G(\g\phi,\phi,x;\ld)dxdy.\ea\ee  Hence, \eqref{c30} can be obtained by taking $R\rightarrow+\infty$ in \eqref{c32}.

It follows from
\eqref{c16} that
$$\int_{\O_\mu\cap\{x\geq
x_0\}}G(\g\o,\o,x;\ld)dxdy \leq \f{C}{x_0^3},$$ for sufficiently large $x_0>0$, which implies that the results (2) in Lemma
\ref{lc4} still be valid for $\o$.

 For any $\e>0$, Set $\psi_\e(x,y)=\psi(x,y-\e)$ and $\O_{\mu}^\e=\{(x,y)\mid(x, y-\e)\in \O_\mu\}$.
Extend $\o(x,y)=\f{m_0}{x_\mu}x^2$ in $\{(x,y)\mid 0\leq x\leq x_\mu, \mu\leq y\leq \mu+\e\}$. It is easy to check that $\min\{\psi_\e,\o\}\in K_\mu^\e$ and $\max\{\psi_\e,\o\}\in K_\mu$. Therefore, one has
\be\label{c33}J^\e_{\ld,\mu}(\psi_\e)\leq J^\e_{\ld,\mu}(\min\{\psi_\e,\o\}).\ee
Similar to the proof of Theorem 4.1 in \cite{CD}, we can check that  \be\label{c34}J_{\ld,\mu}(\o)+J^\e_{\ld,\mu}(\psi_\e)= J_{\ld,\mu}(\max\{\psi_\e,\o\})+J^\e_{\ld,\mu}(\min\{\psi_\e,\o\}).\ee

With the aid of the asymptotic behaviors of $\psi$ and $\o$ in Lemma \ref{lc4}, we have
$$\psi_\e(x,y)\leq\o(x,y)\ \ \text{in}\ \ \O_\mu\setminus B_R\ \ \text{for any sufficiently large $R>0$}.$$ Thus $\max\{\psi_\e,\o\}=\o$ on $\p\O_R$, it follows from \eqref{c30} that $$J_R(\o)\leq J_R(\max\{\psi_\e,\o\}).$$ Taking $R\rightarrow+\infty$ in the above inequality yields that $$J_{\ld,\mu}(\o)\leq J_{\ld,\mu}(\max\{\psi_\e,\o\}),$$ which together with \eqref{c33} and \eqref{c34} gives that
$$J^\e_{\ld,\mu}(\psi_\e)=J^\e_{\ld,\mu}(\min\{\psi_\e,\o\}).$$ Since the minimizer $\psi_\e$ to the variational problem $(P_{\ld,\mu}^\e)$ is unique, one has
\be\label{c35}\psi(x,y-\e)=\psi_\e(x,y)=\min\{\psi_\e,\o\}\leq \o(x,y)\ \ \text{in}\ \ \O_\mu.\ee Similarly, we can show that \be\label{c36}\psi(x,y+\e)\geq \o(x,y)\ \ \text{in}\ \ \O_\mu.\ee Taking $\e\rightarrow 0$ in \eqref{c35} and \eqref{c36}, we have that $\psi=\o$ in $\O_\mu$.

\end{proof}

Next, we will obtain the continuous dependence of the free
boundary $\Gamma_{\ld,\mu}$ with respect to $\ld$. We remark that
even though the ideas of the proof borrow from the one
for incompressible jet as done in Theorem 3.1 in \cite{ACF3} and Theorem 6.1 in Chapter 3 in \cite{FA1}, due to the difference of the governing equations and the functional, we have to overcome several additional difficulties. Actually, the stream function satisfies the linear elliptic equation for the incompressible flows, and here we have to deal with a quasilinear elliptic equation. 

\begin{lem}\label{lc7}
The free boundary $y=k_{\ld_n,\mu}(x)$ of the minimizer
$\psi_{\ld_n,\mu}$ with $\ld_n\leq\Pi_{\ld_n}-4\t\e$ satisfies that
$$k_{\ld_n,\mu}(x)\rightarrow k_{\ld,\mu}(x) \ \ \text{for any $x\in(b,+\infty)$},$$
as $\lambda_n\rightarrow\ld$, where $y=k_{\ld,\mu}(x)$ is the free
boundary of the minimizer $\psi_{\ld,\mu}$.
\end{lem}
\begin{proof}
For any fixed $x\in(b,+\infty)$, the convergence of
$k_{\ld_n,\mu}(x)$ can be obtained by using the Step 1 in the proof of Lemma \ref{lc6}.

Next, we will consider the initial point of the free boundary and
show that $k_{\ld_n,\mu}(b)\rightarrow k_{\ld,\mu}(b)$ as
$\lambda_n\rightarrow\ld$. Suppose not, then there exists a
subsequence $\{k_{\ld_n,\mu}(b)\}$, such that
$k_{\ld_n,\mu}(b)\rightarrow k_{\ld,\mu}(b)+\beta$ and $\beta\neq
0$. We will show that it is impossible case by case based on the
sign of $\beta$.

{\bf Case 1.} $\beta<0$. The monotonicity of $\psi_{\ld,\mu}(x,y)$
with respect to $y$ gives that $k_{\ld,\mu}(b)+\beta\geq 1$. In
fact, if $k_{\ld,\mu}(b)+\beta<1$, it follows from Lemma \ref{lc6}
that $$\psi_{\ld,\mu}(b,y)=m_0\ \ \ \text{for}\ \ \
k_{\ld,\mu}(b)+\beta\leq y\leq\min\{k_{\ld,\mu}(b),1\},$$ which
contradicts to the fact $\psi_{\ld,\mu}<m_0$ for
$y<\min\{k_{\ld,\mu}(b),1\}$, due to the fact that
$\psi_{\ld,\mu}(x,y)$ is monotone increasing with respect to $y$.

 Denote $$T_\beta=\left\{(b,y)\mid k_{\ld,\mu}(b)+\f{3\beta}4<y<k_{\ld,\mu}(b)+\f\beta4\right\}.$$ Next, we claim that
\be\label{c37}\f{\p\psi_{\ld,\mu}(b+0,y)}{\p x}=-\lambda\ \
\text{on}\ \ T_\beta.\ee

For small $\e>0$, set $$U_\e=\left\{(x,y)\mid b-\e<x<b+\e,k_{\ld,\mu}(b)+\f{3\beta}4<y<k_{\ld,\mu}(b)+\f\beta4\right\}.$$ It is easy to check that
$$T_\beta=U_\e\cap\p\{\psi_{\ld,\mu}<m_0\}.$$
Set $\phi_n=m_0-\psi_{\ld_n,\mu}$ and $\phi=m_0-\psi_{\ld,\mu}$,
then one has
$$\left\{\begin{array}{ll}&\t Q_{\ld_n}\phi_n=0\ \ \text{in}\ U_\e\cap\{\phi_n>0\},\\
&\f1x\f{\p\phi_n}{\p\nu_n}=\lambda_n\ \text{on}\
U_\e\cap\p\{\phi_n>0\},\end{array}\right.$$
 where $\nu_n$ is the inner normal vector to $U_\e\cap\p\{\phi_n>0\}$. Now, in order to show the claim \eqref{c37}, it suffices to check that $\phi$ satisfies the following boundary value problem,
\be\label{c38}\left\{\begin{array}{ll}&\t Q_\ld\phi=0\ \ \text{in}\ U_\e\cap\{\phi>0\},\\
&\f1x\f{\p\phi}{\p \nu}=\lambda\ \text{on}\
U_\e\cap\p\{\phi>0\},\end{array}\right.\ee where $\nu=(1,0)$ is the inner normal vector to
$\p\{\phi>0\}$ at $T_\beta$.

We divide the proof into two steps to show that \eqref{c38} holds.

{\bf Step 1.}  In this step, we will verify that
\be\label{c39}\f1x\f{\p\phi}{\p \nu}\geq\ld\ \ \text{on}\ \
U_\e\cap\p\{\phi>0\}.\ee

It follows from \eqref{c24} that for
$\alpha\in(0,1)$, $$\phi_n\rightarrow\phi \ \ \text{uniformly
in $C^\alpha(U_\e)$},$$ and $U_\e\cap\{\phi_n>0\}\rightarrow
U_\e\cap\{\phi>0\}$ in the Hausdorff metric space, and
$$\g\phi_n\rightarrow\g\phi \ \ \text{weakly in $L^2(U_\e)$},$$ as
$n\rightarrow+\infty$. Moreover, by virtue of the bounded gradient Lemma \ref{lb3} and the Step 4 in the proof of Lemma \ref{lc6}, one has
\be\label{cj40}|\g\phi_n|\leq C\Lambda x\ \ \text{in $U_\e$ and}\ \ \g\phi_n\rightarrow\g\phi \ \ \text{a.e. in $U_\e$, as $n\rightarrow\infty$},\ee where the constant $C$ is independent of $n$.

Since $F_1(t;\ld)=\f{1}{\t\rho(t;\ld^2)}$ is $C^{1,\alpha}$, it follows from \eqref{cj40} that
\be\label{c40}\ba{rl}&\lim_{n\rightarrow+\infty}F_1(\ld_n^2;\ld_n)\ld_n\int_{U_\e\cap\p\{\phi_n>0\}}\xi
dS\\
=&-\lim_{n\rightarrow+\infty}\int_{U_\e\cap\{\phi_n>0\}}\f1x F_1\left(\left|\f{\g\phi_n}{x}\right|^2;\ld_n\right)\g\phi_n\cdot\g\xi
dxdy\\
=&-\int_{U_\e\cap\{\phi>0\}}\f1x F_1\left(\left|\f{\g\phi}{x}\right|^2;\ld\right)\g\phi\cdot\g\xi
dxdy\\
=&\int_{U_\e\cap\p\{\phi>0\}}\f1x F_1\left(\left|\f{\g\phi}{x}\right|^2;\ld\right)\f{\p\phi}{\p \nu}\xi
dS,\ea\ee  for any non-negative $\xi\in C_0^\infty(U_\e)$.

On other hand, it follows from (3.6) in Chapter 3 in \cite{FA1} that  \be\label{c41}\ba{rl}\int_{U_\e\cap\p\{\phi>0\}}\xi
dS\leq\liminf_{n\rightarrow+\infty}\int_{U_\e\cap\p\{\phi_n>0\}}\xi
dS.\ea\ee

In view of \eqref{c40} and \eqref{c41}, one has
\be\label{c42}\ba{rl}F_1(\ld^2;\ld)\ld\int_{U_\e\cap\p\{\phi>0\}}\xi
dS\leq\int_{U_\e\cap\p\{\phi>0\}}\f1x F_1\left(\left|\f{\g\phi}{x}\right|^2;\ld\right)\f{\p\phi}{\p \nu}\xi
dS,\ea\ee for any non-negative $\xi\in C_0^\infty(U)$.
Since $\phi$ is $C^{1,\alpha}$ up to the boundary $\p\{\phi>0\}$, it follows from \eqref{c42} that
\be\label{c43} F_1\left(\left|\f{\g\phi}{x}\right|^2;\ld\right)\left|\f{\g\phi}{x}\right|= F_1\left(\left|\f{\g\phi}{x}\right|^2;\ld\right)\f1x\f{\p\phi}{\p
\nu}\geq  F_1\left(\ld^2;\ld\right)\ld\ \ \text{on}\ \ U_\e\cap\p\{\phi>0\}.\ee Since
$F_1(t^2;\ld)t$ is increasing with respect to $t$,
\eqref{c43} gives that \eqref{c39} holds.

{\bf Step 2. } In this step, we will check that
\be\label{c44}\f1x\f{\p\phi}{\p \nu}\leq\ld\ \ \text{on}\ \
U_\e\cap\p\{\phi>0\}.\ee

By virtue of the non-oscillation Lemma \eqref{lc2} and the
flatness of the free boundary in Section 5 in \cite{ACF4}, we have
that $$\text{the free boundary $\Gamma_{\ld_n,\mu}$ is a $y$-graph in
$U_\e$ for sufficiently large $n$,}$$ where we denote
$\Gamma_{\ld_n,\mu}: x=f_{\ld_n,\mu}(y)$ in the region $U_\e$.
It follows from the result in Theorem 6.3 and Remark 6.4 in
\cite{ACF4} that
 $$\left|f_{\ld_n,\mu}^{(j)}(y)\right|\leq C\ \ \text{for}\ \ k_{\ld,\mu}(b)+\f{3\beta}4<y<k_{\ld,\mu}(b)+\f\beta4,\ \ j=1,2.$$
Then we have
$$\text{$U_\e\cap\p\{\phi_n>0\}\rightarrow U_\e\cap\p\{\phi>0\}$  in
$C^{1,\alpha}$ for some $\alpha\in(0,1)$.}$$

For any fixed $X_0=(x_0,y_0)\in U_\e\cap\p\{\phi>0\}$, it follows from
Lemma \ref{lc6} that there exists a sequence $X_n\in U_\e$ with
$\phi_n(X_n)=0$, such that $X_n\rightarrow X_0$ as
$n\rightarrow+\infty$. Take a small $r>0$ and domain $E_n\subset
B_r(X_0)\subset U_\e$, such that \be\label{c45}B_r(X_0)\cap \p
E_n\rightarrow B_r(X_0)\cap \p\{\phi>0\} \ \text{in}\ C^{1,\gamma},\
\phi_n>0\ \text{in}\ E_n\ \text{and}\ X_n\notin E_n,\ee for
$0<\gamma<\alpha$. We can take a sequence $\{\e_n\}$ with $\e_n\downarrow 0$, such that
$$E_n=B_r(X_0)\cap\{x>b+\e_n\}.$$  Define a function $f_{\delta,n}(y)$ as
follows
$$f_{\delta,n}(y)=b+\e_n-\delta\eta\left(\f{2(y-y_0)}{r}\right),\ \ \delta>0,$$
where $$ \eta(y)=\left\{\begin{array}{ll} e^{-\f{y^2}{1-y^2}}\ \ \
&\text{for}\
\ |y|<1,\\
0 &\text{for}\ \ |y|\geq1.
\end{array}\right.$$ Denote the domain
$E_{\delta,n}=B_r(X_0)\cap\{x>f_{\delta,n}(y)\}$. It is easy to
check that $E_{0,n}=E_n$. By virtue of the definitions of $E_n$ and
$E_{\delta,n}$, there exists the largest number $\delta=\delta_n$,
such that $\phi_n>0$ in $E_{\delta_n,n}$, and $B_r(X_0)\cap\p
E_{\delta_n,n}$ contains a point of the free boundary of $\phi_n$,
which is denoted as $\t X_n=(\t x_n,\t y_n)$ with $\t
x_n=f_{\delta_n,n}(\t y_n)$. Furthermore,
$$\delta_n\rightarrow0 \ \ \text{as}\ \ n\rightarrow+\infty.$$

Let $\o_n$ be the solution of the following Dirichlet problem
$$\left\{\begin{array}{ll}&\t Q_{\ld_n}\o_n=0\ \ \text{in}\ E_{\delta_n,n},\\
&\o_n=0\ \text{on}\ \p E_{\delta_n,n}\cap B_{\f r2}(X_0),\
\o_n=\zeta\phi_n\ \text{on}\ \p E_{\delta_n,n}\cap(
B_r(X_0)\setminus B_{\f r2}(X_0)),\\
&\o_n=\phi_n\ \text{on}\ \p E_{\delta_n,n}\cap \p B_r(X_0),
\end{array}\right.$$ where $\zeta(X)=\min\left\{\max\left\{\f{2|X-X_0|-r}{r},0\right\},1\right\}$.
 Since $\phi_n$ satisfies the quasilinear equation $\t Q_{\ld_n}\phi_n=0$
 in $E_{\delta_n,n}$ and $\o_n\leq\phi_n$
on $\p E_{\delta_n,n}$, the maximum principle implies that
$\o_n\leq\phi_n$ in $E_{\delta_n,n}$. Hence, one has
 \be\label{c46}\ld_n=\f1{\t x_n}\f{\p\phi_n(\t X_n)}{\p\nu_n}\geq\f1{\t x_n}
\f{\p\o_n(\t X_n)}{\p\nu_n},\ee where $\nu_n$ is the inner normal
vector to $\p E_{\delta_n,n}$ at $\t X_n$.

It follows from the fact \eqref{c45} that
$$f_{\delta_n,n}(y)\rightarrow b\ \ \text{in}\ \
C^{1,\gamma}.$$ Thanks to the standard estimates of the solutions of
the elliptic equation of second order, we conclude that $\o_n$ in
$E_{\delta_n,n}\cap B_{\f r2}(X_0)$ converges to $\phi$ in
$\{\phi>0\}\cap B_{\f r2}(X_0)$ in $C^{1,\gamma}$-sense. Suppose that
$$\t X_n\rightarrow \t X=(\t x,\t y)\in\p\{\phi>0\}\cap B_{\f r2}(X_0).$$ This
together with \eqref{c46} gives that $$\f1{\t x_n}\f{\p\o_n(\t
X_n)}{\p\nu}\rightarrow\f1{\t x}\f{\p\phi(\t X)}{\p\nu} \ \ \text{and}\ \
\f1{\t x}\f{\p\phi(\t X)}{\p\nu}\leq \ld.$$ Taking $r\rightarrow 0$, we have $\t
X\rightarrow X_0$ and $$\ld\geq\f1{\t x}\f{\p\phi(\t
X)}{\p\nu}\rightarrow\f1{x_0}\f{\p\phi( X_0)}{\p\nu},\ \ X_0\in
U_\e\cap\p\{\phi>0\},$$ which gives the inequality \eqref{c44} holds.

Therefore, the claim \eqref{c38} follows from \eqref{c39} and \eqref{c44}.

With the aid of the claim \eqref{c37}, $\psi_{\ld,\mu}$ satisfies
$$\t Q_\ld \psi_{\ld,\mu}=0\ \ \text{in}\ E_\e,\ \ \f1x\f{\p\psi_{\ld,\mu}}{\p x}=-\ld\ \ \text{and}\ \ \psi_{\ld,\mu}=m_0\ \ \text{on}\ \p
E_\e\cap\{x=b\},$$ for small $\e>0$, where $E_\e=\left\{(x,y)\mid
b<x<b+\e,
k_{\ld,\mu}(b)+\f{3\beta}4<y<k_{\ld,\mu}(b)+\f\beta4\right\}$.

 It
follows from the Cauchy-Kovalevskaya theorem that 
$$\psi_{\ld,\mu}=-\lambda (x^2-b^2)+m_0\ \text{in}\ \left\{(x,y)\mid
b<x<b+\e,g_0(b)<y<+\infty\right\}\cap \O_{\mu}.
$$
This contradicts to the fact $\psi_{\ld,\mu}=0$ on $N_0$.

{\bf Case 2.} $\beta>0$ and $k_{\ld,\mu}(b)< 1$. By using the
similar arguments in Case 1, we can conclude that
$$\f1b\f{\p\psi_{\ld,\mu}(b-0,y)}{\p x}=\lambda\ \ \text{if}\ \
k_{\ld,\mu}(b)+\f{\beta}4<y<\min\left\{k_{\ld,\mu}(b)+\f{3\beta}4,1\right\}.$$
We can obtain a contradiction by using Cauchy-Kovalevskaya theorem as
in Case 1 again.

{\bf Case 3.} $\beta>0$ and $k_{\ld,\mu}(b)\geq 1$.
 It follows from the arguments in Lemma \ref{lb2} that
$$\f1b\f{\p\psi_{\ld_n,\mu}(b+0,y)}{\p x}\leq-\lambda_n\ \ \text{on}\ \
\left\{x=b,
k_{\ld,\mu}(b)+\f{\beta}4<y<k_{\ld,\mu}(b)+\f{3\beta}4\right\},$$
for sufficiently large $n$. Let $D_n$ be bounded by $x=b$,
$y=k_{\ld_n,\mu}(x)$, $y=k_{\ld,\mu}(b)+\f{\beta}4$ and
$y=k_{\ld,\mu}(b)+\f{3\beta}4$. Furthermore, we have
$$x_n=\min\left\{x\mid
k_{\ld_n,\mu}(x)=k_{\ld,\mu}(b)+\f{3\beta}4\right\}\rightarrow b\ \
\text{as}\ \ n\rightarrow+\infty.$$ Thanks to the non-oscillation
Lemma 4.4 in \cite{ACF5} for $\psi_{\ld_n,\mu}$ in $D_n$, one has
$$\f\beta2\leq C(x_n-b),\ \  \text{the constant $C$ is independent of $n$},$$
which leads a contradiction for sufficiently large $n$.

\end{proof}

\subsection{The continuous fit condition of the free boundary $\Gamma_{\ld,\mu}$}

In the following, we will consider the continuous fit condition of the free boundary at $A$.

\begin{prop}\label{lc8} For any $\mu>1$ and small $m_0>0$, there exists a $\ld_\mu\leq\Pi_{\ld_\mu}-4\t\e$, such that the free boundary $\Gamma_{\ld_\mu,\mu}$ satisfies the continuous fit condition at $A$, namely, $k_{\ld_\mu,\mu}(b)=1$. Moreover,
$$\ld_\mu\leq C_0m_0,\ \ \text{the constant $C_0$ is independent of $\mu$ and $m_0$}.$$
\end{prop}
\begin{proof} {\bf Step 1.} In this step, we will show that $$\text{for $m_0>0$, if $\ld>0$ is sufficiently small, then $k_{\ld,\mu}(b)>1$.}$$  Suppose that there exists a small $\ld_0$, such that $k_{\ld_0,\mu}(b)\leq 1$. Let $S$ be a ring
centered at $P=(b,k_{\ld_0,\mu}(b))$ with some suitable radius
$R_1$ and $R_2$ which are independent of $m_0$ and $R_1<R_2$, such
that $\G_{\ld_0,\mu}\cap S\cap\{x>b\}$ and $N_0\cap \bar S\cap\{x>b\}$ are
nonempty.

It follows from Lemma \ref{lb3} that there exists a constant $C_0$ depending on $N, N_0, \vartheta$
(independent of $m_0$), such that
$$\left|D\psi_{\ld_0,\mu}\right|\leq C_0\Lambda(\ld_0)~~\text{in}~~~S\cap\{x>b\},$$
where $\Lambda(\ld_0)=\sqrt{2F_1(\ld_0^2;\ld_0)\ld_0^2-F(\ld_0^2;\ld_0)}$.

Choosing $X_1\in \G_{\ld_0,\mu}\cap \bar S\cap\{x>b\}$, $X_2\in
N_0\cap \bar S\cap\{x>b\}$ with $|X_1-P|=|X_2-P|$, shows that
\be\label{c47}m_0=\int_\gamma\left|\f{\p\psi_{\ld_0}}{\p
s}\right|dS\leq|D\psi_{\ld,\mu}|2\pi R_2\leq C\Lambda(\ld^2_0),\ee
where $\gamma\subset S$ is disc curve which connects $X_1$ and
$X_2$, $s$ is the unit tangent vector of $\gamma$. On another hand,
it follows from \eqref{b7} that
$$\Ld(\ld_0)\leq\ld_0\sqrt{\f{1}{\rho_0}\left(2-\left(\f{2}{\gamma+1}\right)^{\f1{\gamma-1}}\right)},$$ which together with \eqref{c47} gives that $$0<m_0\leq C\ld_0.$$
This is impossible for sufficiently small $\ld_0$.

{\bf Step 2.} For $\ld=\f{\ld_{cr}}2$, we will show that $$\text{if $m_0>0$ is sufficiently small, then $k_{\ld,\mu}(b)<1$. }$$ Suppose that $k_{\ld,\mu}(b)\geq 1$ for some small $m_0>0$, then there exists a disc $B_R(X_0)\subset\O_\mu$ ($R$ is fixed), such that
$X_0\in E$ and $B_{\f R8}(X_0)\cap\Gamma_{\ld,\mu}\neq\varnothing$.
According to the non-degeneracy Lemma \ref{lb5}, we have
$$\f{m_0}{R}\geq\f{1}{R}\left(\fint_{B_R(X_0)}(m_0-\psi_{\ld,\mu})^2 dxdy\right)^ {\f{1}{2}}\geq c\Lambda(\ld^2),$$ which together with \eqref{b13} gives that
$$m_0\geq cR\Ld(\ld^2)\geq C\ld=\f{C\ld_{cr}}2>0.$$ This leads a contradiction for sufficiently small $m_0$.

{\bf Step 3.} In this step, we will show that there exists a $\ld_\mu\leq\Pi_{\ld_\mu}-4\t\e$, such that $k_{\ld_\mu,\mu}(b)=1$.

For any small $m_0>0$, with the aid of the results in Step 1, we can define a set $$\Sigma_\mu=\{\ld\mid
k_{\ld,\mu}(b)>1\}.$$ Define
\be\label{c48}\ld_\mu=\sup_{\ld\in\Sigma_\mu}\ld.\ee The result in Step 2 gives that \be\label{c49}\ld_\mu\leq C_0m_0,\ee where $C_0$ is a constant depending on $N,N_0$ and $\vartheta$, independent of $m_0$ and $\mu$.
It is easy to check that there exists a $C_0>0$ (independent of $\mu$), such that
$$\ld_\mu\leq C_0m_0\leq \Pi_{\ld_\mu}-4\t\e,$$ for sufficiently small $m_0$. The continuous dependence of $k_{\ld,\mu}(b)$ with respect to $\ld$ gives that
$$k_{\ld_\mu,\mu}(b)=1.$$ If not, the definition of $\ld_\mu$ implies that $k_{\ld_\mu,\mu}(b)>1$.
By using the continuous dependence of $k_{\ld,\mu}(b)$ with respect
to $\ld$ in Lemma \ref{lc7}, there exists a $\ld\in(\ld_\mu,\Pi_\ld-3\t\e)$, such that
$$\text{$\ld-\ld_\mu$ is small and $k_{\ld,\mu}(b)>1$}.$$
Therefore, $\ld\in\Sigma_\mu$, which contradicts to the definition
of $\ld_\mu$ in \eqref{c48}.

\end{proof}

\section{The existence of subsonic solution to the impinging jet flow problem}
To establish the existence of subsonic solution to the impinging jet flow problem, we will take limit
$\mu\rightarrow\infty$ to the solution $\psi_{\ld_{\mu},\mu}$ of the
truncated variational problem $(P_{\ld_{\mu},\mu})$ and show the
limit $\psi_\ld$ is indeed a solution to the following variational
problem $(P_{\ld})$ stated as follows in the whole domain $\O$.

{\bf{The variational problem $(P_{\ld})$:}}
 For any bounded domain $D\subset\O$, find a $\psi_\ld\in K$ such that
$$J_{D}(\psi_\ld)\leq J_{D}(\psi),$$
for any $\psi\in K_0$ with $\psi=\psi_\ld$ on $\p D$, where
$J_D(\psi)=\int_D G(\g\psi,\psi,x;\ld)dxdy$ and  $$ K_0=\left\{\psi\in K\mid  \psi\leq\min\left\{\f{m_0}{a^2}x^2,m_0\right\}\ \ \text{a.e. in}\ \ \O_\mu \right\}.$$

\begin{thm}\label{ld1} If $m_0>0$ is sufficiently small, there exist a $\ld\leq\Pi_\ld-4\t\e$ and a subsonic solution $(\psi_\ld,\Gamma_\ld)$ to the compressible impinging jet flow. \end{thm}

\begin{proof}

By virtue of the uniform gradient estimate
$|\nabla\psi_{\ld_{\mu_n},\mu_n}|\leq C$ in any compact subset of
$\mathbb{R}^2$, it follows from the similar arguments in the proof of Lemma \ref{lc6} that there
exist a subsequence $\{\psi_{\ld_{\mu_n},\mu_n}\}$ and
$\{\ld_{\mu_n}\}$ with $\ld_{\mu_n}\leq\Pi_{\ld_{\mu_n}}-4\t\e$ and $k_{\ld_{\mu_n},\mu_n}(b)=1$, such that
$$\ld_{\mu_n}\rightarrow \ld\leq\Pi_\ld-4\t\e,$$and$$\psi_{\ld_{\mu_n},\mu_n}\rightarrow \psi_\ld\ \text{weakly~
in}~~H_{loc}^1(\mathbb{R}^2) ~\text{and uniformly in any compact
subset of $\mathbb{R}^2$},$$ as $n\rightarrow\infty$. Moreover, it follows from \eqref{c49} that
\be\label{d1}\ld\leq C_0m_0,\ \ \text{the constant $C_0$ depends on $N,N_0,\vartheta$, not on $m_0$}.\ee

{\bf Step 1.} $\psi_\ld$ is a subsonic solution of the free boundary
problem \eqref{b11}.

By using the similar arguments in Step 4 in the proof of Lemma \ref{lc6}, we can check that $\psi_\ld$ is a
minimizer to the variational problem $(P_\ld)$. Moreover, the monotonicity of $\psi_{\ld_{\mu_n},\mu_n}(x,y)$ in Lemma \ref{lc1} gives that $\psi_\ld(x,y)$ is monotonic increasing with respect to
$y$, which implies that
the free boundary of $\psi_\ld$ is $x$-graph. Applying the similar arguments in the proof of Lemma \ref{lc3}, there exists a
continuous function $k_{\ld}(x)$, such that the free boundary
$\Gamma_{\ld}$ of the minimizer $\psi_\ld$ can be
described as
$$\Gamma_\ld=E\cap\p\{\psi_\ld<m_0\}: y=k_\ld(x)\ \ \text{for}\ \ x\in(b,+\infty).$$

Furthermore, it follows from the similar arguments in Lemma \ref{lc7} that
$$k_\ld(x)=\lim_{n\rightarrow\infty}k_{\ld_{\mu_n},\mu_n}(x) \ \ \text{for any $x\in[b,+\infty)$}.$$
In particular, one has
$$k_\ld(b)=1,$$ which is the continuous fit condition to
the axially symmetric compressible subsonic impinging jet flow.

In
Theorem 6.1 in \cite{ACF4} and Section 3.11 in \cite{FA1}, Alt, Caffarelli and Friedman obtained that the continuous fit condition implies the smooth fit condition at the detachment point $A$, namely,
$N\cup\Gamma_{\ld}$ is $C^1$ at $A$. Since $\psi_\ld$ is a minimizer to the variational problem $(P_\ld)$, it follows from Theorem 6.3 in \cite{ACF4} that the free boundary $\Gamma_\ld$ is $C^{1,\alpha}$, and thus $\psi_{\ld}$ is $C^{1,\alpha}$-smooth up to the free boundary $\Gamma_\ld$. In view of $\ld\leq\Pi_\ld-3\t\e$, the subsonic cut-off can be removed near $\Gamma_\ld$. Then $F(t;\ld)$ is analytic with respect to $t$, near the free boundary $\Gamma_\ld$. Recalling Remark 6.4 in \cite{ACF4}, we can conclude that the free boundary $\Gamma_\ld$ is analytic.
By using the similar arguments in the proof of Theorem 9.1 in \cite{CDW}, we can conclude that
$$\text{$\f{1}{x}\f{\p\psi_\ld}{\p\nu}=\left|\f{\nabla\psi_\ld}{x}\right|=\ld$
on $\Gamma_\ld$,}$$ where $\nu$ is outer normal vector to $\Gamma_\ld$. Utilizing Lemma 6.4 in \cite{ACF5}, one has
$$\text{$\g\psi_\ld$ is uniformly continuous in a $\{\psi_\ld<m_0\}$-neighborhood of $A$},$$ and $$\text{$\left|\f{\nabla\psi_\ld}{x}\right|=\ld$
at $A$.}$$
Since $\psi_\ld$ is a minimizer to the variational problem $(P_\ld)$, applying the results (2) in Lemma \ref{lb2}, one has
$$\t Q_\ld\psi_\ld=0\ \ \text{in}\ \ \O\cap\{\psi_\ld<m_0\}.$$

Hence, the minimizer $\psi_\ld$ is a solution of the truncated free boundary problem \eqref{b11}.

{\bf Step 2.} The asymptotic behavior of $\psi_\ld$ will be obtained. For the asymptotic behavior of $\psi_\ld$ in upstream, we can use the blow-up arguments in the proof of Lemma 5 in \cite{XX2}, and obtain that
$$\g\psi_\ld(x,y)\rightarrow(-v_{in},0),~~\t\rho\rightarrow\rho_{in}\ \ \text{and}~~\nabla(\p_x\psi_\ld,\p_y\psi_\ld,\tilde{\rho})\rightarrow0,$$
uniformly in any compact subset of $(0,a)$, as $y\rightarrow+\infty$, where
$v_{in}=-\f{2m_0}{\rho_0a^2}$ and
$\rho_{in}=\t\rho\left(\f{4m_0^2}{a^4};\ld^2\right)$.

By virtue of \eqref{c16}, there exists a constant $C>0$, such that
\be\label{d0}\int_{\O\cap\{x>x_0\}}x\left|\f{\nabla\psi}{x}-\l
e\chi_{\{\psi<m_0\}}\right|^2dxdy\leq \f{C}{x_0^3},\ee for sufficiently large $x_0>b$, where the constant $C$ is independent of $x_0$.

With the aid of \eqref{d0}, by using the similar arguments in the proof of Lemma \ref{lc4}, we can show that
 \be\label{d2}x(k_{\ld}(x)-g_0(x))\rightarrow\f{m_0}{\l\cos\theta}~~ \text{as}~~x\rightarrow+\infty,\ee
 and
  \be\label{d3}\f{\nabla\psi_{\ld}(x,y)}{x}\rightarrow(-\ld\sin\th,\ld\cos\th)~~ \text{for $(x,y)\in\O_\mu\cap\{\psi_{\ld}<m_0\}$, as $x\rightarrow+\infty$.}\ee

{\bf Step 3.} In this step, we will remove the subsonic cut-off for $\t\rho_\ld$ in \eqref{b9}.

For $\ld\leq\Pi_\ld-3\t\e$, it is easy to check that $\t\rho\left(\left|\f{\g\psi_\ld}{x}\right|^2;\ld^2\right)$ is a monotonic decreasing function of
$\left|\f{\g\psi_\ld}{x}\right|\in(0,\Pi_\ld)$. Moreover,
\be\label{d4}\text{$\f{|\g\psi_\ld|}{x}$ takes the maximum at $X_0$ if and
only if $q$ takes the maximum at $X_0$,}\ee where $q=\sqrt{u^2+v^2}$ is the speed of the fluid.

By using the similar arguments in Page 114 in \cite{ACF5}, one has
$$\mathcal{Q} q^2=D_i\left(e^{\gamma q^2}a_{ij}(X;\ld)D_j
q^2\right)\geq 0\ \ \text{in the fluid region $\O_0$,}$$ for some
$\gamma>0$, where
$$a_{ij}(X;\ld)=\t\rho(|\g\varphi(X)|^2;\ld^2)\delta_{ij}+2\rho_1(|\g\varphi(X)|^2;\ld^2)D_i\varphi(X)
D_j\varphi(X)$$ and the potential function $\varphi$ satisfies that $\g\varphi=(u,v)$. In view of the maximum
principle for $q^2$ in $\O_0$, we have that $q^2$ cannot take its maximum in $\O_0$. Since the flow is assumed to be symmetric with respect to the symmetric axis $I$, thus $I$ can be regarded as the interior of the fluid field by the even extension
of $\varphi$. Thus we conclude that the speed $q$ cannot take its maximum at the symmetric axis $I$. By virtue of \eqref{d4}, $\f{|\g\psi_\ld(X)|}x$ takes its maximal value at $N\cup\Gamma_\ld\cup N_0$ or in the far field. Next, we consider
the following three cases for the maximal momentum $\f{|\g\psi_\ld(X)|}x$.

{\bf Case 1.} $\f{|\g\psi_\ld(X)|}x$ takes its maximum in the far field or on the free boundary $\Gamma_\ld$. By virtue of \eqref{d1}, it follows from the asymptotic behavior and the free boundary condition that
\be\label{d5}\sup_{X\in\bar{\O}_0}\f{|\g\psi_\ld(X)|}x=
\max\left\{\f{m_0}{a},\ld\right\}\leq C_0m_0,\ee where
$C_0$ is a constant depending only on $N, N_0$ and $\vartheta$.

{\bf Case 2.} $\f{|\g\psi_\ld(X)|}x$ takes its maximum on walls
$N_0\cap\left\{x\leq\f{a+b}{2}\right\}$ or on $N\cap\left\{x\leq\f{a+b}{2}\right\}$. By using the similar arguments in Section 3 in \cite{XX2}, we have
\be\label{d6}\sup_{X\in\bar{\O}_0}\f{|\g\psi_\ld(X)|}{x}\leq C_0m_0,\ee where $C_0$
is a constant depending only on $N,N_0$ and $\vartheta$.


{\bf Case 3.} $\f{|\g\psi_\ld(X)|}x$ takes its maximum at the nozzle wall
$N_0\cap\left\{x\geq\f{a+b}2\right\}$ or $N\cap\left\{x\geq\f{a+b}2\right\}$. Applying the similar arguments in the proof of Theorem \ref{lb3} and Lemma \ref{lc5}, we have
\be\label{d7}\sup_{X\in\bar{\O}_0}\f{|\g\psi_\ld(X)|}{x}\leq C\ld\leq C_0m_0,\ee where $C_0$
is a constant depending only on $N,N_0$ and $\vartheta$.

It follows from \eqref{d6}-\eqref{d7} that
$$\f{|\g\psi_{\ld}(X)|}{x}\leq C_0m_0\ \ \text{in $\O_0$},$$
which implies  that $C_0m_0\leq\Pi_\ld-4\t\e$ for sufficiently small $m_0>0$. Then the subsonic cut-off can be taken away
 $\rho(t;\ld^2)=\t\rho(t;\ld^2)$.

{\bf Step 4.} In this step, the positivity of horizontal velocity will be obtained, namely,  then
\be\label{d8}\f{\p\psi_\ld}{\p y}>0~~\text{in}~~\bar{\O}_0\setminus I,\ee
where $\O_0=\O\cap\{\psi_\ld<m_0\}$.

Set $\o=\p_y\psi_\ld$, which solves
$$\p_i\left(\f{\p_i\o x^2\rho-2\rho_1\p_i\psi_\ld\p_j\psi_\ld\p_j\o}{x^3\rho^2}\right)=0\ \ \text{in}\ \ \O\cap\{\psi_\ld<m_0\}.$$
Since $\o\geq0$ in $\O_0$, the strong maximum principle gives that $\o>0$ in $\O_0$.

Owing to that $\psi_\ld$ attains its maximal value $m_0$ on $N$, it follows from Hopf's lemma that
$$0<\f{\p\psi_\ld}{\p\nu}(x,g(x))=\p_y\psi_\ld\sqrt{1+(g'(x))^2}=\o\sqrt{1+(g'(x))^2}\ \ \text{on}\ \ N\setminus A,$$ where $\nu=\f{(g'(x),1)}{\sqrt{1+(g'(x))^2}}$ is the outer normal vector to $N$. Similarly, we can show that
$$\o>0\ \ \text{on}\  \ N_0\setminus(0,g_0(0)).$$
Next, we will show that \be\label{d9}\o>0\ \ \text{on}\ \
\Gamma_\ld.\ee Suppose that there exists a free boundary point
$X_0=(x_0,y_0)$, such that $\o(X_0)=0$. Without loss of generality,
we take $\nu=(0,1)$ as the outer normal vector of $\Gamma_\ld$ at
$X_0$. Since $\Gamma_\ld$ is analytic at $X_0$, thanks to Hopf's
lemma, one has \be\label{d10}\p_{xy}\psi_\ld=\f{\p\o}{\p
x}=\f{\p\o}{\p\nu}<0\ \ \text{at}\ \ X_0.\ee On another hand, it
follows from $|\g\psi_\ld|^2=\ld^2 x^2$ on $\Gamma_{\ld}$ that
$$0=\f{\p(\ld^2x^2)}{\p s}=\f{\p|\g\psi_\ld|^2}{\p s}=2\p_x\psi_\ld\p_{xy}\psi_\ld+2\p_y\psi_\ld\p_{yy}\psi_\ld=2\p_x\psi_\ld\p_{xy}\psi_{\ld},$$ where $s=(0,1)$ is the tangential vector of $\Gamma_\ld$ at $X_0$. This contradicts to \eqref{d10}.

Recalling that $|g'(b-0)|<+\infty$, one has
$$\o=\f{|\g\psi_\ld|}{\sqrt{1+(g'(x))^2}}=\f{\ld x}{\sqrt{1+(g'(x))^2}}>0\ \ \text{at}\ \ A.$$

Hence, we complete the proof of \eqref{d8}. In view of \eqref{d9}, the implicit function theorem gives that $k_\ld(x)\in C^1((b,\infty))$. The analyticity of free boundary $\Gamma_\ld$ gives that $k_\ld(x)$ is analytic in $(b,+\infty)$.

\end{proof}

\subsection{Uniqueness of the compressible subsonic jet flow}

In this subsection, we will consider the uniqueness of subsonic solution of
the compressible jet flow problem for any given incoming mass flux $m_0$.

\begin{thm}\label{ld2}
For any given $m_0>0$, suppose that $(\psi_{\lambda_1},\t\Gamma_{\ld_1})$ and
$(\t\psi_{\lambda_2},\Gamma_{\ld_2})$ be two subsonic solutions to compressible impinging jet flow problem, respectively.
Then $\ld_1=\ld_2$ and $\psi_{\lambda_1}=\t\psi_{\lambda_2}$.
\end{thm}

\begin{proof}

We will divide the proof into two steps.

 {\bf Step 1.} We will show
that $\ld_1=\ld_2$. If not, without loss of generality, one may
assume that $\ld_1>\ld_2$. In view of the asymptotic behaviors of
$\psi_{\ld_1}$ and $\t\psi_{\ld_2}$ in downstream (see Step 2 in the
proof of Theorem \ref{ld1}), one has  $$k_{\ld_1}(x)-g_0(x)\sim
\f{m_0}{\ld_1x\cos\th}\ \ \text{and}\ \ \t k_{\ld_2}(x)-g_0(x)\sim
\f{m_0}{\ld_2x\cos\th}$$ for sufficiently large $x>0$, which implies
\be\label{d12}k_{\ld_1}(x)<\t k_{\ld_2}(x)\ \ \text{for sufficiently
large $x>0$.}\ee   Denote $\phi_1=\f{m_0-\psi_{\ld_1}}{\Pi_{\ld_1}}$
and $\t\phi_{2}=\f{m_0-\t\psi_{\ld_2}}{\Pi_{\ld_2}}$, where
$\Pi_{\ld_1}=\rho_{\ld_1,cr}q_{\ld_1,cr}$ and
$\Pi_{\ld_2}=\rho_{\ld_2,cr}q_{\ld_2,cr}$. It is easy to check that
\be\label{d13}\f{|\g\phi_1|^2}{2x^2\rho^2}+\f{\rho^{\gamma-1}}{\gamma-1}=\f{\gamma+1}{2(\gamma-1)}\
\ \text{in}\ \ \O\cap\{\phi_1>0\},\ee and
\be\label{d14}\f{|\g\t\phi_2|^2}{2x^2\rho^2}+\f{\rho^{\gamma-1}}{\gamma-1}=\f{\gamma+1}{2(\gamma-1)}\
\ \text{in}\ \ \O\cap\{\t\phi_2>0\},\ee where
$\rho(t)=\f{\rho\left(\f{t}{\Pi^2_{\ld_1}};\ld^2_1\right)}{\rho_{\ld_1,cr}}=\f{\rho\left(\f{t}{\Pi^2_{\ld_2}};\ld^2_2\right)}{\rho_{\ld_2,cr}}$.
Moreover, $\rho=\rho\left(\left|\f{\g\phi}{x}\right|^2\right)$ is
monotone decreasing with respect to $\left|\f{\g\phi}x\right|\in [
0,1)$. Thus $\phi_1$ and $\t\phi_2$ satisfy the following
quasilinear elliptic equations,
$$Q\phi_1=\text{div}\left(\f{\g\phi_1}{x\rho(|\f{\g\phi_1}{x}|^2)}\right)=0\ \
\text{in the fluid $\O\cap\{\phi_1>0\}$},$$ and
$$Q\t\phi_2=\text{div}\left(\f{\g\t\phi_2}{x\rho(|\f{\g\t\phi_2}{x}|^2)}\right)=0\ \
\text{in the fluid $\O\cap\{\t\phi_{2}>0\}$},$$ respectively. Denote
$\t\phi_2^{\e}(x,y)=\t\phi_{2}(x,y-\e)$ for $\e\geq0$, let
$\t\Gamma_{\ld_2}^{\e}: y=\t k_{\ld_2}(x)+\e$ to be the free boundary of $\t\phi_2^{\e}$. Choose $\e_0\geq0$ to be the smallest one, such that
$$\t\phi_2^{\e_0}(X)\geq\phi_{1}(X)\ \text{in $\O$, and}\ \t\phi_2^{\e_0}(X_0)=\phi_{1}(X_0)\ \text{for some $X_0\in\overline{\O\cap\{\phi_1>0\}}$}.$$
 Next, we consider
the following two cases for $\e_0$.

{\bf Case 1.} $\e_0=0$, then we can choose $X_0=A$. The strong maximum principle gives that
$$\t\phi_{2}(X)>\phi_{1}(X)\ \text{and} \ Q\phi_{1}=Q\t\phi_{2}=0\ \text{in $\O\cap\{\phi_1>0\}$}.$$

Since $\Gamma_{\ld_1}\cup N$ and $\t\Gamma_{\ld_2}\cup N$ are $C^1$
at $A$, one has
\be\label{d15}\f{\ld_1}{\Pi_{\ld_1}}=\f 1b\f{\p\phi_{1}}{\p\nu}\leq\f1 b\f{\p\t\phi_{2}}{\p\nu}=\f{\ld_2}{\Pi_{\ld_2}}\
\ \ \text{at $A$,}\ee where $\nu$ is the inner normal vector of
$\Gamma_{\ld_1}$ and $\t\Gamma_{\ld_2}$ at $A$. After a direct computation, one has
$$\f{d}{d\ld}\left(\f{\Pi_\ld}{\ld}\right)=\f{\Pi_\ld(\ld^2-\ld_{cr}^2)}{(\gamma-1)\ld\rho^2_{0}\mathcal{B}(\ld^2)}<0
\ \ \text{ for $\ld<\ld_{cr}$,}$$ which implies that
\be\label{d16}\f{\Pi_{\ld_1}}{\ld_1}<\f{\Pi_{\ld_2}}{\ld_2}\ \
\text{for any $0<\ld_2<\ld_1<\ld_{cr}$}.\ee This contradicts to \eqref{d15}.

{\bf Case 2.} $\e_0>0$. It follows from strong maximum principle that $X_0\notin
\O\cap\{\phi_{1}>0\}$.
In fact, if there exists a point
$X_0\in\O\cap\{\phi_1>0\}$,
the continuity of $\phi_{1}(X)$ and $\t\phi_2^{\e_0}(X)$ give that
there exists a disc $B_r(X_0)\subset \O\cap\{\phi_1>0\}$ with $r>0$, such that
$$Q\phi_{1}=Q\t\phi_2^{\e_0}=0\ \text{in $B_r(X_0)$, and}\
\t\phi_2^{\e_0}(X)\geq\phi_{1}(X)\ \text{in $B_r(X_0)$}.$$ Since $\phi_{1}(X_0)=\t\phi_2^{\e_0}(X_0)$, the strong
maximum principle implies that $\phi_{1}(X)\equiv\t\phi_2^{\e_0}(X)$
in $B_r(X_0)$. By using the strong maximum principle again, one has
$$\phi_{1}(X)\equiv\t\phi_2^{\e_0}(X)\ \ \text{ in
$\O\cap\{\phi_{1}>0\}$},$$ which leads a contradiction.

In view of $\e_0>0$, it follows from \eqref{c16} that $|X_0|<+\infty$, and thus $X_0\in
\t\Gamma_{\ld_2}^{\e_0}\cap\Gamma_{\ld_1}$. Moreover,
$$\t\phi_2^{\e_0}(X)>\phi_{1}(X)\ \text{and} \
Q\phi_{1}=Q\t\phi_2^{\e_0}=0\ \text{in $\O\cap\{\phi_1>0\}$}.$$ Since
$\Gamma_{\ld_1}$ and $\t\Gamma_{\ld_2}^{\e_0}$ are analytic at $X_0$, it follows from
Hopf's lemma that
$$\f{\ld_2}{\Pi_{\ld_2}}=\f{1}{x}\f{\p\t\phi_2^{\e_0}}{\p\nu}>\f1{x}\f{\p\phi_{1}}{\p\nu}=\f{\ld_1}{\Pi_{\ld_1}}\
\ \ \text{at $X_0$,}$$ where $\nu$
is the inner normal vector to $\t\Gamma_{\ld_2}^{\e_0}$ and $\Gamma_{\ld_1}$
at $X_0$, which leads a contradiction to the assumption
$\ld_1>\ld_2$, due to \eqref{d16}.

Hence, we obtain that $\ld_1=\ld_2$, and denote $\ld=\ld_1=\ld_2$ in
the following.

{\bf Step 2.} $\psi_\ld=\t\psi_\ld$. Suppose that $\psi_\ld\neq\t\psi_\ld$,
without loss of the generality, one may assume that there exists some
$x_0\in(0,+\infty)$, such that \be\label{d17}k_{\ld}(x_0)>\t
k_{\ld}(x_0)\ \ \text{for some $x_0>0$.}\ee Consider a function
$\psi_\ld^{\e}(x,y)=\psi_{\ld}(x,y-\e)$ for $\e\geq0$ and $\Gamma_\ld^{\e}: y=k_{\ld}(x)+\e$
is the free boundary of $\psi_\ld^{\e}$, choosing the smallest
$\e_0\geq0$ such that
$$\psi_\ld^{\e_0}(X)\leq\t\psi_{\ld}(X)\ \text{in $\O$, and}\ \psi_\ld^{\e_0}(X_0)=\t\psi_{\ld}(X_0)\ \text{for some $X_0\in\bar\O$}.$$
It follows from \eqref{d17} that $\e_0>0$, which together with the
strong maximum principle and the asymptotic behavior imply that
$X_0\notin\O\cap\{\t\psi_{\ld}<m_0\}$ and $X_0\in
\Gamma_\ld^{\e_0}\cap\t\Gamma_{\ld}$ with $|X_0|<+\infty$. Then we have
$$\psi_\ld^{\e_0}(X)<\t\psi_{\ld}(X)\ \text{and} \
Q_\ld\psi_\ld^{\e_0}=Q_\ld\t\psi_{\ld}=0\ \text{in
$\O\cap\{\t\psi_{\ld}<m_0\}$}.$$

 Thanks to the Hopf's lemma, one has
$$\ld=\f1{x}\f{\p\psi_\ld^{\e_0}}{\p\nu}>\f1{x}\f{\p\t\psi_{\ld}}{\p\nu}=\ld\
\ \ \text{at $X_0$,}$$ where $\nu$ is the outer normal vector of $\Gamma_\ld^{\e_0}\cap\t\Gamma_{\ld}$ at $X_0$, which leads
a contradiction.

\end{proof}

\section{The existence of the critical mass flux}

For any sufficiently small
$m_0>0$, we have shown that there exist a unique $\ld\leq\Pi_\ld-4\t\e$ and a unique solution
$(u,v,\rho,\Gamma_\ld)$ to the free boundary problem in previous section. One key point is that the smallness of $m_0$ guarantee
the global subsonicity of the compressible jet flow. In this
section, we will increase $m_0$ as large as possible, and
obtain the critical upper bound of the incoming mass flux $m_0$.

Let $\{\e_n\}_{n=1}^{\infty}$ be a strictly decreasing sequence with
$\e_n\downarrow 0$. Denote $\psi^n_{\ld,m}(x,y)$ as the solution
of the following free boundary value problem
\be\label{e1}\left\{\begin{array}{ll}&\g\cdot\left(\f{\nabla
\psi}{x\rho^{n}(|\f{\nabla\psi}{x}|^2;\ld^2)}\right)=0\ \ \text{in}\ \O\cap\{\psi<m\},\\
&\psi=0\ \text{on}\ T,\ \psi=m\ \text{on}\ N\cup
\Gamma^n_{\ld,m},\\
&\f1x\f{\p\psi}{\p \nu}=\ld\ \ \text{on}\ \
\Gamma^n_{\ld,m},\end{array}\right.\ee for any sufficiently small
$m>0$ and the free boundary $\Gamma_{\ld,m}^n: y=k^n_{\ld,m}(x)$
satisfies the continuous fit condition $k^n_{\ld,m}(b)=1$, where
$\nu$ is the outer normal vector. Here, $\rho^{n}(t;\ld^2)$ is a
smooth function satisfying
$$\rho^{n}(t;\ld^2)=\left\{\begin{array}{ll} \rho(t;\ld^2)\ \ \
&\text{if}\ \
0\leq t\leq(\Pi_\ld-2\e_n)^2,\\
\rho\left((\Pi_\ld-\e_n)^2;\ld^2\right)\ \ \ &\text{if}\ \
t\geq(\Pi_\ld-\e_n)^2,
\end{array}\right.
$$ and
$\rho^{n}(t;\ld)-2\rho_1^{n}(t;\ld)t<\gamma_n<+\infty$ with some constant $\gamma_n>0$. 

First,  we define a set
$$\mathcal{K}_n(m)=\left\{\psi_{\ld,m}^n\mid \text{$\psi_{\ld,m}^n$ is a subsonic solution to free boundary problem \eqref{e1}}\right\},$$ for any small $\e_n>0$ and
 $m>0$. For any small $m>0$ and $\e_n>0$, it follows from Theorem \eqref{ld1} that
there exist a $\ld=\ld(m)$ and a unique subsonic solution
$\psi^n_{\ld(m),m}$ to the free boundary problem with
$$\text{$\ld(m)\leq\Pi_{\ld(m)}-4\e_n$ and $\sup_{X\in\O\cap\{0<\psi^n_{\ld(m),m}<m\}}\left|\f{\g\psi^n_{\ld(m),m}}{x}\right|^2-\Pi_{\ld(m)}\leq -4\e_n$}.$$  Thus $m\in\mathcal{K}_n(m)$ for small $m>0$ and the set $\mathcal{K}_n(m)$ is not empty. The uniqueness of subsonic solution and $\ld$ are established in Section 4. Denote $\ld=\ld(m)$ and $\psi^n_{\ld(m),m}$ as the unique subsonic solution to the free boundary problem \eqref{e1} with
 $$\text{$\ld(m)\leq\Pi_{\ld(m)}-4\e_n$ and $\sup_{X\in\O\cap\{0<\psi^n_{\ld(m),m}<m\}}\left|\f{\g\psi_{\ld(m),m}^n}{x}\right|-\Pi_{\ld(m)}\leq -4\e_n$}.$$

Denote \be\label{e2}T_n(m)=\inf_{\psi_{\ld,m}^n\in
\mathcal{K}_n(m)}\left(\sup_{(x,y)\in\O^n_{\ld,m}}\left|\f{\g\psi_{\ld,m}^n}{x}\right|-\Pi_\ld\right),\ee
where $\O^n_{\ld,m}=\O\cap\{0<\psi^n_{\ld,m}<m\}$.

Define a set
$$\ba{rl}\Sigma_n=\{m\mid &\text{for any $\tau\in(0,m)$, there exists a subsonic solution
$\psi_{\ld,m}^n$ to the
}\\
&\text{free boundary problem with
$T_n(\tau)\leq-4\e_n$}\}.\ea$$

Along the above arguments, we have $m\in\Sigma_n$ for small $m>0$, which implies that the set
$\Sigma_n$ is non-empty. Obviously,
$\Sigma_n\subset\Sigma_{n+1}$.

Set \be\label{e3}m_n=\sup_{m\in \Sigma_n} m.\ee The definition of $m_n$ implies that $m_n$ is
monotone increasing with respect to $n$. For $m>0$, it is easy to check that
\be\label{e4}m=\psi_{\ld,m}^n(B)-\psi_{\ld,m}^n(b,g_0(b))\leq
\sup_{X\in\O^n_{\ld,m}}|\g\psi_{\ld,m}^n(X)||1-g_0(b)|\leq\ld_{cr}b|1-g_0(b)|.\ee

With the aid of \eqref{e4}, we can define \be\label{e5}m_{cr}=\lim_{n\rightarrow+\infty}
m_n.\ee

\begin{lem}\label{le1}
For any $m\in(0,m_n]$, $T_n(m)$ is left-continuous with respect to $m$, namely, $T_n(m)=\lim_{\tau\rightarrow m^-}
T_n(\tau)$.

\end{lem}

\begin{proof}For any $m\in(0,m_n]$, there exists a sequence $\{\tau_k\}$ with $\tau_k\uparrow m$. The definition of $\Sigma_n$
gives that there exists a subsonic solution $\psi^n_{\ld,\tau_k}$ to
the free boundary problem, which satisfies that
$$T_n(\tau_k)=\inf_{\psi_{\ld,\tau_k}^n\in
\mathcal{K}_n(\tau_k)}\left(\sup_{X\in\O^n_{\ld,\tau_k}}\left|\f{\g\psi_{\ld,\tau_k}^n(X)}{x}\right|-\Pi_\ld\right)\leq
-4\e_n.$$ By using the uniqueness result in Section 4, we have that
$\ld=\ld(\tau_k)$ and solution
$\psi^n_{\ld,\tau_k}=\psi^n_{\ld(\tau_k),\tau_k}$ is the unique
subsonic solution to the free boundary problem \eqref{e1}. Then one
has
\be\label{e6}T_n(\tau_k)=\sup_{X\in\O^n_{\ld(\tau_k),\tau_k}}\left|\f{\g\psi^n_{\ld(\tau_k),\tau_k}(X)}{x}\right|-\Pi_{\ld(\tau_k)}\leq-4\e_n.\ee
By using the similar arguments in the proof of Lemma \ref{lb5}, we
can take a subsequence $\{\tau_k\}$, such that
$$\ld(\tau_k)\rightarrow \ld_0\leq\Pi_{\ld_0}-4\e_n,$$ and
$$\psi^n_{\ld(\tau_k),\tau_k}\rightarrow\psi^n_{\ld_0,m} \
\text{weakly in $H_{loc}^1(\O)$ and uniformly in any compact subset
of $\mathbb{R}^2$},$$ as $k\rightarrow+\infty$. Moreover, the inequality \eqref{e6} gives that
$$\lim_{\tau_k\rightarrow m}T_n(\tau_k)=\sup_{X\in\O^n_{\ld_0,m}}\left|\f{\g\psi^n_{\ld_0,m}(X)}{x}\right|-\Pi_{\ld_0}\leq-4\e_n.$$
Thus $\psi^n_{\ld_0,m}$ is a subsonic solution to the free boundary problem \eqref{e1}. Applying the uniqueness results in Section 4, we conclude that $\ld_0=\ld(m)$ and
$\psi_{\ld_0,m}^n=\psi_{\ld(m),m}^n\in\mathcal{K}_n(m)$. It follows from
the definition of $T_n(m)$ in \eqref{e2} and the uniqueness result in Section 4 that
$$T_n(m)=\lim_{\tau_k\rightarrow m^-}T_n(\tau_k).$$

\end{proof}

\begin{lem}\label{le2}
There exists a critical mass flux $m_{cr}>0$, such that for any
$m\in(0,m_{cr})$, there exist a unique $\ld=\ld(m)<\ld_{cr}$ and a
unique subsonic solution $\psi_{\ld,m}$ to the free boundary problem \eqref{e1}, such that
\be\label{e7}T(m)=\sup_{X\in\O_{\ld,m}}\left|\f{\g\psi_{\ld,m}(X)}{x}\right|-\Pi_\ld<0,\ee
where $\O_{\ld,m}=\O\cap\{0<\psi_{\ld,m}<m\}$. And $m_{cr}$ is the
upper critical mass flux for the existence of subsonic solution in the
following sense: either \be\label{e8}T(m)\rightarrow 0 \ \
\text{as}\ \ m\rightarrow m_{cr},\ee or there is no $\sigma>0$, such
that for any $m\in(m_{cr},m_{cr}+\sigma)$, there exist a
$\ld<\ld_{cr}$ and a subsonic solution $\psi_{\ld,m}$ to the free boundary problem \eqref{e1}, and
\be\label{e9}\sup_{m\in(m_{cr},m_{cr}+\sigma)}T(m)<0.\ee

\end{lem}

\begin{proof} For any $m\in(0,m_{cr})$, the definition of $m_{cr}$ in \eqref{e5} gives that there exists a $N$, such
that $m<m_n$ for any $n>N$. Therefore, it follows from the definition of $m_n$
that
 we have
$$T_n(m)=\inf_{\psi^n_{\ld,m}\in
\mathcal{K}_n(m)}\left(\sup_{X\in\O_{\ld,m}}\left|\f{\g\psi^n_{\ld,m}(X)}{x}\right|-\Pi_\ld\right)\leq
-4\e_n\ \ \text{for}\ \ n>N.$$ By virtue of Theorem \ref{ld1}, we can conclude that there exist a unique $\ld(m)\leq
\Pi_{\ld(m)}-4\e_n$ and a unique subsonic solution $\psi^n_{\ld(m),m}$ to the
free boundary problem \eqref{e1}, such that
$$T_n(m)=\sup_{X\in\O_{\ld(m),m}}\left|\f{\g\psi^n_{\ld(m),m}(X)}{x}\right|-\Pi_{\ld(m)}\leq
-4\e_n.$$ Taking $\psi_{\ld,m}=\psi^n_{\ld(m),m}$, then $\psi_{\ld,m}$
is the unique subsonic solution to the compressible impinging
jet flow problem \eqref{e1} and  $T(m)=T_n(m)\leq -4\e_n<0$.

If $\sup_{m\in(0,m_{cr})}T(m)<0$, there exists a large $N$, such
that \be\label{e11}\sup_{m\in(0,m_{cr})}T(m)<-4\e_n\ee for any $n>N$. It is easy to
check that $m_{cr}\in\Sigma_n$, and thus $m_{cr}\leq m_n$ for any
$n>N$.

It follows from Lemma \ref{le1} that $T_n(m)$ is left-continuous for
$m\in(0,m_n]$, and thus \be\label{e12}T(m_{cr})=T_n(m_{cr})\leq -4\e_n.\ee Suppose that there
exists a $\sigma>0$, such that for any $m\in(m_{cr},m_{cr}+\sigma)$,
there exists a subsonic solution $\psi_{\ld,m}$ with $\ld<\ld_{cr}$ to the
free boundary problem \eqref{e1} and
\be\label{e10}\sup_{m\in(m_{cr},m_{cr}+\sigma)}T(m)=\sup_{m\in(m_{cr},m_{cr}+\sigma)}
\left(\sup_{X\in\O_{\ld,m}}\left|\f{\g\psi_{\ld,m}(X)}{x}\right|-\Pi_\ld\right)<0.\ee
In view of \eqref{e10}, there exists a large $k>0$, such that
\be\label{e13}\sup_{m\in(m_{cr},m_{cr}+\sigma)}T(m)=
\sup_{m\in(m_{cr},m_{cr}+\sigma)}\left(\sup_{X\in\O_{\ld,m}}\left|\f{\g\psi_{\ld,m}(X)}{x}\right|-\Pi_\ld\right)\leq-4\e_{n+k}.\ee
By virtue of \eqref{e11}, \eqref{e12} and \eqref{e13}, one has
$$T_{n+k}(m_{cr}+\sigma)=T(m_{cr}+\sigma)\leq\sup_{m\in(0,m_{cr}+\sigma)} T(m)\leq 4\e_{n+k},$$ for any $n>N$,
which implies that $m_{cr}+\sigma\in\Sigma_{n+k}$. The definition of
$m_{n+k}$ in \eqref{e3} gives that $$m_{n+k}\geq
m_{cr}+\sigma>m_{cr}.$$ This leads a contradiction to the definition
of $m_{cr}$ in \eqref{e5}.

\end{proof}

\section{The proof of the main results}

Based on the previous sections, we will complete the proof of Theorem \ref{the1} and Theorem \ref{the2} in this section.

{\bf Proof of Theorem \ref{the1}.} For any given atmosphere pressure $p_{atm}>0$, it follows from  Lemma \ref{le2} that there exists a critical mass flux $M_{cr}>0$, such that for any $M_0\in(0,M_{cr})$, there exist a unique $\ld=\ld(m_0)<\ld_{cr}$ and a
unique subsonic solution $(\psi_{\ld,m_0},\Gamma_{\ld,m_0})$ to the axially symmetric compressible impinging flow, where $$M_{cr}=2\pi m_{cr}\ \ \text{and}\ \ M_0=2\pi m_0.$$ In view of the proof of Theorem \ref{ld1}, we conclude that $\psi_{\ld,m_0}$ and the free boundary $\Gamma_{\ld,m_0}:y=k_{\ld,m_0}(x)$ satisfy that
$$\psi_{\ld,m_0}\in C^{2,\alpha}(\O_0)\cap C^1(\bar\O_0),\ k_{\ld,m_0}(x)\in C^1([b,+\infty)),$$ and $$k_{\ld,m_0}(b+0)=1,~
~~k_{\ld,m_0}^{'}(b+0)=g'(b-0),\ k'_{\ld,m_0}(x)\rightarrow \tan\theta,~~~
k_{\ld,m_0}(x)-g_0(x)\rightarrow0$$ as $x\rightarrow+\infty$.
Moreover, $\sup_{(x,y)\in\bar{\O}_0}\f{|\g\psi_{\ld,m_0}|}{x\Pi_\ld}<1$ and $p=p_{atm}$ on $\Gamma_{\ld,m_0}$.

By virtue of the Bernoulli's law \eqref{a12}, one has $$\f{M^2_0}{2\pi a^4\rho^2_{in}}+\f{\mathcal{A}\gamma}{\gamma-1}\rho_{in}^{\gamma-1}
=\f{\ld^2}{2\rho_0^2}+\f{\mathcal{A}\gamma}{\gamma-1}\rho_0^{\gamma-1},\ \ \rho_0=\left(\f{p_{atm}}{\mathcal{A}}\right)^{\f1\gamma},$$
 which implies that the incoming pressure $p_{in}=A\rho^\gamma_{in}$ is determined uniquely by $\ld$. Moreover, the subsonicity of solution $\psi_{\ld,m_0}$ gives that $p_{in}\in(p_1,p_2)$, where $p_1$ and $p_2$ satisfy that $p_1<p_2$ and
$$\f{M^2_0}{2\pi^2a^4\left(\f{p_1}{\mathcal{A}}\right)^{\f2\gamma}}+\f{\mathcal{A}\gamma}{\gamma-1}\left(\f{p_1}{\mathcal{A}}\right)^{\f{\gamma-1}\gamma}
=\f{\mathcal{A}\gamma}{\gamma-1}\left(\f{p_{atm}}{\mathcal{A}}\right)^{\f{\gamma-1}{\gamma}}$$ and
$$\f{M^2_0}{2\pi^2a^4\left(\f{p_2}{\mathcal{A}}\right)^{\f2\gamma}}+\f{\mathcal{A}\gamma}{\gamma-1}\left(\f{p_2}{\mathcal{A}}\right)^{\f{\gamma-1}\gamma}
=\f{\mathcal{A}\gamma(\gamma+1)}{2(\gamma-1)}\left(\f{p_{atm}}{\mathcal{A}}\right)^{\f{\gamma-1}{\gamma}}.$$

Denote
 $$u=\f{1}{x\rho(\f{|\g\psi_{\ld,m_0}|^2}{|x|^2};\ld)}\f{\p\psi_{\ld,m_0}}{\p y}, v=-\f{1}{x\rho(\f{|\g\psi_{\ld,m_0}|^2}{|x|^2};\ld)}\f{\p\psi_{\ld,m_0}}{\p x},\ \ \Gamma=\O\cap\{\psi_{\ld,m_0}<m_0\}$$ and the density $\rho$ is determined uniquely by $$\f{|\g\psi_{\ld,m_0}|^2}{2x^2\rho^2}+\f{\mathcal{A}\gamma}{\gamma-1}\rho^{\gamma-1}
=\f{\ld^2}{2\rho_0^2}+\f{\mathcal{A}\gamma}{\gamma-1}\rho_0^{\gamma-1}.$$
Thus, $(u,v,\rho,\Gamma)$ satisfies the conditions (1)-(3) in Definition \ref{def1}, and $(u,v,\rho,\Gamma)$ is the unique subsonic solution to FBP.

The statements (1)-(3) of Theorem \ref{the1} follows from the Step 2 and Step 4 in the proof of Theorem \ref{ld1} directly. The final statement (4) of Theorem \ref{the1} is proved in Section 5.

Hence, we complete the proof of Theorem \ref{the1}.

{\bf Proof of Theorem \ref{the2}.} By virtue of Theorem \ref{the1}, the subsonic solution $(\psi_{\ld,m_0},\Gamma_{\ld,m_0})$ established in Theorem \ref{the1} satisfies that
\be\label{lll2}\text{$\g\psi_{\ld,m_0}$ is uniformly continuous in a $\{\psi_{\ld,m_0}<m_0\}$-neighborhood of $A$,}\ee and \be\label{lll3}\text{$\Gamma_{\ld,m_0}\cup N$ is $C^1$ at $A$.}\ee
Under the assumption that $N$ is $C^{1,\alpha}$ near $A$, with the aid of \eqref{lll2} and \eqref{lll3}, the proof of Theorem \ref{the2} follows from Theorem 1.1 in \cite{ACF6} directly.

\subsection*{Acknowledgments}
The authors would like to acknowledge the two anonymous reviewers
for improving this paper with their comments. Cheng  was supported
in part by NSFC grant 12001387 and the Fundamental Research Funds
for the Central Universities No. YJ202046. Du was supported by NSFC
grant 11971331. Zhang was supported by NSFC grant 12001071.

\end{document}